\documentclass[11pt, reqno]{amsart}
\usepackage{hyperref}
\usepackage[numbers]{natbib}
\usepackage{graphicx}
\usepackage{enumerate}
\usepackage{tikz-cd}
\usepackage{xcolor}

\usepackage[ansinew]{inputenc} 

\usepackage[scr=rsfso,cal=zapfc,frak=euler,bb=ams]{mathalfa}

\usepackage{amssymb}
\usepackage{bbm}

\usepackage{mathtools}
\usepackage{amsthm}
\usepackage{cases}
\usepackage{MnSymbol}

\usepackage{fullpage}
\usepackage{pifont} 
\usepackage{manfnt} 

\renewcommand{\Re}{\operatorname{Re}}
\renewcommand{\Im}{\operatorname{Im}}

\makeatletter
\def\@tocline#1#2#3#4#5#6#7{\relax
  \ifnum #1>\c@tocdepth 
  \else
    \par \addpenalty\@secpenalty\addvspace{#2}%
    \begingroup \hyphenpenalty\@M
    \@ifempty{#4}{%
      \@tempdima\csname r@tocindent\number#1\endcsname\relax
    }{%
      \@tempdima#4\relax
    }%
    \parindent\z@ \leftskip#3\relax \advance\leftskip\@tempdima\relax
    \rightskip\@pnumwidth plus4em \parfillskip-\@pnumwidth
    #5\leavevmode\hskip-\@tempdima
      \ifcase #1
       \or\or \hskip 1em \or \hskip 2em \else \hskip 3em \fi%
      #6\nobreak\relax
    \dotfill\hbox to\@pnumwidth{\@tocpagenum{#7}}\par
    \nobreak
    \endgroup
  \fi}
\makeatother

\newcommand{\bC}{\mathbb{C}}

\newcommand{\cD}{\mathscr{D}}

\newcommand{\E}{\mathbb E }
\newcommand{\R}{\mathbb{R}}
\newcommand{\N}{\mathbb{N}}
\newcommand{\C}{\mathbb{C}}
\newcommand{\Z}{\mathbb{Z}}

\renewcommand{\P}{\mathbb{P}}

\newcommand{\1}{\mathbbm{1}}

\newcommand{\Var}{\mathop{\mathrm{Var}}\nolimits}

\newcommand{\supp}{\mathop{\mathrm{supp}}\nolimits}

\newcommand{\eee}{{\rm e}}
\newcommand{\ii}{{\rm{i}}}

\newcommand{\toweak}{\overset{w}{\underset{n\to\infty}\longrightarrow}}

\newcommand{\ton}{\overset{}{\underset{n\to\infty}\longrightarrow}}

\newcommand{\unif}{{\rm Unif}}

\newcommand{\dd}{{\rm d}}

\newcommand{\stirling}[2]{\genfrac{[}{]}{0pt}{}{#1}{#2}}

\newcommand{\eps}{\varepsilon}

\theoremstyle{plain}
\newtheorem{theorem}{Theorem}[section]
\newtheorem{proposition}[theorem]{Proposition}

\newtheorem{lemma}[theorem]{Lemma}
\newtheorem{corollary}[theorem]{Corollary}

\theoremstyle{definition}
\newtheorem{definition}[theorem]{Definition}

\newtheorem{example}[theorem]{Example}

\theoremstyle{remark}
\newtheorem{remark}[theorem]{Remark}

\makeatletter
\@namedef{subjclassname@2020}{\textup{2020} Mathematics Subject Classification}
\makeatother

\begin{document}

\title[Zeros and exponential profiles of polynomials I]{Zeros and exponential profiles of polynomials I:\\ Limit distributions, finite free convolutions\\ and repeated differentiation}

\begin{abstract}
Given a sequence of polynomials $(P_n)_{n \in \mathbb{N}}$ with only nonpositive zeros, the aim of this article is to present a user-friendly approach for determining the limiting zero distribution of $P_n$ as $\deg P_n \to \infty$. The method is based on establishing an equivalence between the existence of a limiting empirical zero distribution $\mu$ and the existence of an exponential profile $g$ associated with the coefficients of the polynomials $(P_n)_{n \in \mathbb{N}}$. The exponential profile $g$, which can be roughly described by $[z^k]P_n(z) \approx \exp\big(n g(k/n)\big)$, offers a direct route to computing the Cauchy transform $G$ of $\mu$: the functions $t \mapsto tG(t)$ and $\alpha \mapsto \exp\big(-g'(\alpha)\big)$ are mutual inverses. This relationship, in various forms, has previously appeared in the literature, most notably in the paper [Van Assche, Fano and Ortolani, SIAM J. Math. Anal., 1987].

As a first contribution, we present a self-contained probabilistic proof of this equivalence by representing the polynomials as generating functions of sums of independent Bernoulli random variables. This probabilistic framework naturally lends itself to tools from large deviation theory, such as the exponential change of measure. The resulting theorems generalize and unify a range of previously known results, which were traditionally established through analytic or combinatorial methods.

Secondly, using the profile-based approach, we investigate how the exponential profile and the limiting zero distribution behave under certain operations on polynomials, including finite free convolutions, Hadamard products, and repeated differentiation. In particular, our approach yields new proofs of the convergence results `$\boxplus_n \to \boxplus$' and `$\boxtimes_n \to \boxtimes$', extending them to cases where the  distributions are not necessarily compactly supported.

This paper is the first in a series of three. In Part II, we shall explore a wide range of examples, covering many classical families of polynomials, while Part III will focus on random polynomials and establish functional limit theorems for their random profiles.
\end{abstract}

\author{Jonas Jalowy}
\address{Jonas Jalowy: Paderborn University, Institute of Mathematics, Warburger Str. 100, 33098 Paderborn, Germany}
\email{jjalowy@math.upb.de}

\author{Zakhar Kabluchko}

\address{Zakhar Kabluchko: Institute of Mathematical Stochastics, Department of Mathematics and Computer Science, University of M\"{u}nster, Orl\'{e}ans-Ring 10, D-48149 M\"{u}nster, Germany}
\email{zakhar.kabluchko@uni-muenster.de}

\author{Alexander Marynych}
\address{Alexander Marynych: School of Mathematical Sciences, Queen Mary University of London, Mile End Road, London E14NS, United Kingdom; Faculty of Computer Science and Cybernetics, Taras Shevchenko National University of Kyiv, Kyiv 01601, Ukraine}
\email{o.marynych@qmul.ac.uk, marynych@knu.ua}

\maketitle

\tableofcontents

\section{Introduction}

Polynomials are fundamental objects in mathematics, with wide-ranging applications in both theoretical and applied sciences. A central challenge in their study is understanding the relationship between a polynomial's coefficients and its zeros. The forward direction, deriving symmetric functions of the roots from the coefficients, is straightforward and classically captured by Vieta's formulas. In contrast, the inverse problem of inferring properties of the zeros based on the coefficients, has driven mathematical inquiry for centuries, leading to the development of entire fields such as Galois theory and numerical analysis. Even for polynomials of small degree, computing exact zeros is often a complex and computationally demanding task.

Rather than focusing on the exact zeros of individual polynomials, this paper considers sequences of polynomials of increasing degree. While studying individual zeros can be difficult, analyzing their collective statistical behavior through empirical distributions offers a more tractable and insightful approach. The empirical distribution captures the global
behavior of the zeros with negligible influence of individual zeros.

Throughout this paper, we work with polynomials over the field of real numbers and, unless stated otherwise, the term ``polynomial'' refers to a polynomial with real coefficients. We exclude the identically zero polynomial from consideration.

\begin{definition}\label{def:empirical_distr}
The \textit{empirical distribution of zeros} of a polynomial $P\not\equiv 0$ of degree at most $n$ is the probability measure $\lsem P \rsem_n$ on $\overline{\C}:=\C\cup\{\infty\}$ which assigns equal weight $1/n$ to each zero (counting multiplicities) and places the remaining mass at $\infty$. Thus, letting $\delta_x$ denote the Dirac measure at $x\in\overline{\mathbb{C}}$, we define
\begin{equation}\label{eq:empirical_distr_zeros_def}
\lsem P \rsem_n := \frac 1n \sum_{\substack{z\in \C\,:\, P(z) = 0}} \delta_z+ \frac{n-\deg P}{n}\delta_{\infty}.
\end{equation}
\end{definition}
As mentioned earlier, we shall work with a sequence of polynomials $(P_n)_{n\in\N}$ rather than with a single polynomial $P$, and the variable $n$ in Definition~\ref{def:empirical_distr} should be understood as the variable indexing the sequence $(P_n)_{n \in \N}$. In many cases, $n$ coincides with the degree of $P_n$, but in general, we shall assume only $\deg P_n \leq n$. The compactification of the support of $\lsem P_n\rsem_n$, as described above, proves useful even when \(\deg P_n = n\), as it enables us to capture the asymptotic behavior of zeros that may ``escape to infinity'' as \(n \to \infty\). We are particularly interested in polynomials whose zeros are all real and share the same sign. In this setting, we identify the point at infinity, $\infty$, with either $+\infty$ (when the roots are nonnegative) or $-\infty$ (when the roots are nonpositive). Accordingly, $\lsem P\rsem_n$ is supported on $[0,+\infty]$ in the former case, and on $[-\infty,0]$ in the latter. We also consider a slightly more general setting in which $\lsem P\rsem_n$ is supported on either $[-A, +\infty]$ or $[-\infty, A]$ for some $A \geq 0$, with the same convention on the point at infinity applied.

It turns out that the following asymptotic characteristic of a sequence of polynomials $(P_n)_{n\in\N}$, called an \emph{exponential profile}, plays a central role in determining the limiting empirical distribution of their zeros via their coefficients.

\begin{definition}\label{def:exp_profile}
For every $n\in \N$ let $P_n(x) = \sum_{k=0}^n a_{k:n} x^k$ be a polynomial of degree at most $n$ with nonnegative coefficients $a_{k:n}\ge 0$.  We say that the sequence $(P_n)_{n\in \N}$ has an \emph{exponential profile} $g$  if $g:(\underline{m},\overline{m})\to \R$ is a function defined on a nonempty interval $(\underline{m},\overline{m})\subseteq [0,1]$  such that
\begin{equation}\label{eq:exp_profile_definition}
\lim_{n\to\infty}\frac{1}{n}\log a_{\lfloor \alpha n \rfloor:n}=
\begin{cases}
g(\alpha),&\text{if }\alpha\in (\underline{m},\overline{m}),\\
-\infty,&\text{if }\alpha\in [0,\underline{m})\cup (\overline{m},1].
\end{cases}
\end{equation}
Here, and in the following, we stipulate that $\log 0 = -\infty$.
\end{definition}

In Section~\ref{sec:zeros_profiles}, we shall present two theorems (Theorems~\ref{theo:exp_profile_implies_zeros} and~\ref{theo:zeros_imply_exp_profile}) that essentially state that for a sequence of polynomials $(P_n)_{n \in \mathbb{N}}$ with only nonpositive roots, the sequence of empirical measures $\lsem P_n \rsem_n$ converges weakly if and only if the sequence $(P_n)_{n \in \mathbb{N}}$ possesses an exponential profile. Furthermore, the limiting measure can be recovered from the profile $g$, and conversely, the profile $g$ can be recovered from the limiting measure. While these theorems are not new and have been known in the literature in various degrees of generality, to the best of our knowledge, our results are the most general ones. A review of previous results will be provided in Section~\ref{sec:literature1}.

In contrast to the analytic methods previously used, our approach to proving the aforementioned equivalence adopts a probabilistic perspective. Specifically, we interpret polynomials with only nonpositive roots as generating functions for sums of independent Bernoulli random variables. In this framework, the coefficients correspond to the probability mass function of the sum of these Bernoulli variables, which take values in \(\{0,1\}\). This probabilistic viewpoint allows us to leverage powerful tools from probability theory, such as large deviation techniques and the exponential change of measure, to investigate the relationship between the zeros and the coefficients. The probabilistic proofs of Theorems~\ref{theo:exp_profile_implies_zeros} and~\ref{theo:zeros_imply_exp_profile} will be given in Section~\ref{sec:proof_main1}.

Even though the proposed method for investigating the empirical distribution of zeros via profiles is highly versatile and applicable to a broad class of classical polynomial ensembles, in this paper we focus on specific applications of the aforementioned equivalence. A unifying source for limiting zero distributions of a great number of examples will be~\cite{jalowy_kabluchko_marynych_zeros_profiles_part_II}, where we apply our approach to hypergeometric polynomials including the classical Hermite, Laguerre and Jacobi polynomials as well as Touchard, Fubini, Eulerian and Narayana polynomials. Here, we concentrate on two particular applications of the method, each of which is of independent interest. Namely:
\begin{enumerate}
    \item In Section~\ref{sec:convolutions}, we study the behavior of zeros under various operations on polynomials, collectively referred to as `convolutions'. In particular, we demonstrate how the method of profiles can be employed to characterize the limiting empirical distribution of zeros under Hadamard convolution, as well as finite free additive and multiplicative convolutions~\cite{marcus2021polynomial,marcus_spielman_srivastava}. These operations are known to preserve real-rootedness.   These problems have been studied by~\citet{arizmendiperales} and~\citet{arizmendi}  who used a combinatorial approach based on finite free cumulants.
    \item In Section~\ref{sec:rep_diff}, we examine the dynamics of real zeros under repeated differentiation, yet another operation that preserves real-rootedness.
        Starting with the work of Steinerberger~\cite{steinerberger_real,Steiner21}, this problem, along with its relatives,  has received much attention and has been studied from different viewpoints  in~\cite{alazard_lazar_nguyen,arizmendi,bogvad_etal,CampbellAppell,COR23,COR24,feng_yao,galligo,galligo_najnudel_dynamics_randomized,galligo_najnudel_vu,gorin_marcus,diff-paper,HK21,kabluchko_rep_diff_free_poi,kiselev_tan,martinezfinkelshtein2024flowzerospolynomialsiterated,martinezfinkelshtein2025weightedequilibriumflowderivatives,michelen_vu,michelen_vu_almost_sure,orourke_steinerberger_nonlocal,shlyakhtenko_tao,totik_critical,totik_distribution_of_critical}.
\end{enumerate}

Finally, we note that the method of profiles proves particularly powerful when applied to random polynomials. The central idea is to relate the random fluctuations of profiles (expressed through a functional limit theorem) to the rate of convergence of the empirical zero distribution. Since the current paper is already quite dense, a detailed exploration of this direction is deferred to~\cite{jalowy_kabluchko_marynych_zeros_profiles_part_III}, the third installment in this series.

The remainder of the paper is structured as follows. In Section~\ref{sec:zeros_profiles}, we present our main tools for the method of profiles, specifically Theorems~\ref{theo:exp_profile_implies_zeros} and~\ref{theo:zeros_imply_exp_profile}. Section~\ref{sec:convolutions} investigates polynomial convolutions. In particular, Theorems~\ref{theo:finite_free_additive_conv_free_additive} and~\ref{theo:finite_free_mult_conv} provide new proofs of the convergence relations `$\boxplus_n \to \boxplus$' and `$\boxtimes_n \to \boxtimes$'. Section~\ref{sec:rep_diff} is devoted to analyzing repeated differentiation of real-rooted polynomials via the method of profiles. The main result of this section, Theorem~\ref{theo:repeated_diff_general}, characterizes the limiting empirical distribution of zeros for the polynomials $(z^a ({\rm d}/{\rm d}z)^b)^{\ell}P_n(z)$, where $\ell/n$ converges to a positive constant. Proofs of the main results are provided in Sections~\ref{sec:proof_main1},~\ref{sec:proof_convolutions}, and~\ref{sec:proof_repeated}.

\section{Zeros and profiles of polynomials}\label{sec:zeros_profiles}
\subsection{Main results on zeros and profiles}
We are interested in the limiting empirical distribution of zeros. As we shall see, the most convenient framework to characterize the weak limit of $\lsem P \rsem_n$ is via Cauchy transforms.

\begin{definition}
For any probability measure $\mu$ on $\R$ its \emph{Cauchy transform} $G_{\mu}:\C\backslash \supp(\mu)\to\C$ is defined by
\begin{align}
G_{\mu}(t)=\int_{\R}\frac{\mu({\rm d}z)}{t-z},\quad t\in\C\backslash \supp(\mu).
\end{align}
\end{definition}
By the Stieltjes--Perron inversion formula, see~\cite[Proposition 2.1.2 on p.~35]{pastur_shcherbina_book} and~\cite[pp.~124--125]{akhiezer_book}, the Cauchy transform uniquely determines the measure $\mu$ via
$$
\mu (I) =
-\lim_{y\to 0+} \int_I\frac 1 \pi\Im G_{\mu}(x + iy){\rm d}x,
$$
where $I$ is any interval such that $\mu$ is continuous at its endpoints. Furthermore, locally uniform convergence on $\mathbb{C}\backslash\R$ of the sequence of Cauchy transforms is equivalent to weak convergence of the corresponding sequence of probability measures. We shall omit the index $\mu$ and simply write $G$ for the Cauchy transform of $\mu$ if the measure is clear from the context.

Observe that the definition of Cauchy transforms extends naturally to probability measures on $(-\infty,+\infty]$ or $[-\infty,+\infty)$. However, one has to be careful with two-sided compactifications. For example, two different probability measures $\mu_1(\pm\infty)=1/2$ and $\mu_2(+\infty)=1$ have identical Cauchy transforms $G_{\mu_1}=G_{\mu_2}\equiv 0$. Fortunately, in this paper we shall not encounter such problems for all the measures appearing throughout the paper are supported either by $(-\infty,+\infty]$ or $[-\infty,+\infty)$. For such measures, a value of a possible atom at $\pm\infty$ can be identified using the formula
$$
1 - \mu(\{+\infty\})-\mu(\{-\infty\}) = \lim_{t\to+\infty} (\ii t) G_{\mu}(\ii t).
$$
If $\mu$ is supported by $(-\infty,A]$ (respectively, $[A,+\infty)$) for some $A\in\mathbb{R}$, then
\begin{equation}\label{eq:atom_at_infinity_over_R}
1 -\mu(\{-\infty\}) = \lim_{t\to+\infty} t G_{\mu}(t)\quad\quad \left(\text{respectively, }1 -\mu(\{+\infty\}) = \lim_{t\to-\infty} t G_{\mu}(t)\right).
\end{equation}

We are now ready to formulate a result which is our first main tool.

\begin{theorem}[Profile for coefficients implies limiting distribution of zeros]
\label{theo:exp_profile_implies_zeros}
For every $n\in \N$ let $P_n(x) = \sum_{k=0}^n a_{k:n} x^k$ be a polynomial of degree at most $n$. Suppose that the sequence $(P_n)_{n\in\N}$ possesses an exponential profile $g:(\underline{m},\overline{m})\to\mathbb{R}$ in the sense of Definition~\ref{def:exp_profile}. If, additionally, all roots of $P_n$ are nonpositive for all sufficiently large $n$,  then the following hold.
\begin{enumerate}[(a)]
\item We have weak convergence
$$\lsem P_n\rsem_n\toweak\mu$$
of probability measures on $[-\infty,0]$ for some $\mu$.
\item The measure $\mu$ is uniquely determined by the profile $g$ as follows. If $G_{\mu}$ denotes the Cauchy transform of $\mu$, then $t\mapsto tG_{\mu}(t)$, $t>0$, is the inverse of the function\footnote{The derivative $g'$ exists on $(\underline{m},\overline{m})$, see part (d) below.} $\alpha\mapsto\eee^{-g'(\alpha)}$, $\alpha\in (\underline{m},\overline{m})$. The function $t\mapsto tG_{\mu}(t)$ is a strictly increasing continuous bijection between $(0,+\infty)$ and $(\underline{m},\overline{m})$.
\item Furthermore,
\begin{equation}\label{eq:main_polys_converse_assumption_p_n(1)}
\lim_{n\to\infty}\frac{\log P_n(1)}{n}=\sup_{\alpha\in (\underline{m},\overline{m})} g(\alpha) = :M_g.
\end{equation}
and the function $-g(\cdot) + M_g$ is the Legendre transform of $u\mapsto \Psi(\eee^u)$, where
\begin{equation}\label{eq:psi_for_polys}
\Psi(t)=\int_{(-\infty,0]}\log \left(\frac{t-z}{1-z}\right) \mu({\rm d}z),\quad t>0,
\end{equation}
is the normalized logarithmic potential of $\mu$. This means that $g(\alpha) = M_g+\inf_{u\in \R} ( \Psi(\eee^u)-\alpha u )$ for all $\alpha \in (\underline{m},\overline{m})$.
\item The function $g$ is automatically infinitely differentiable and strictly concave on $(\underline{m},\overline{m})$. Moreover,  even the function $\alpha\mapsto g(\alpha)+\alpha\log\alpha$ remains strictly concave.
\item It holds $\mu(\{0\})=\underline{m}$ and $\mu(\{-\infty\})=1-\overline{m}$.
\end{enumerate}
\end{theorem}

\begin{remark}
The usual logarithmic potential $U$ (which may be not well defined for unboundedly supported  $\mu$ and is undefined if $\mu(\{-\infty\})>0$) is given by
$$
U(t):=\int_{(-\infty,0]}\log |t-z|\mu({\rm d}z)=\Re\left(\int_{(-\infty,0]}\log (t-z)\mu({\rm d}z)\right)
$$
and satisfies $\Psi(t)=U(t)-U(1)$ for real $t>0$. Up to an additive constant, the logarithmic potentials $\Psi$ and $U$ are antiderivatives of $G_{\mu}$, satisfying $\Psi^{\prime}(t)=G_{\mu}(t)$ for $t>0$.
\end{remark}

Our second main tool is the converse of Theorem~\ref{theo:exp_profile_implies_zeros}, revealing that existence of limit zero distributions and exponential profiles are equivalent up to a normalization of the polynomial which leaves distribution of zero unchanged.

\begin{theorem}[Distribution of zeros implies exponential profile]
\label{theo:zeros_imply_exp_profile}
For every $n\in \N$ let $P_n(x) = \sum_{k=0}^n a_{k:n} x^k$ be a polynomial of degree at most $n$ with real, nonpositive roots. If $\lsem P_n\rsem_n$ converges weakly to some probability measure $\mu$ on $[-\infty, 0]$, then the following hold.
\begin{enumerate}[(a)]
\item Define $\underline{m}:= \mu(\{0\})$ and $\overline{m} := 1-\mu(\{-\infty\})$ and observe that $0\leq \underline{m} \leq \overline{m}\leq 1$.  There exists a strictly concave and infinitely differentiable function $g:(\underline{m},\overline{m})\to (-\infty,0]$ such that
\begin{equation}\label{eq:zeros_imply_exp_profile_claim}
\sup_{(\underline{m}+\eps) n \leq k  \leq (\overline{m}-\eps)n}
\left|\frac 1n  \log \frac {a_{k :n}}{P_n(1)} -g\left(\frac k  n\right)  \right| \ton 0,
\end{equation}
for all sufficiently small $\eps > 0$ as well as
\begin{equation}\label{eq:divergence_outside_polys}
\sup_{0\leq k  \leq (\underline{m}-\eps)n} \frac 1n   \log \frac {a_{k :n}}{P_n(1)} \ton  -\infty,
\qquad
\sup_{(\overline{m}+\eps)n \leq k  \leq n} \frac 1n   \log \frac {a_{k :n}}{P_n(1)} \ton  -\infty.
\end{equation}
\item If $\underline{m}  < \overline{m}$, then the Cauchy transform $G_{\mu}$ of $\mu$ is such that $(0,+\infty)\ni t\mapsto tG_{\mu}(t)$ is the inverse of $(\underline{m},\overline{m})\ni\alpha\mapsto \eee^{-g'(\alpha)}$. 
\item In particular, if $\underline{m}  < \overline{m}$, then $g'(\underline{m}+)=\lim_{\alpha \to \underline{m}+} g'(\alpha) = +\infty$ and $g'(\overline{m}-)=\lim_{\alpha  \to \overline{m}-} g'(\alpha) = -\infty$.
\item Analogously\footnote{Observe that $g$ in~\eqref{eq:zeros_imply_exp_profile_claim} is defined as the profile of the normalized polynomials $P_n(x)/P_n(1)$. Thus, $M_g=\sup_{\alpha\in (\underline{m},\overline{m})}g(\alpha)=0$, see~\eqref{eq:main_polys_converse_assumption_p_n(1)}.}, the function $-g$ is the Legendre transform of $u\mapsto \Psi(\eee^u)$ defined by~\eqref{eq:psi_for_polys}.
\item Furthermore,
\begin{equation}\label{eq:P_s_n_k_quotients_polys}
\sup_{(\underline{m}+\varepsilon)n\leq k \leq (\overline{m}-\varepsilon)n}\left|\frac{a_{k +1:n}}{a_{k :n}} -\eee^{g'\left(\frac {k }{n}\right)}\right|\ton 0.
\end{equation}
Moreover, writing $\underline{r}_n := \inf\{k\geq 0:\,a_{k:n}\neq 0\}$, we have
\begin{equation}\label{eq:P_s_n_k_quotients_infinity_polys}
\inf_{\underline{r}_n\leq k \leq (\underline{m}-\varepsilon) n} \frac{a_{k +1:n}}{a_{k :n}}
\ton +\infty
\quad\text{and}\quad
\inf_{(\overline{m}+\varepsilon)n\leq k  < \deg P_n} \frac{a_{k :n}}{a_{k +1:n}}
\ton +\infty.
\end{equation}
\item If $\underline{m}=\overline{m}$, then~\eqref{eq:divergence_outside_polys} and~\eqref{eq:P_s_n_k_quotients_infinity_polys} hold true.
\end{enumerate}
\end{theorem}

\begin{remark}
In the above statements of Theorem \ref{theo:zeros_imply_exp_profile} and throughout the paper, we stipulate that $\eps>0$ is sufficiently small to make the sets over which a supremum or an infimum is taken nonempty.
\end{remark}

\begin{remark}[A non-rigorous argument]\label{rem:non_rig_argument_profile_zeros}
Assuming the existence of an exponential profile $g$ of $P_n(x)$ and the limit zero distribution $\mu$, let us provide a quick (but non-rigorous) derivation of the formula $g(\alpha) = M_g + \inf_{u\in \R} (\Psi(\eee^u)-\alpha u )$ in Theorem~\ref{theo:exp_profile_implies_zeros}. Assume for simplicity that $\deg P_n=n$. Observe that, for positive $x>0$, with $-\lambda_{1:n}\leq 0, \ldots, -\lambda_{n:n}\leq 0$ denoting zeros of $P_n$, we have
$$
\frac 1n \log \frac{P_n(x)}{P_n(1)} = \frac 1n \sum_{k=1}^n \log \frac{x + \lambda_{k:n}}{1+\lambda_{k:n}}\ton \int_{(-\infty,0]}\log \left(\frac{x-z}{1-z}\right)\mu(\dd z)=\Psi(x).
$$
On the other hand, since we assume existence of the profile $g$ of $P_n$,
$$
P_n(x) = \sum_{k=0}^n a_{k:n} x^k= \sum_{k=0}^n \eee^{ng(k/n) + k \log x + o(n)}.
$$
Formally, by a Laplace method we expect, for $x>0$,
$$
\frac 1n \log P_n(x) \ton \sup_{\alpha \in (\underline{m},\,\overline{m})} \left(g(\alpha)+\alpha \log x\right)
$$
and, plugging $x=1$,
$$
\frac 1n \log P_n(1) \ton \sup_{\alpha \in (\underline{m},\,\overline{m})} g(\alpha)=M_g.
$$
Therefore, for $u\in\mathbb{R}$,
$$
\Psi(\eee^u)=\sup_{\alpha \in (\underline{m},\,\overline{m})} \left(g(\alpha)+\alpha u\right)-M_g.
$$
Therefore, $u\mapsto \Psi(\eee^u)$ is the Legendre transform of $-g(\cdot)+M_g$, or, equivalently, $-g(\cdot)+M_g$ is the Legendre transform of $u\mapsto \Psi(\eee^u)$.
\end{remark}

In the next result we characterize the class of functions that can appear as exponential profiles of nonpositive-rooted polynomials. It is similar to Proposition~6.3 in~\cite{bercovici_voiculescu} which characterizes the class of $\chi$-transforms, and to Bochner's theorem for characteristic functions.

\begin{proposition}[Characterization of profiles]\label{prop:char_profiles}
Let $g:(\underline{m},\overline{m})\to \R$ be a function defined on a nonempty interval $(\underline{m},\overline{m})\subseteq [0,1]$. Then, $g$ is an exponential profile of some sequence $(P_n(x))_{n\in \N}$ of polynomials with only nonpositive roots in the sense of Definition~\ref{def:exp_profile} if and only if the following hold:
\begin{itemize}
\item [(i)] $g$ is infinitely differentiable and strictly concave on $(\underline{m},\overline{m})$ with $g'(\underline{m}+) = +\infty$, $g'(\overline{m}-) = -\infty$;
\item [(ii)] $\eee^{g'}$ admits an extension to an analytic, bijective function $\eee^{g'}: \cD \to \{z\in \C: \Re z >0\}$ defined in some domain $\cD\subseteq \C$ which is invariant with respect to complex conjugation and satisfies $\cD \cap \R = (\underline{m},\overline{m})$ and, moreover, the matrix
\begin{equation}\label{eq:nevanlinna_pick_matrix_proof_char_profiles}
\left(\frac{y_j \eee^{g'(y_j)} - \overline{y_k} \eee^{g' (\overline{y_k})}}{\eee^{g'(y_j)} - \eee^{g'(\overline{y_k})}}\right)_{j,k=1}^\ell
\end{equation}
is positive semi-definite for all $\ell\in \N$ and all  $y_1,\ldots, y_\ell \in \cD$ with $\Re y_1>0,\ldots, \Re y_\ell>0$.
\end{itemize}
\end{proposition}

The proof of Theorems~\ref{theo:exp_profile_implies_zeros} and~\ref{theo:zeros_imply_exp_profile} will be given in Section~\ref{sec:proof_main1}. The proof of Proposition~\ref{prop:char_profiles} will be given in Section~\ref{subsec:proof_char_profiles}.

\subsection{Related results in the literature}\label{sec:literature1}
Theorems~\ref{theo:exp_profile_implies_zeros} and~\ref{theo:zeros_imply_exp_profile} are not entirely new. \citet[Theorem~1]{van_assche_fano_ortolani} proved a version of Theorem~\ref{theo:zeros_imply_exp_profile}, including~\eqref{eq:zeros_imply_exp_profile_claim} and~\eqref{eq:P_s_n_k_quotients_polys}, under more restrictive assumptions which enforce $\underline m=0$ and $\overline m =1$. These authors~\cite[Section~3]{van_assche_fano_ortolani} also give several applications.  They compute exponential profiles of certain orthogonal polynomials  and of iterations $T(z), T(T(z)), T(T(T(z))), \ldots$ of a fixed polynomial $T$ assuming that the Julia set of $T$ is contained in $(-\infty,0]$ . All these  results can be found in~\citet[Section~5.1]{van_assche_book_asymptotics_ortho_polys}. Earlier works containing weaker results or special cases are~\cite{fano_poly_maps,fano_galavotti_dense_sums,fano_ortolani_van_assche_orthogonal}.  Motivated by zero distribution of certain biorthogonal polynomials, \citet[Theorem~3.1]{lubinski_sidi_composite_biorthogonal} proved an extension of~\cite[Theorem~1]{van_assche_fano_ortolani} which is very close to  Theorem~\ref{theo:zeros_imply_exp_profile} but requires $\lsem P_n\rsem$ to be concentrated on $[-A,0]$ for some constant $A$.
Lubinski and Stahl~\cite[Theorem~2]{lubinski_stahl_some_explicit} applied~\cite[Theorem~1]{van_assche_fano_ortolani} to determine the limiting zero distribution of certain biorthogonal polynomials. Although the papers listed above do not explicitly state a general result of the form ``profile implies zeros'', the idea is implicitly used to identify the limiting distribution of zeros for several specific families of polynomials.

In a very recent contribution, \citet{arizmendi2024s} proved that the existence of a limiting empirical zero distribution for a sequence of polynomials $(P_n)_{n\in \N}$ with nonpositive roots is, essentially, equivalent to the existence of the limit of the ratios of consecutive coefficients of $P_n$. More precisely, using our notation, their Theorem 1.1 states that $\lsem P_n\rsem \toweak \mu$ is equivalent to~\eqref{eq:P_s_n_k_quotients_polys}, under the assumptions that all roots of $P_n$ are nonpositive, $\deg P_n = n$, and $\overline{m} = 1$. Moreover, they relate the limiting function $\eee^{g'}$ in~\eqref{eq:P_s_n_k_quotients_polys} to the so-called $S$-transform of $\mu$. It is not difficult to verify that~\eqref{eq:P_s_n_k_quotients_polys} is equivalent to~\eqref{eq:zeros_imply_exp_profile_claim}. Therefore, Theorem 1.1 in~\cite{arizmendi2024s} implies Theorems~\ref{theo:exp_profile_implies_zeros} and~\ref{theo:zeros_imply_exp_profile} when $\overline{m} = 1$ and $\deg P_n = n$. Note that their proof is based on \cite[Theorem 1.2]{arizmendi2024s}, going back to \cite[\S 4]{HK21} which in turn relies upon the work \cite[Theorem~1]{van_assche_fano_ortolani} again. The methods used in the proof of Theorem 1.2 in~\cite{arizmendi2024s} are primarily analytical, relying on the analysis of free cumulants. In contrast, we present purely probabilistic proofs of Theorems~\ref{theo:exp_profile_implies_zeros} and~\ref{theo:zeros_imply_exp_profile}, relying on large deviation estimates for sums of independent Bernoulli random variables. This connection stems from the observation that $P_n(x)/P_n(1)$ serves as the generating function for such a sum. We shall further elaborate on this connection and provide additional references to the relevant probabilistic literature in Section~\ref{subsec:large_deviations}.

\citet{Eremenko} considers a sequence $f_n(z) = \sum_{k=0}^n a_{k:n} z^k$, $n\in \N$, of polynomials  and assumes that (i) $\frac 1n \log |f_n(z)| \to u(z)$ in $L^1_{\text{loc}}$ for some subharmonic function $u\nequiv -\infty$  and (ii) the sequence $\frac 1n \log |f_n(z)|$ is uniformly bounded above on every compact subset of $\C$. Then, $\lsem f_n \rsem_n \to (2\pi)^{-1} \Delta u$ as $n\to\infty$, where $\Delta$ is the distributional Laplacian.  Note that the polynomials $f_n$  are not required to be real-rooted.  Let $\psi_n(x)$ be the least convex minorant of the mapping $\frac k n \mapsto -\frac 1n  \log |a_{k:n}|$, more precisely
$$
\psi_n(x) := \sup_{\substack{\psi: [0,\infty) \to \R, \; \psi \text{ is convex}}} \left\{ \psi(x): \psi\left(\frac kn\right)  \leq - \frac 1n \log |a_{k:n}| \text{ for all } k\in \{0,\ldots, n\} \right\}, \qquad x\geq 0.
$$
\citet[Theorem~1]{Eremenko} proves that $\psi_n(x)$ converges to some limit $\varphi$ and characterizes it as $\varphi(x) = \sup_{t\in \R} (tx - \Phi(t))$, the Legendre transform of $\Phi(t) := \max_{|z| = \eee^t}|u(z)|$. Conversely, if $\psi_n \to \varphi$, then $\frac 1n \log |f_n(z)| \to u(z)$ holds along a subsequence.   If  the sequence $(|a_{k:n}|)_{k=0}^n$ is log-concave (for example, if all zeros of $f_n$ are nonpositive), then $\psi_n$ becomes the linear interpolation of the map $\frac k n \mapsto -\frac 1n  \log |a_{k:n}|$ and Eremenko's result gives a formula for the exponential profile of the coefficients.   \citet{BE15} characterize the set of possible weak limits of empirical zero distributions for polynomials with positive coefficients.

Several authors proved limit theorems for elementary symmetric polynomials of i.i.d.\ random variables. More precisely, let $X_1,X_2,\ldots$ be i.i.d.\ random variables and consider the polynomials $P_n(x):= \prod_{k=1}^{n} (x+X_i) =: \sum_{k=0}^n a_{k:n} x^k$. Their coefficients are the elementary symmetric polynomials
$$
a_{n-k:n}= \sum_{1\leq i_1< \ldots < i_{n-k} \leq n} X_{i_1}\ldots X_{i_{k}}, \qquad k\in \{0,\ldots, n\}.
$$
Assuming that $X_1\geq 0$ a.s., $\E \log (1+X_1)<\infty$ and $k,n\to\infty$ such that $k/n\to \alpha\in (0,1)$,  \citet{halasz_szekely} proved that $(a_{n-k:n}/\binom nk)^{1/k}$ converges a.s.\ and characterized the limit. This result follows from Theorem~\ref{theo:zeros_imply_exp_profile} since, with probability $1$, $\lsem P_n\rsem$ converges weakly to the distribution of $-X_1$. \citet{bochi_etal} called these limits (which depend on $\alpha$) H\'al\'asz--Sz\'ekely barycenters (of the distribution of $X_1$) and studied their properties.  Central limit theorems for $a_{n-k:n}$ in the regime when $k\sim \alpha n$, $\alpha \in [0,1]$ were proved by~\citet{szekely_limit_theorem}, \citet{mori_szekely}, \citet{van_es}, \citet{van_es_helmers}, \citet{major_symm_iid}. Using our methods for deriving Theorems~\ref{theo:exp_profile_implies_zeros} and~\ref{theo:zeros_imply_exp_profile} we shall complement these results with functional limit theorems in the upcoming work~\cite{jalowy_kabluchko_marynych_zeros_profiles_part_III}. Central limit theorems for random permanents were obtained by~\citet{rempala_wesolowski}.

The paper~\cite{KZ14} studied random polynomials of the form $R_n(z) = \sum_{k=0}^n a_{k:n} \xi_k z^k$, where  $\xi_0,\xi_1,\ldots$ are i.i.d.\ random variables with values in $\C$  and $a_{k:n}$ are complex numbers such that
$$
\sup_{k= 0,\ldots, n} \left|\frac{1}{n} \log |a_{k:n}| - g\left(\frac{k}{n}\right)\right| \to 0,\quad n\to\infty,
$$
where $g:[0,1]\to \R$ is a concave function,  the ``profile''.     Assuming that $\xi_0$ is not a.s.\ constant and $\E \log (1+|\xi_0|) <\infty$, it was shown in~\cite[Theorem~2.9]{KZ14} that the random probability measure $\lsem R_n\rsem_n$ converges weakly to $\nu$, a deterministic rotationally invariant probability measure on $\C$ characterized by the property $\nu (\{z\in \C: |z|\leq \eee^{-g'(\alpha-0)}\}) = \alpha$ for all $\alpha \in (0,1)$.  Since $\nu$ is rotationally invariant, its Cauchy transform $G_{\nu} (t) := \int_{\C}\frac{\nu({\dd}z)}{t-z}$ satisfies $t \, G_\nu(t) = \nu (\{z\in \C: |z| \leq  |t|\})$, for all $t\in \C\backslash \{0\}$ such that $\nu (\{z\in \C: |z|= |t|\}) = 0$. So, the functions $r\mapsto rG_\nu(r)$ and $\alpha\mapsto \eee^{-g'(\alpha)}$ are (generalized) inverses of each other. This is the same relation as in the real-rooted case studied in the present paper. The reason is that the non-rigorous argument sketched in Remark~\ref{rem:non_rig_argument_profile_zeros} applies to both settings. The techniques used to make the argument rigorous in both settings are, however, very different.

\section{Applications to convolutions of polynomials}\label{sec:convolutions}
\subsection{Hadamard product}
As a simple warm-up for the upcoming subsection, we begin by studying the \emph{Hadamard product of polynomials}. Recall that it is defined by taking component-wise products of coefficients, that is, for $P_n^{(1)}(x)=\sum_{k=0}^{n} a_{k:n}^{(1)} x^k$ and $P_n^{(2)}(x)=\sum_{k=0}^{n} a_{k:n}^{(2)} x^k$ with degrees at most $n$ we define
$$
\big[P_n^{(1)}\odot_n P_n^{(2)}\big](x)=\sum_{k=0}^{n} a_{k:n}^{(1)}a_{k:n}^{(2)} x^k.
$$
In what follows, we denote by $f^{\leftarrow}$ the inverse of a continuous strictly monotone function $f$ defined on some interval of $\R$.
\begin{theorem}\label{theo:hadamard-product}
For every $n\in\mathbb{N}$ let $P_n^{(1)}(x)$ and $P_n^{(2)}(x)$ be polynomials of degree at most $n\in\N$.  Let their roots be nonpositive and
$$
\lsem P_n^{(1)}\rsem_n \toweak \mu_1,
\qquad
\lsem P_n^{(2)}\rsem_n \toweak \mu_2,
$$
for some probability measures $\mu_1$ and $\mu_2$ on $[-\infty,0]$ with Cauchy transforms $G_{\mu_1}$ and $G_{\mu_2}$, respectively. Also suppose that the intervals $(\mu_1(\{0\}), 1 - \mu_1 (\{-\infty\}))$ and  $(\mu_2(\{0\}), 1 - \mu_2 (\{-\infty\}))$ have a nonempty intersection. Then, the limiting zero distribution
\begin{equation}\label{eq:hadamard_claim_weak_convergence}
\lsem P_n^{(1)} \odot_n P_n^{(2)}\rsem_n \toweak \mu
\end{equation}
exists as a probability measure on $[-\infty, 0]$ and is determined by its Cauchy transform $G_\mu$ by the rule: the inverse of $t\mapsto tG_{\mu}(t)$ is the product of the inverse of $t\mapsto tG_{\mu_1}(t)$ and the inverse of $t\mapsto tG_{\mu_2}(t)$.
\end{theorem}
Here is a quick proof based on a straightforward application of Theorems~\ref{theo:exp_profile_implies_zeros} and~\ref{theo:zeros_imply_exp_profile}.
\begin{proof}
It is well known that the Hadamard product preserves real-rootedness, see, for example,~\cite{braenden_on_lin_transfor_preserving}. According to Theorem~\ref{theo:zeros_imply_exp_profile}(a), the normalized polynomials $P_n^{(1)}(x)/P_n^{(1)}(1)$ and $P_n^{(2)}(x)/P_n^{(2)}(1)$ have profiles $g_1$ and $g_2$ defined on
$(\mu_1(\{0\}), 1 - \mu_1 (\{-\infty\}))$ and  $(\mu_2(\{0\}), 1 - \mu_2 (\{-\infty\}))$, respectively. Thus, according to~\eqref{eq:zeros_imply_exp_profile_claim} and~\eqref{eq:divergence_outside_polys}, the polynomial $[P_n^{(1)}\odot P_n^{(2)}](x)/(P_n^{(1)}(1)P_n^{(2)}(1))$ has the profile $g_1+g_2$ defined on the intersection of $(\mu_1(\{0\}), 1 - \mu_1 (\{-\infty\}))$ and $(\mu_2(\{0\}), 1 - \mu_2 (\{-\infty\}))$ (and is equal to $-\infty$ on interior of the complement). According to Theorem~\ref{theo:exp_profile_implies_zeros}(a) this implies~\eqref{eq:hadamard_claim_weak_convergence} and the Cauchy transform of the limit $\mu$ is recovered by
Theorem~\ref{theo:exp_profile_implies_zeros}(b). Observe that the inverses of $t\mapsto tG_{\mu_1}(t)$ and $t\mapsto tG_{\mu_2}(t)$ are both defined on the intersection of $(\mu_1(\{0\}), 1 - \mu_1 (\{-\infty\}))$ and $(\mu_2(\{0\}), 1 - \mu_2 (\{-\infty\}))$ and can be multiplied there. The proof is complete.
\end{proof}
\begin{remark}
The intervals $(\mu_1(\{0\}), 1 - \mu_1 (\{-\infty\}))$ and $(\mu_2(\{0\}), 1 - \mu_2 (\{-\infty\}))$ have a nonempty intersection if and only if
$$
\max\{\mu_1(\{0\}),\mu_2(\{0\})\}+\max\{\mu_1(\{-\infty\}),\mu_2(\{-\infty\})\}<1.
$$
If this intersection is empty, then $P_n^{(1)} \odot_n P_n^{(2)}$ can be zero everywhere and $\lsem P_n^{(1)} \odot_n P_n^{(2)}\rsem_n$ can be undefined.
\end{remark}

\subsection{Finite free convolutions of polynomials}
We shall now proceed to more sophisticated convolutions on polynomials coming from finite free probability.

\begin{definition}[Finite free convolutions]\label{def:finite_free_convolution}
Consider two polynomials of degree at most $n$  
$$
P_n^{(1)}(x) = \sum_{k=0}^{n} a_{k:n}^{(1)} x^k,
\qquad
P_n^{(2)}(x) = \sum_{k=0}^{n} a_{k:n}^{(2)} x^k.
$$
The finite free additive convolution and the finite free multiplicative convolution of these polynomials are defined by
\begin{align*}
\sum_{k=0}^n a_{k:n}^{(1)} x^k \boxplus_n \sum_{k=0}^n a_{k:n}^{(2)} x^k &= \frac 1 {n!} \sum_{k=0}^n \frac{x^k}{k!} \sum_{j_1+j_2 = n+k} a_{j_1:n}^{(1)} a_{j_2:n}^{(2)} j_1! j_2!,
\\
\sum_{k=0}^n a_{k:n}^{(1)} x^k \boxtimes_n \sum_{k=0}^n a_{k:n}^{(2)} x^k &=  \sum_{k=0}^n (-1)^{n-k} \frac{a_{k:n}^{(1)} a_{k:n}^{(2)}}{\binom n k} x^k.
\end{align*}
\end{definition}

Both convolutions first appeared more than a century ago. The finite free additive convolution is due to Walsh~\cite{walsh} and the finite free multiplicative convolution due to Szeg\"o~\cite{szegoe_bemerkungen_grace}. Initiated by the rediscovery of the convolutions by Marcus, Spielman and Srivastava \cite{marcus_spielman_srivastava} (see also~\cite{marcus2021polynomial}), the connection to free probability now became an active field of research, see for example~\cite{arizmendiperales,arizmendi,arizmendi2024s,CampbellAppell,COR24,HK21,kabluchko_rep_diff_free_poi,kabluchko2024leeyangzeroescurieweissferromagnet,martinez,mirabelli2021hermitian}.
The next theorem is due to  Walsh~\cite[Theorem~VI and p.~180]{walsh} and Szeg\"o~\cite[Satz~4']{szegoe_bemerkungen_grace}; see also Theorems~1.3 and~1.6 in~\cite{marcus_spielman_srivastava}.

\begin{theorem}[Convolutions preserve real-rootedness]\label{thm:real_rootedness}
The following claims hold true:
\begin{itemize}
\item[(i)] If $P_n^{(1)}(x)$ and $P_n^{(2)}(x)$ are real-rooted polynomials of degree $\leq n$, such that $\deg P_{n}^{(1)} + \deg P_n^{(2)}\geq n$, then  $P_n^{(1)} \boxplus_n P_{n}^{(2)}$ is also real-rooted\footnote{If $\deg P_{n}^{(1)} + \deg P_n^{(2)} < n$, then $P_n^{(1)} \boxplus_n P_{n}^{(2)}\equiv 0$.}.
\item[(ii)] Let  $Q_n^{(1)}(x)$ and $Q_n^{(2)}(x)$ be non-constant polynomials, both of degree $\leq n$, that have only nonnegative roots. Assume also that none of these polynomials is divisible by $x^{\min (\deg Q_n^{(1)},\deg Q_n^{(2)})}$. Then, all roots of  $Q_n^{(1)}\boxtimes_n Q_n^{(2)}$ are nonnegative.
\end{itemize}
\end{theorem}

\citet[Theorems~1.3 and~1.6]{marcus_spielman_srivastava} state both parts for polynomials of degree exactly $n$, but they also explain how to extend both claims to polynomials of degree $\leq n$. For Part~(i),  if $d_1 := \deg P_n^{(1)}\leq n$ and $d_2 := \deg P_n^{(2)}\leq n$ satisfy $d_1+d_2\geq n$ (which guarantees that $P_n^{(1)} \boxplus_n P_{n}^{(2)}$ is not identically $0$ and has degree $d_1+d_2 - n$),  Lemma~1.16 in~\cite{marcus_spielman_srivastava} applied inductively gives
$$
P_n^{(1)}(x) \boxplus_n P_{n}^{(2)}(x) =  \frac{\left((\dd /\dd x)^{n-d_2} P_n^{(1)}(x)\right) \boxplus_{d_1+d_2-n}  \left((\dd /\dd x)^{n-d_1} P_n^{(2)}(x)\right)}{n(n-1) \ldots (d_1+d_2-n+1)}.
$$
Both polynomials on the right-hand side are real-rooted by Rolle's theorem and both have the same ``correct'' degree $d_1+d_2 - n$. Part~(ii), as stated above, follows from~\cite[Corollary~5.5.8]{rahman_schmeisser_book_polys} (taking $k=\ell = -1$ there). Alternatively, one can argue using Lemma~4.9 in~\cite{marcus_spielman_srivastava}.

\subsection{Convergence of \texorpdfstring{$\boxplus_n$}{boxplus\textunderscore n} to \texorpdfstring{$\boxplus$}{boxplus}}
It has been shown in~\cite[Corollary 5.5]{arizmendiperales} and~\cite[Theorem 4.3]{marcus2021polynomial}, that, in a suitable sense, the finite free additive convolution $\boxplus_n$ approaches the free additive convolution $\boxplus$,  as $n\to\infty$. By definition, if $\mu_1$ and $\mu_2$ are compactly supported measures on $\R$, then their free additive convolution $\mu_1\boxplus\mu_2$ is the distribution of the sum of two freely independent random variables with distributions $\mu_1,\mu_2$ respectively. The additive free convolution $\boxplus$ has been introduced in~\cite{voiculescu_addition}, \cite[\S 3.2]{voiculescu_nica_dykema_book} for compactly supported measures and extended in~\cite[\S5]{bercovici_voiculescu},  \cite{maassen_free} to probability measures on $\R$; see also~\cite[Chapter~3]{mingo_speicher_book}.
General references on  free probability are the books~\cite{mingo_speicher_book,nica_speicher_book,voiculescu_nica_dykema_book}.

The fact that $\boxplus_n$ converges to $\boxplus$ is deeper than it may look at a first sight. To prove it, Arizmendi and Perales~\cite{arizmendiperales} developed an elegant combinatorial  theory of finite free cumulants. The theory of zeros and profiles presented in the previous section gives rise to an entirely different proof of the aforementioned result while at the same time increasing generality by allowing for non-compactly supported distributions and for polynomials of degree $\leq n$ (rather than $=n$). Another advantage of our approach is that it applies, with minimal changes, to other types of polynomial convolutions that preserve real-rootedness, such as the free multiplicative convolution (see Theorem~\ref{theo:finite_free_mult_conv}) or the Hadamard product (see Theorem~\ref{theo:hadamard-product}).

\begin{theorem}[$\boxplus_n$ converges to $\boxplus$]\label{theo:finite_free_additive_conv_free_additive}
For every $n\in \N$ let $P_n^{(1)}(x)$ and $P_n^{(2)}(x)$ be polynomials of degree at most $n$.   Suppose that
\begin{itemize}
\item[(i)] all roots of these polynomials are real and upper bounded by $A$, where $A\in \R$ does not depend on $n$;
\item[(ii)] for some probability measures $\mu_1$ and $\mu_2$ on $[-\infty, A]$ we have
$$
\lsem P_n^{(1)}\rsem_n \toweak \mu_1,
\qquad
\lsem P_n^{(2)}\rsem_n \toweak \mu_2,
\qquad
\text{weakly on $[-\infty, A]$.}
$$
\item[(iii)]  $\mu_1(\{-\infty\}) + \mu_2(\{-\infty\}) < 1$.
\end{itemize}
Then, as $n\to\infty$, the probability measure $\lsem P_n^{(1)} \boxplus_n P_n^{(2)}\rsem_n$ converges weakly on $[-\infty, 2A]$ to some probability measure $\mu$ on $[-\infty, 2A]$.  If $\mu_1(\{-\infty\}) = \mu_2(\{-\infty\}) = 0$, then $\mu = \mu_1 \boxplus \mu_2$ is the classical free additive convolution of probability measures $\mu_1$ and $\mu_2$.
\end{theorem}

We shall recall an analytic definition of $\boxplus$ in Sections~\ref{subsec:R_S_transf} and~\ref{subsec:free_conv_defs} and prove Theorem~\ref{theo:finite_free_additive_conv_free_additive} in Section~\ref{subsec:proof_additive_conv}.

\begin{remark}
The additive finite free convolution is closely related to the classical convolution of functions, defined by the formula
\begin{equation}\label{eq:usual_conv}
[P_n^{(1)} * P_n^{(2)}](z) = \int_{0}^z P_n^{(1)}(u) P_n^{(2)}(z-u) \dd u.
\end{equation}
Namely, we claim that
\begin{equation}\label{eq:finite_free_conv_vs_usual_conv}
\frac 1 {n!} \left(\frac{\dd}{\dd z}\right)^{n+1} [P_n^{(1)} * P_n^{(2)}](z) = P_n^{(1)}(z) \boxplus_n P_n^{(2)}(z).
\end{equation}
By bilinearity of $*$ and $\boxplus_n$, it suffices to verify this identity for the polynomials $z^j$ and $z^k$ with $j,k\in \{0,\ldots, n\}$, which is straightforward.
\end{remark}

\subsection{Convergence of \texorpdfstring{$\boxtimes_n$}{boxtimes\textunderscore n} to \texorpdfstring{$\boxtimes$}{boxtimes}}
For every two probability measures $\nu_1$ and $\nu_2$ on $[0,\infty)$ it is possible to define their free multiplicative convolution $\nu_1\boxtimes \nu_2$, which is also a probability measure on $[0,\infty)$; see~\cite{voiculescu_multiplication} and~\cite[\S3.6]{voiculescu_nica_dykema_book} for the compactly supported case, \cite[Section~6]{bercovici_voiculescu} for the general case and~\cite{bercovici_voiculescu_levy_hincin,chistyakov_goetze_arithmetic,arizmendi_hasebe_kitagawa_free_mult_conv} for further properties. It is known that $\nu_1\boxtimes \nu_2 = \nu_2\boxtimes \nu_1$ and $\delta_0 \boxtimes \nu_2 = \delta_0$, where $\delta_0$ is the Dirac measure at $0$.
In a suitable sense, the finite free multiplicative convolution $\boxtimes_n$ converges, as $n\to\infty$,  to the free multiplicative convolution $\boxtimes$. This has been shown in terms of the so-called $S$-transform by~\citet[Theorem 4.8]{marcus2021polynomial} and in terms of weak convergence by~\citet[Theorem 1.4]{arizmendi}, assuming compactly supported distributions.  Using our results on exponential profiles we shall show that this assumption can be removed and beyond, the result holds for polynomials of degree $\leq n$ (rather than $=n$).

\begin{theorem}[$\boxtimes_n$ converges to $\boxtimes$]\label{theo:finite_free_mult_conv}
For every $n\in \N$ let $Q_n^{(1)}(x)$ and $Q_n^{(2)}(x)$ be polynomials of degree at most $n$, with nonnegative roots. Suppose that
\begin{equation}\label{eq:theo:finite_free_mult_assumpt_with_infinity}
\lsem Q_n^{(1)}\rsem_n \toweak \nu_1,
\qquad
\lsem Q_n^{(2)}\rsem_n \toweak \nu_2,
\qquad
\text{ weakly on } [0,+\infty],
\end{equation}
where $\nu_1$ and $\nu_2$ are probability measures on $[0, +\infty]$. Also suppose that the intervals $(\nu_1(\{0\}), 1 - \nu_1 (\{+\infty\}))$ and  $(\nu_2(\{0\}), 1 - \nu_2 (\{+\infty\}))$ have a nonempty intersection. Then,
\begin{equation}\label{eq:theo:finite_free_mult_statement}
\lsem Q_n^{(1)} \boxtimes_n Q_n^{(2)}\rsem_n \toweak \nu \text{ weakly on } [0,+\infty],
\end{equation}
for some probability measure $\nu$ on $[0,+\infty]$. If $\nu_1(\{+\infty\})= \nu_2(\{+\infty\}) = 0$, then $\nu = \nu_1 \boxtimes \nu_2$ is the classical free multiplicative convolution of $\nu_1$ and $\nu_2$.
\end{theorem}

\begin{remark}
The intervals $(\nu_1(\{0\}), 1 - \nu_1 (\{+\infty\}))$ and  $(\nu_2(\{0\}), 1 - \nu_2 (\{+\infty\}))$ have a nonempty intersection if and only if
\begin{equation}\label{eq:non-empty_intersection}
\max\{\nu_1(\{0\}),\nu_2(\{0\})\}+\max\{\nu_1 (\{+\infty\}),\nu_2 (\{+\infty\})\}<1.
\end{equation}
\end{remark}

We shall recall a definition of $\boxtimes$ in Sections~\ref{subsec:R_S_transf} and \ref{subsec:free_conv_defs}. The proof of Theorem~\ref{theo:finite_free_mult_conv} is given in Section~\ref{subsec:proof_multipl_conv}.

\section{Applications to the repeated action of \texorpdfstring{$z^a (\dd / \dd z)^b$}{z\textasciicircum a (d/dz)\textasciicircum b}}\label{sec:rep_diff}
As another application, let us study how repeatedly differentiating a polynomial affects its zeros.
The history of this research direction can be traced
back to at least the Gauss--Lucas theorem, stating that the
roots of the derivative of any polynomial are in the convex hull of the roots of the original polynomial.

\subsection{Main result on the repeated action of \texorpdfstring{$z^a (\dd / \dd z)^b$}{z\textasciicircum a (d/dz)\textasciicircum b}}
Consider a sequence of polynomials $Q_n(x)$, $n\in \N$,  such that $\deg Q_n\leq n$ and $\lsem Q_n \rsem_n \to \nu$ weakly on $\C$.
Is it true that the empirical distribution of zeros of the $\ell_n$-th derivative of $Q_n$ converges weakly as $n\to\infty$, where the order
of differentiation satisfies $\ell_n\sim \kappa n$ for $\kappa\in [0,1)$? For real-rooted polynomials with all roots in a fixed compact interval, it was shown in~\cite{arizmendi,HK21,Steiner21} that the limiting distribution of zeros of $Q_n^{(\ell_n)}$ turns out to be $\nu^{\boxplus \frac{1}{1-\kappa }}\big(\frac{\cdot}{1-\kappa }\big)$, where $\nu^{\boxplus \alpha}$ denotes the free self-convolution of $\nu$ of order $\alpha \geq 1$; see, \cite[Corollary~14.13]{nica_speicher_book} and~\cite[Section~1.2]{shlyakhtenko_tao} for its definition and pointers to the literature.

As it was observed in \cite{diff-paper}, it will be instructive to consider the repeated action of the more general differential operator
$$
\mathcal{A}_{a,b} = n^{-b} z^a (\dd / \dd z)^b,\qquad a,b\in \N_0,
$$
acting on the space $\R[z]$ of polynomials over $\R$, where $\N_0:=\{0,1,2,3,\ldots\}$. Our aim is to describe the limit of $\lsem \mathcal{A}_{a,b}^\ell Q_n(z)\rsem_n$, as $\ell, n\to\infty$ such that $\ell/n \to \kappa >0$. As a first step, let us rewrite $\lsem \mathcal{A}_{a,b}^\ell Q_n(z)\rsem_n$ in terms of the finite free convolution $\boxtimes_n$. Define
$$
\Delta:= a-b,
$$
which describes the amount of degree-increase (if $\Delta\geq 0$) or decrease (if $\Delta\leq 0$). It is easy to check that for  $\ell\in \N$, the $\ell$-th power of $\mathcal{A}_{a,b}$ acts on the monomial $z^j$ with $j\in \N_0$ as
$$
\mathcal{A}_{a,b}^\ell (z^j) = \frac{1}{n^{\ell b}} \frac{j!}{(j-b)!} \frac{(j+\Delta)!}{(j-b+\Delta)!}
\ldots \frac{(j+(\ell-1)\Delta)!}{(j-b + (\ell-1)\Delta)!} \cdot z^{j+\ell \Delta},
$$
provided $j\geq b$ (if $\Delta\geq 0$) or $j + \ell\Delta \geq a$ (if $\Delta\leq 0$). Let $\mathcal{J_\ell}$ be the set of $j\in \N_0$ fulfilling this condition. For $j\notin \mathcal{J}_\ell$ we have $\mathcal{A}_{a,b}^\ell (z^j) = 0$.
If $Q_n(z)  = \sum_{j=0}^n (-1)^{n-j} a_{j:n} z^j$ is a polynomial, then
$$
z^{-\ell \Delta} \cdot \mathcal{A}_{a,b}^\ell (Q_n(z))
=
\sum_{\substack{j\in \{0,\ldots, n\}\\ j\in \mathcal J_\ell}}
\frac{1}{n^{\ell b}} \frac{j!}{(j-b)!} \frac{(j+\Delta)!}{(j-b+\Delta)!}
\ldots \frac{(j+(\ell-1)\Delta)!}{(j-b + (\ell-1)\Delta)!} \cdot (-1)^{n-j} a_{j:n} z^{j}.
$$
Using Definition~\ref{def:finite_free_convolution}, this can be written as
\begin{equation}\label{eq:diff_as_finite_free_mult_convolution}
z^{-\ell \Delta} \cdot \mathcal{A}_{a,b}^\ell (Q_n(z)) = Q_n(z) \boxtimes_n T_{n,\ell}^{(a,b)}(z),
\end{equation}
where
\begin{align}
T_{n,\ell}^{(a,b)}(z)
&=
\sum_{\substack{j\in \{0,\ldots, n\}\\ j\in \mathcal{J}_\ell}}
(-1)^{n-j}\binom nj
\frac{1}{n^{\ell b}} \frac{j!}{(j-b)!} \frac{(j+\Delta)!}{(j-b+\Delta)!}\ldots \frac{(j+(\ell-1)\Delta)!}{(j-b + (\ell-1)\Delta)!} \cdot z^{j}\label{eq:q_n_m_definition}\\
&=z^{-\ell\Delta}\mathcal{A}_{a,b}^\ell \left(\sum_{j=0}^{n}(-1)^{n-j}\binom{n}{j}z^j\right)=z^{-\ell\Delta}\mathcal{A}_{a,b}^\ell ((z-1)^n)\notag.
\end{align}
Let us mention some particular cases of $T_{n,\ell}^{(a,b)}(z)$. If $a=b\in \N$, then $\Delta = 0$ and
$$
T_{n,\ell}^{(a,a)}(z) = n^{-\ell a} \sum_{j=a}^n (-1)^{n-j}\binom nj (j(j-1) \ldots (j-a+1))^\ell z^j.
$$
In particular, if $a = b = 1$, then $T_{n,\ell}^{(1,1)}(z) = \sum_{j=1}^n (-1)^{n-j}\binom nj j^\ell z^j$. These polynomials are a special case of Sidi polynomials appearing in~\cite[Theorems~4.2, 4.3]{sidi_numerical_quad_nonlinear_seq} and~\cite[Theorem~3.1]{sidi_num_quad_infinite_range}. For results on the asymptotic zero distribution of Sidi and related polynomials, see~\cite[Theorem~1.3]{lubinski_sidi_strong_asympt}, \cite[Theorem~2]{lubinski_stahl_some_explicit} and~\cite{lubinski_sidi_composite_biorthogonal}.

Recalling formula~\eqref{eq:diff_as_finite_free_mult_convolution} we shall first study the roots of $T_{n,\ell}^{(a,b)}$. The proof of the results below will be given in Section~\ref{sec:proof_repeated}.

\begin{lemma}\label{lem:q_n_m_roots_real}
For every $a,b\in \N_0$, $\ell\in \N$ and $n \in \N_0\cap \mathcal{I}_{\ell}$, the polynomial $T_{n,\ell}^{(a,b)}(z)$ does not vanish identically and all of its roots are nonnegative. If $\Delta \geq 0$, then $z=0$ is a root of $T_{n,\ell}^{(a,b)}(z)$ of multiplicity $b$. If $\Delta \leq 0$, then  $z=0$ is a root of $T_{n,\ell}^{(a,b)}(z)$ of multiplicity $a-\ell \Delta$.
\end{lemma}

\begin{theorem}\label{thm:q_n_m_measures_converge}
If $n,\ell\to\infty$ such that $\ell/n \to \kappa >0$ and $1+\Delta \kappa>0$, then $\lsem T_{n,\ell}^{(a,b)}(z)\rsem_n$ converges weakly to a certain compactly supported probability measure $\nu_{a,b;\kappa}$ on $[0,\infty)$ with $\nu_{a,b;\kappa}(\{0\})=\max(0,-\Delta \kappa)$.
\begin{itemize}
\item If $\Delta\ge 0$, then $\nu_{a,b;\kappa}$ is $\boxtimes$-infinitely divisible.
\item For $\Delta=0$ (that is, if $a=b$), $\nu_{a,a;\kappa}$ is $\boxtimes$-free-Poisson distribution.
\item If $\Delta<0$, then $\nu_{a,b;\kappa}=(-\Delta\kappa\delta_0+(1+\Delta\kappa)\delta_1)^{\boxtimes {\frac b {-\Delta}}}$ is a free multiplicative self-convolution (with possibly fractional order $\frac b {-\Delta}\geq 1$) of the Bernoulli measure $-\Delta\kappa\delta_0+(1+\Delta\kappa)\delta_1$.
\end{itemize}
In particular, $\nu_{0,1;\kappa}=\kappa\delta_0+(1-\kappa)\delta_{1}$ for $a=0$, $b=1$ and $\kappa\in (0,1]$. If $a=b\in \N$ and $\kappa>0$ is arbitrary, then $\nu_{a,a;\kappa}$ is a probability measure on $[0,1]$ with the following
\begin{align}
&\text{Cauchy transform:}&   &G_{a,a;\kappa}(t) = \frac{1} {t \left(1 + \frac 1 {a\kappa} W_0\left(-\frac {a\kappa} {t \eee^{a\kappa}}\right)\right)},  &\quad& t\in  \C \backslash [0,a\kappa \eee^{1-a\kappa}],
\label{eq:theo:repeated_action_stieltjes}\\
&\text{Lebesgue density:}&      &p_{a,a;\kappa}(x) = -\frac{1}{\pi x} \Im\frac{1}{1+ \frac{1}{a\kappa} W_0\left(-\frac {a\kappa} {x \eee^{a\kappa}}+\ii 0\right)},  &\quad& x\in  (0,a\kappa \eee^{1-a\kappa}),\label{eq:theo:repeated_action_density}\\
& & &\text{of the total mass }\min(a\kappa,1) & &\notag \\
& & & \text{plus an atom of the mass } \max(1-a\kappa,0) \text{ at } & & x=1.\notag
\end{align}
Here, $W_0$ denotes the principal branch of the Lambert $W$-function satisfying $W_0(z) \eee^{W_0(z)} = z$, see~\cite[Chapter~1]{mezo_book_lambert_function}, and $W_0(y+\ii 0) := \lim_{\eps \to 0+} W_0(y+\ii \eps)$ for $y<-1/\eee$.
\end{theorem}

Formula~\eqref{eq:diff_as_finite_free_mult_convolution} in conjunction with Theorems~\ref{theo:finite_free_mult_conv} and~\ref{thm:q_n_m_measures_converge} imply the following.
\begin{theorem}\label{theo:repeated_diff_general}
Fix $a,b\in\mathbb{N}_0$. For every $n$ let  $Q_n(z)$ be a polynomials of degree $\leq n$ with only nonnegative roots such that $\lsem Q_n \rsem_n \toweak \nu$ for some probability measure $\nu$ on $[0,+\infty)$. Assume that  $n,\ell\to\infty$ such that $\ell/n\to \kappa>0$ and $1+(a-b)\kappa>0$. Then,
$$
\left\lsem z^{-\ell(a-b)}\left(z^a (\dd / \dd z)^b\right)^\ell Q_n(z)\right\rsem_{n} \toweak \nu \boxtimes \nu_{a,b; \kappa}.
$$
\end{theorem}

Taking $a=0$, $b=1$ we recover the limit distribution of repeatedly differentiated real-rooted polynomials which has been obtained in \cite{arizmendi,HK21,Steiner21} in terms of the free additive self-convolution.

\begin{corollary}\label{cor:repeated_diff}
For every $n\in\N$ let $Q_n(x)$ be a polynomial of degree $\leq n$. Suppose that $\lsem Q_n\rsem_n \toweak \nu$ for some probability measure $\nu$ on $[A,\infty)$ and that all roots of all $Q_n$'s  are bounded below by $A$, where $A\in \R$.   If $n,\ell\to\infty$ such that $\ell/n\to \kappa\in (0,1)$, then
$\lsem (\tfrac{\dd}{\dd z})^{\ell} Q_n\rsem_{n-\ell}\toweak \left(\nu^{\boxplus\frac{1}{1-\kappa}}\right)\big(\tfrac{\cdot}{1-\kappa}\big).$
\end{corollary}
\begin{proof}
Since differentiation commutes with shifts along the real line and so does the map $\nu\mapsto \left(\nu^{\boxplus\frac{1}{1-\kappa}}\right)\big(\tfrac{\cdot}{1-\kappa}\big)$, we can assume that $A>0$. From Theorem~\ref{theo:repeated_diff_general} applied with $a=0, b=1$ we conclude that
\begin{equation}\label{eq:rep_diff_corollary_proof1}
\left\lsem z^{\ell}(\tfrac{\dd}{\dd z})^{\ell} Q_n\right\rsem_{n}\toweak \nu\boxtimes \nu_{0,1;\kappa}=\nu\boxtimes (\kappa\delta_0+(1-\kappa)\delta_1).
\end{equation}
Exactly as in the proof of Theorem~3.7 in~\cite{arizmendi}, we use a known connection between the additive free self-convolution and multiplicative free convolution, see~\cite[Equation (14.13)]{nica_speicher_book}. For any distribution $\nu$ on $\mathbb R$ and any $\kappa\in[0,1)$ it holds
\begin{equation}\label{eq:free_additive_convolution_flow_def}
\nu\boxtimes\big(\kappa\delta_0+(1-\kappa)\delta_{\frac{1}{1-\kappa}}\big)=\kappa\delta_0+(1-\kappa)\nu^{\boxplus \frac{1}{1-\kappa}}.
\end{equation}
Thus,~\eqref{eq:rep_diff_corollary_proof1}  implies
\begin{equation}\label{eq:rep_diff_corollary_proof2}
\left\lsem z^{\ell}(\tfrac{\dd}{\dd z})^{\ell} Q_n\right\rsem_{n}\toweak \kappa\delta_0+(1-\kappa)\nu^{\boxplus \frac{1}{1-\kappa}}\left(\frac{\cdot}{1-\kappa}\right).
\end{equation}
Observe that
$$
\lsem (\tfrac{\dd}{\dd z})^{\ell} Q_n\rsem_{n-\ell}=\frac{n}{n-\ell}\left(\left\lsem z^{\ell}(\tfrac{\dd}{\dd z})^{\ell} Q_n\right\rsem_{n}-\frac{\ell}{n}\delta_0\right).
$$
The claim now follows from~\eqref{eq:rep_diff_corollary_proof2} by sending $n,\ell\to\infty$.
\end{proof}

\begin{remark}
As $\kappa\uparrow 1$ and assuming that $\mu$ is centered, we have the free Central Limit Theorem $\left(\mu^{\boxplus\frac{1}{1-\kappa}}\right)\big(\tfrac{\cdot}{\sqrt{1-\kappa}}\big)\toweak \mathsf{sc}_1$ weakly, where $\mathsf{sc}_1$ is Wigner's semicircle distribution, that is the limiting zero distribution of Hermite polynomials $\mathrm{He}_n$ as $n\to\infty$. It has recently been shown in \cite[Theorem 2.8]{COR24} that the polynomials $(\tfrac{\dd}{\dd z})^{n-d} Q_n$ converge to $\mathrm{He}_d$ (after rescaling and for $n\to\infty, d\in\N$ fixed). Similar to related results, the proof of~\cite{COR24} relies on free cumulants. However, a close inspection of \cite[Lemma 2.2]{COR24} reveals that quotients of large degree coefficients determine the first few moments, just as derivatives of the profile $g$ near $\alpha=1$ do: Moments can be represented by derivatives of the moment generating function and hence in terms of $g$, see~\eqref{eq:summary_psi} below, for instance the first moment is given by $m_1=\frac 1 {tg''(g'^{\leftarrow}(\log t))}\big|_{t=0}$.
\end{remark}

Motivated by~\cite[Section~6]{bogvad_etal} and~\cite[Section~4.2]{HK21}, and the Rodrigues formula for the Legendre polynomials, we analyze the zeros of the repeated derivatives of the polynomial $z^m (1-z)^m$.
\begin{corollary}
If  $m,\ell\to\infty$ such that $\ell/(2m)\to \kappa \in (0,1)$, then
$$
\lsem   (\tfrac{\dd}{\dd z})^{\ell} \left(z^{m} (1-z)^{m}\right) \rsem_{2m-\ell} \overset{w}{\underset{m,\ell\to\infty}\longrightarrow}
\frac{1}{1-\kappa} \cdot  \left(\left(\frac 12 \delta_0 + \frac 12 \delta_1\right)\boxtimes (\kappa \delta_0 + (1-\kappa)\delta_1) - \kappa \delta_0\right).
$$
\end{corollary}
\begin{proof}
Theorem~\ref{theo:repeated_diff_general} with $n=2m$, $Q_n(z) = z^m (1-z)^m$, $a=0$, $b=1$ yields
$$
\lsem z^{\ell}  (\tfrac{\dd}{\dd z})^{\ell} \left(z^{m} (1-z)^{m}\right) \rsem_{2m}\toweak \left(\frac 12 \delta_0 + \frac 12 \delta_1\right)\boxtimes (\kappa \delta_0 + (1-\kappa)\delta_1).
$$
This measure has an atom of size $\max (\frac 12, \kappa)$ at $0$ by~\cite[Lemma~6.9]{bercovici_voiculescu}. Next we observe that
$$
\lsem   (\tfrac{\dd}{\dd z})^{\ell} \left(z^{m} (1-z)^{m}\right) \rsem_{2m-\ell}
=
\frac {2m}{2m-\ell} \left(\lsem z^{\ell}  (\tfrac{\dd}{\dd z})^{\ell} \left(z^{m} (1-z)^{m}\right) \rsem_{2m} -\frac \ell{2m} \delta_0\right).
$$
The proof is completed by letting $m,\ell\to\infty$.
\end{proof}

\begin{remark}
\citet[Section~2]{lubinski_sidi_composite_biorthogonal} consider several families of polynomials admitting a Rodrigues-type representation of the form $(\tfrac{\dd} {\dd z})^\ell Q_n(z)$ or $(z\tfrac{\dd} {\dd z})^\ell Q_n(z)$, where the limiting zero distribution of $Q_n$ is easy to determine. For all these examples, the limiting distribution of zeros can be identified using Theorem~\ref{theo:repeated_diff_general}.
\end{remark}

An analogue of Theorem \ref{theo:repeated_diff_general} for random polynomials with \emph{complex} isotropic limit distribution of zeros as in~\cite{KZ14} has been obtained in~\cite{COR23,feng_yao,diff-paper}.  Repeated differentiation of trigonometric polynomials has been studied in~\cite{kabluchko_rep_diff_free_poi}.

\subsection{Repeated differentiation of polynomials with equally spaced roots}
Consider the polynomials $p_n(z):=z(z-\frac 1 n)\cdots(z-\frac{n-1}n)$, $n\in \N$,  with zeros uniformly spaced  in $[0,1]$. Figure~\ref{fig:zeroes_repeated_diff_unif_interval} shows zeros of $p_n, p_n',p_n'',\ldots, p_n^{(n-1)}$ for $n=50$. One may ask for the limiting  distribution of these zeros as $n\to\infty$. Despite the simple nature of the problem, the solution turns out to be non-trivial and is new to the best of our knowledge.
\begin{figure}[t]
	\centering
	\includegraphics[width=0.6\columnwidth]{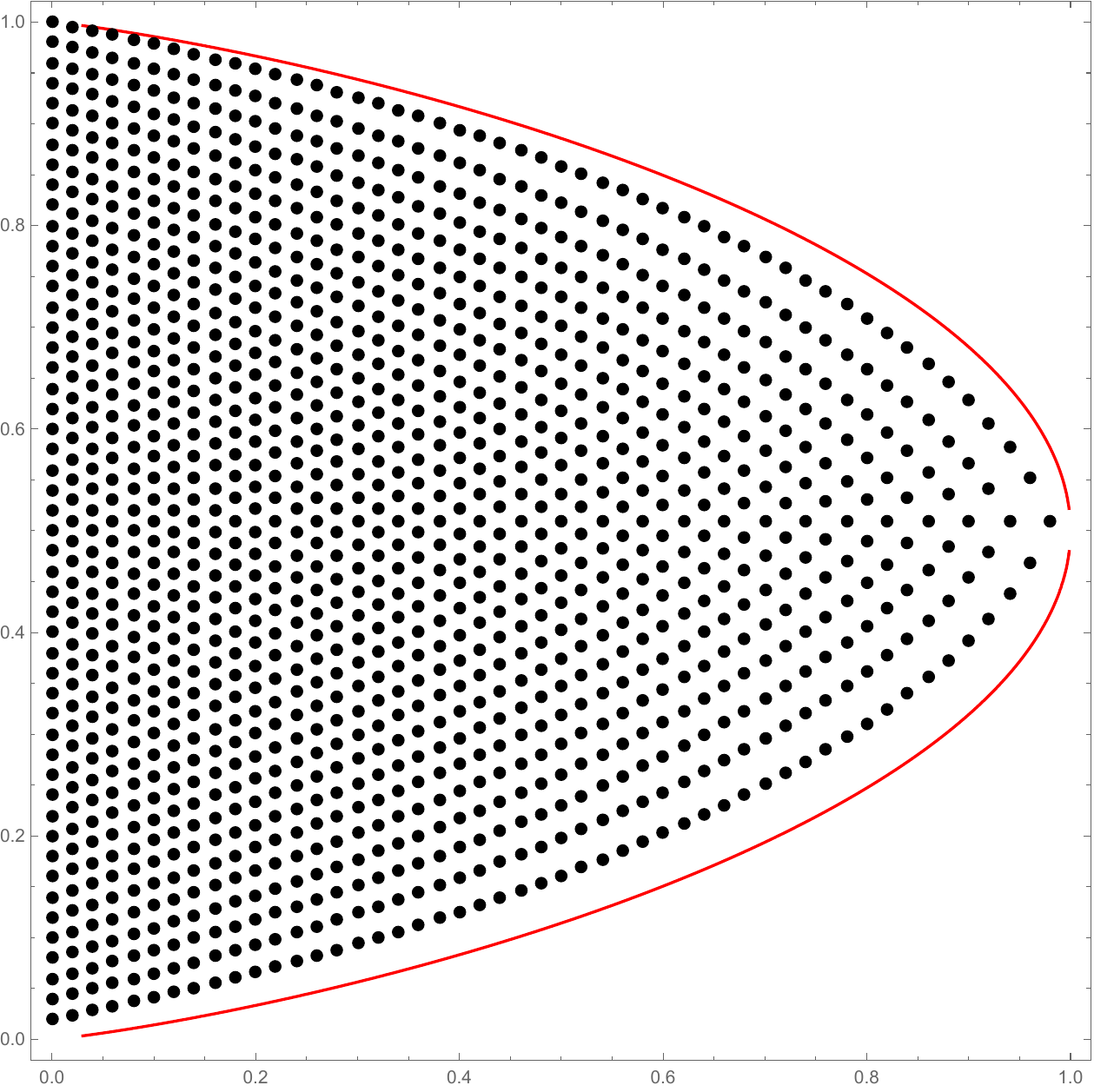}
	\caption{Roots of repeated derivatives of $p_n(z) =  \prod_{j=0}^{n-1} \left(z - \frac j n\right)$ with $n=50$. The black dots have coordinates $(\frac \ell n, y_{i,\ell})$, where $\ell\in \{0,1,\ldots, n\}$ is the order of differentiation and $y_{1,\ell},\ldots, y_{n-\ell, \ell}\in [0,1]$ are the roots of $p_n^{(\ell)}$. The red curves show the support of the probability measure $\unif_{[0,1]} \boxtimes {\nu}_{0,1;\kappa}$ as a function of $\kappa\in (0,1)$.}
\label{fig:zeroes_repeated_diff_unif_interval}
\end{figure}
According to Theorem~\ref{theo:repeated_diff_general} with $a=0$, $b=1$, for every fixed $\kappa\in (0,1)$,
$$
\lsem (\tfrac{\dd}{\dd z})^{\lfloor \kappa n\rfloor} p_n(z)\rsem_n\toweak \unif_{[0,1]} \boxtimes {\nu}_{0,1;\kappa},
$$
where $\unif_{[0,1]}$ is the uniform distribution on $[0,1]$.
It will be convenient to map the zeros from the interval $[0,1]$ to the interval $[-1,1]$. Namely, if $\xi_{\kappa}$ is a random variable with the distribution $\unif_{[0,1]}\boxtimes {\nu}_{0,1;\kappa}$, then
\begin{equation}\label{eq:stieltjes_polys_repeated_diff}
\lsem (\tfrac{\dd}{\dd z})^{\lfloor \kappa n\rfloor} p_n((z+1)/2)\rsem_n\toweak \mu_{\kappa},
\end{equation}
where $\mu_{\kappa}$ is the distribution of $2\xi_{\kappa}-1$. In the next result we characterize the limiting measure $\mu_{\kappa}$. A proof will be given in Section~\ref{sec:proof_repeated}.
\begin{proposition}\label{eq:prop_rep_diff_stiel}
For $\kappa \in (0,1)$ define $Y_{\kappa}(t)=\coth(t)-\kappa/t$ and let $z_{\kappa}$ be the unique positive solution of
\begin{equation}\label{eq:sinh_equation}
\sinh(z)=z/\sqrt{\kappa}.
\end{equation}
The measure $\mu_{\kappa}$ in~\eqref{eq:stieltjes_polys_repeated_diff} is supported by $[Y_{\kappa}(-z_{\kappa}),Y_{\kappa}(z_{\kappa})]$ and has the Cauchy transform
\begin{align*}
G_{\mu_{\kappa}}(z)=Y^{\leftarrow}_{\kappa}(z),\quad z\in \R\backslash [Y_{\kappa}(-z_{\kappa}),Y_{\kappa}(z_{\kappa})].
\end{align*}
\end{proposition}

Functions similar to $Y^{\leftarrow}_{\kappa}(z)$ arose for example in~\cite[Section~2.2.6]{mezoe_keady} (the inverse Langevin function), \cite[Section~7.2.2]{mezo_book_lambert_function}, or in the context of the Curie-Weiss model~\cite[Equation~(2.12)]{friedli_velenik_book}. It seems that $Y^{\leftarrow}_{\kappa}$ can be expressed through the function $\zeta$  studied in~\cite[Section~4]{kabluchko_rep_diff_free_poi}, but we shall not do it here.

\section{Proof of Theorems~\ref{theo:exp_profile_implies_zeros} and~\ref{theo:zeros_imply_exp_profile}}\label{sec:proof_main1}
The goal of this section is to prove Theorem \ref{theo:exp_profile_implies_zeros} and Theorem \ref{theo:zeros_imply_exp_profile}. To this end, we shall view an arbitrary polynomial $P_n(x) = \sum_{k=0}^n a_{k:n} x^k$ with nonpositive roots $-\lambda_{k:n}$ and after dividing by the positive constant $P_n(1)$ as the generating function of the sum of mutually independent Bernoulli random variables with parameters $1/(1+\lambda_{k:n})$. Hence,  asymptotic behaviour of such sums becomes crucially relevant and thus we begin with the following result of independent interest.

\subsection{Large deviations for sums of independent Bernoulli random variables}\label{subsec:large_deviations}
Let
$$
p_{k:n}\in [0,1],\quad 1\leq k\leq n,\quad n\in \N,
$$
be a collection of probabilities and $(B_{k:n})_{n\in \N, 1\leq k\leq n}$ be a triangular array of Bernoulli random variables defined on some underlying probability space $(\Omega,\mathcal{F},\mathbb{P})$ and such that:
\begin{itemize}
\item for each $n\in\mathbb{N}$, the random variables  $B_{1:n},\ldots,B_{n:n}$ are mutually independent;
\item for each $n\in\mathbb{N}$ and $1\leq k\leq n$, $\mathbb{P}[B_{k:n}=1]=p_{k:n}=1-\mathbb{P}[B_{k:n}=0]$.
\end{itemize}
Consider the random variable $S_{n}:=\sum_{k=1}^{n}B_{k:n}$. It follows from the Lindeberg central limit theorem that $(S_n-\E S_n)/\sqrt{\Var S_n}$ converges weakly to the standard normal distribution assuming only that $\Var S_n\to\infty$ as $n\to\infty$. This result has multiple applications in combinatorics; see~\cite{canfield_asymptotic_normality_survey,harper,pitman_probab_bounds}. The next theorem is a precise large deviation principle for $S_n$.

\begin{theorem}\label{thm:main_probabilistic_direct}
Assume that, as $n\to\infty$, probability measures $M_n:=\frac{1}{n}\sum_{k=1}^{n}\delta_{p_{k:n}}$ converge weakly to some probability measure $M_{\infty}$ on $[0,1]$. Let
$$
\underline{m}:=M_{\infty}(\{1\})\leq 1-M_{\infty}(\{0\})=:\overline{m}.
$$
Suppose that $\underline{m} < \overline{m}$ and $\eps>0$ is sufficiently small.
Then, there exists an infinitely differentiable strictly concave function $g:(\underline{m},\overline{m})\to(-\infty,0]$ such that
\begin{equation}\label{eq:uniform_convergence_of_profiles}
\sup_{(\underline{m}+\varepsilon)n\leq k\leq (\overline{m}-\varepsilon)n}\left|\frac{1}{n}\log\mathbb{P}[S_n=k]-g\left(\frac {k}{n}\right)\right|\ton 0.
\end{equation}
Moreover,
\begin{equation}\label{eq:divergence_outside}
\sup_{0\leq k\leq (\underline{m}-\varepsilon) n}\frac{1}{n}\log\mathbb{P}[S_n=k]\ton
-\infty
\quad\text{and}\quad
\sup_{(\overline{m}+\varepsilon) n\leq k\leq n}\frac{1}{n}\log\mathbb{P}[S_n=k]\ton -\infty.
\end{equation}
The function $-g$ is the Legendre transform of a strictly convex function $u\mapsto \Psi(\eee^u)$, $u\in\mathbb{R}$, where
\begin{equation}\label{eq:psi_def}
\Psi(t):=\int_{[0,1]}\log (1-y+yt)M_{\infty}({\rm d}y),\quad t>0.
\end{equation}
More precisely,
\begin{equation}\label{eq:phi_def}
g(\alpha)=-\sup_{u\in\mathbb{R}}(\alpha u-\Psi(\eee^u))=\inf_{u\in\mathbb{R}}(\Psi(\eee^u)-u\alpha)=\Psi(\Phi^{\leftarrow}(\alpha))-\alpha\log\Phi^{\leftarrow}(\alpha),\quad \alpha\in (\underline{m},\overline{m}),
\end{equation}
where
\begin{equation}\label{eq:psi_func_def}
\Phi(t):=t\Psi^{\prime}(t)=t\int_{[0,\,1]} \frac{y M_{\infty}({\rm d}y)}{1-y+yt},\quad t>0,
\end{equation}
and $\Phi^{\leftarrow}$ denotes the inverse function of $\Phi$. Furthermore, the following identity holds for all $\alpha\in (\underline{m},\overline{m})$,
\begin{equation}\label{eq:phi_Phi_inversion}
\eee^{-g'(\alpha)}= \Phi^{\leftarrow}(\alpha),
\end{equation}
the map $\alpha\mapsto g(\alpha)+\alpha\log\alpha$ is strictly concave and
\begin{equation}\label{eq:phi_derivatives}
\lim_{\alpha\to\underline{m}+}g'(\alpha)=+\infty\quad\text{and}\quad\lim_{\alpha\to \overline{m}-}g'(\alpha)=-\infty.
\end{equation}
Writing $p_* := \int_{[0,1]} y M_\infty(\dd y)=\Phi(1)$ we have $\underline{m} < p_* < \overline{m}$ and $g(p_*) = 0$.
Also,
\begin{equation}\label{eq:P_s_n_k_quotients}
\sup_{(\underline{m}+\varepsilon)n\leq k\leq (\overline{m}-\varepsilon)n}\left|\frac{\mathbb{P}[S_n=k+1]}{\mathbb{P}[S_n=k]} -\eee^{g'\left(\frac {k}{n}\right)}\right|\ton 0.
\end{equation}
Moreover, writing $\underline{r}_n := \sum_{k=0}^n \1[p_{k:n}=1]$ and $\overline{r}_n:= n-\sum_{k=0}^n \1[p_{k:n}=0]$, so that the support of $S_n$ is equal to $\{\underline{r}_n, \ldots, \overline{r}_n\}$, we have
\begin{equation}\label{eq:P_s_n_k_quotients_infinity}
\inf_{\underline{r}_n\leq k\leq (\underline{m}-\varepsilon) n} \frac{\mathbb{P}[S_n=k+1]}{\mathbb{P}[S_n=k]}
\ton +\infty
\quad\text{and}\quad
\inf_{(\overline{m}+\varepsilon)n\leq k < \overline{r}_n} \frac{\mathbb{P}[S_n=k]}{\mathbb{P}[S_n=k+1]}
\ton +\infty.
\end{equation}
Finally, if $\underline{m}=\overline{m}$, then~\eqref{eq:divergence_outside} and~\eqref{eq:P_s_n_k_quotients_infinity} hold true, while~\eqref{eq:uniform_convergence_of_profiles} and~\eqref{eq:P_s_n_k_quotients} become void statements.
\end{theorem}

The converse statement is also true in the following form.

\begin{theorem}\label{thm:main_probabilistic_converse}
Assume that there exists a nonempty interval $(\underline{m},\overline{m})\subseteq [0,1]$ and a function $g:(\underline{m},\overline{m})\to (-\infty,0]$ such that
\begin{equation}\label{eq:main_probabilistic_converse_assumption}
\lim_{n\to\infty}\frac{1}{n}\log \mathbb{P}[S_n=\lfloor \alpha n\rfloor]=
\begin{cases}
g(\alpha),&\text{if }\alpha\in (\underline{m},\overline{m}),\\
-\infty,&\text{if }\alpha\in [0,\underline{m})\cup (\overline{m},1].
\end{cases}
\end{equation}
Then, as $n\to\infty$, probability measures $M_n:=\frac{1}{n}\sum_{k=1}^{n}\delta_{p_{k:n}}$ converge weakly to a probability measure $M_{\infty}$ on $[0,1]$ such that
$$
M_{\infty}(\{1\})=\underline{m}\quad\text{and}\quad M_{\infty}(\{0\})=1-\overline{m}.
$$
Furthermore,~\eqref{eq:uniform_convergence_of_profiles}-\eqref{eq:P_s_n_k_quotients_infinity} hold true.
\end{theorem}

Note that we \emph{do} require $\underline{m}<\overline{m}$ in the above theorem.  The example when $S_n = n-1$ a.s.\ for even $n$ and $S_n =1$ a.s.\  for odd $n$ shows that, in general,  $\frac 1n \log \mathbb{P}[S_n=\lfloor \alpha n\rfloor]\to -\infty$ for all $\alpha\in [0,1]$   does not imply that $M_n$ converges weakly.
The next result complements Theorem~\ref{thm:main_probabilistic_converse}.
\begin{proposition}
Suppose that $\lim_{n\to\infty}\frac{1}{n}\log \mathbb{P}[S_n=\lfloor \alpha n\rfloor] = -\infty$ for all $\alpha$ in some dense subset of $[0,1]$. Then, all subsequential weak limits of $M_n$ are concentrated on $\{0,1\}$.
\end{proposition}
\begin{proof}
If some subsequential weak limit of $M_n$ were not concentrated on $\{0,1\}$, then Theorem~\ref{thm:main_probabilistic_direct} would imply that $\frac{1}{n}\log \mathbb{P}[S_n=\lfloor \alpha n\rfloor]$ has a finite subsequential limit for all $\alpha$ in a nonempty interval. This is a contradiction.
\end{proof}

\begin{proof}[Proof of Theorem~\ref{thm:main_probabilistic_direct}]
We first consider the case $\underline{m}<\overline{m}$. Put $\varphi_n(t):=\mathbb{E}[t^{S_n}]=\prod_{k=1}^{n}(1-p_{k:n}+p_{k:n}t)$, $t\in\mathbb{C}$, and introduce a new probability measure $\widetilde{\mathbb{P}}_{\theta}$ (with the corresponding expectation $\widetilde{\mathbb{E}}_{\theta}$) by an exponential change of measure via the identity
\begin{equation}\label{eq:exponential_tilting}
\widetilde{\mathbb{E}}_{\theta}[f(B_{1:n},\ldots,B_{n:n})]=\frac{\mathbb{E}[f(B_{1:n},\ldots,B_{n:n})\theta^{S_n}]}{\varphi_n(\theta)},\quad \theta>0,
\end{equation}
for any bounded measurable $f:\mathbb{R}^n\to \mathbb{R}$. Note that under $\widetilde{\mathbb{P}}_{\theta}$ the variables $B_{1:n},\ldots,B_{n:n}$ are  still mutually independent and, by putting $f(x_1,\ldots,x_n)=\1[x_j=1]$, we conclude that
\begin{equation}\label{eq:exponential_tilting_probabilities}
\widetilde{\mathbb{P}}_{\theta}[B_{j:n}=1]=\frac{\theta p_{j:n}}{1-p_{j:n}+p_{j:n}\theta},\quad j=1,\ldots,n.
\end{equation}
Substitution $f(x_1,\ldots,x_n)=\1[x_1+\cdots+x_n=k]$ into~\eqref{eq:exponential_tilting} results in
\begin{equation}\label{eq:exponential_tilting_probabilities1}
\mathbb{P}[S_n=k]=\varphi_n(\theta)\theta^{-k}\widetilde{\mathbb{P}}_{\theta}[S_n=k]
\end{equation}
and thereupon
\begin{equation}\label{eq:p_vs_p_theta}
\frac{1}{n}\log \mathbb{P}[S_n=k]=-\frac{k}{n}\log \theta +\frac{\log \varphi_n(\theta)}{n}+\frac{1}{n}\log\widetilde{\mathbb{P}}_{\theta}[S_n=k].
\end{equation}

We are going to pick $\theta_{\ast}=\theta_{\ast}(k,n)$ such that $\widetilde{\mathbb{E}}_{\theta_{\ast}}[S_n]= k$ which, as we shall see, results in a subexponential decay of $\widetilde{\mathbb{P}}_{\theta_{\ast}}[S_n=k]$.  By choosing $f(x_1,\ldots,x_n)=x_1+\cdots+x_n$ in~\eqref{eq:exponential_tilting} we obtain the equation
\begin{equation*}
\widetilde{\mathbb{E}}_{\theta_{\ast}}[S_n]=\frac{\mathbb{E}[S_n\theta_{\ast}^{S_n}]}{\varphi_n(\theta_{\ast})}=k~\Longleftrightarrow~\frac{\theta_{\ast} \varphi^{\prime}_n(\theta_{\ast})}{n \varphi_n(\theta_{\ast})}=\frac{k}{n},
\end{equation*}
which in terms of the measure $M_n$ takes the form $\Phi_n(\theta_{\ast})=k/n$, where
\begin{equation}\label{eq:theta_choice}
\Phi_n(t):=t\int_{[0,\,1]}\frac{yM_n({\rm d}y)}{1-y+yt},\quad t>0.
\end{equation}
We shall first argue that the equation $\Phi_n(\theta_{\ast})=k/n$ has a unique solution for all $k\in\mathbb{N}$ such that $(\underline{m}+\varepsilon)n\leq k\leq (\overline{m}-\varepsilon)n$, where $n$ is sufficiently large. Note that the assumption that $M_n$ converges weakly to $M_{\infty}$ implies that $\Phi_n(t)$ defined in~\eqref{eq:theta_choice} converges to $\Phi(t)$ for every fixed $t>0$ because $y\mapsto ty/(1-y+yt)$ is bounded and continuous on $[0,1]$. The function $\Phi_n$ is strictly increasing and continuous on $(0,+\infty)$ for each $n\in\mathbb{N}$ and the same holds for the limit function $\Phi$. Thus, by Dini's theorem, the convergence is locally uniform on $(0,\infty)$, that is
\begin{equation}\label{eq:local_uniform1}
\lim_{n\to\infty}\sup_{t\in K}|\Phi_n(t)-\Phi(t)|=0,
\end{equation}
for every compact set $K\subset (0,+\infty)$. The function $\Phi$ satisfies
\begin{equation}\label{eq:M_infty_range}
\lim_{t\to 0+}\Phi(t)=M_{\infty}(\{1\})=\underline{m},\quad \lim_{t\to +\infty}\Phi(t)=M_{\infty}((0,1])=\overline{m},
\end{equation}
and therefore possesses a strictly increasing continuous inverse $\Phi^{\leftarrow}:(\underline{m},\overline{m})\to (0,+\infty)$. By strict monotonicity and continuity of all our functions, we also have
\begin{equation}\label{eq:local_uniform2}
\lim_{n\to\infty}\sup_{z\in K_1}|\Phi_n^{\leftarrow}(z)-\Phi^{\leftarrow}(z)|=0,
\end{equation}
for every compact subset $K_1\subset (\underline{m},\overline{m})$. Therefore, for $(\underline{m}+\varepsilon) n\leq k\leq (\overline{m}-\varepsilon)n$ and $n$ large enough, $\theta_{\ast}(k,n)=\Phi_n^{\leftarrow}(k/n)$ is well-defined,
\begin{equation}\label{eq:local_uniform3}
\lim_{n\to\infty}\sup_{(\underline{m}+\varepsilon) n\leq k\leq (\overline{m}-\varepsilon)n}|\theta_{\ast}(k,n)-\Phi^{\leftarrow}(k/n)|=0
\end{equation}
and
\begin{equation}\label{eq:local_uniform4}
\lim_{n\to\infty}\sup_{(\underline{m}+\varepsilon) n\leq k\leq (\overline{m}-\varepsilon)n}|\log \theta_{\ast}(k,n)-\log \Phi^{\leftarrow}(k/n)|=0.
\end{equation}
By another appeal to Dini's theorem, we conclude that the point-wise convergence
\begin{equation}\label{eq:log_potential_converge}
\frac{1}{n}\log \varphi_n(t)=\int_{[0,1]}\log (1-y+yt)M_n({\rm d}y)\ton\Psi(t),\quad t>0,
\end{equation}
can be improved to the locally uniform convergence on $(0,\infty)$. Together with~\eqref{eq:local_uniform3} this yields
$$
\lim_{n\to\infty}\sup_{(\underline{m}+\varepsilon) n\leq k\leq (\overline{m}-\varepsilon)n}\left|\frac{1}{n}\log \varphi_n(\theta_{\ast}(k,n))-\Psi(\Phi^{\leftarrow}(k/n))\right|=0.
$$
To finish the proof of~\eqref{eq:uniform_convergence_of_profiles} it remains to show that
\begin{equation*}
\sup_{(\underline{m}+\varepsilon) n\leq k\leq (\overline{m}-\varepsilon)n}\left|\frac{\log\widetilde{\mathbb{P}}_{\theta_{\ast}(n,k)}[S_n=k]}{n}\right|\ton 0.
\end{equation*}
We shall actually show more, namely,
\begin{equation}\label{eq:llt_uniform}
\sup_{(\underline{m}+\varepsilon) n\leq k\leq (\overline{m}-\varepsilon)n}\left|\frac{\log\widetilde{\mathbb{P}}_{\theta_{\ast}(n,k)}[S_n=k]}{c_n}\right|\ton 0,
\end{equation}
for any sequence $(c_n)$ such that $(\log n) / c_n\to 0$, as $n\to\infty$. To this end, we first put
\begin{equation}\label{eq:variance}
\sigma_n^2:=\widetilde{{\rm Var}}_{\theta_{\ast}}[S_n]=\sum_{j=1}^{n}\frac{\theta_{\ast} p_{j:n}(1-p_{j:n})}{(1-p_{j:n}+\theta_{\ast} p_{j:n})^2}=n\int_{[0,1]}\frac{\theta_{\ast}y(1-y)}{(1-y+y\theta_{\ast})^2}M_n({\rm d}y).
\end{equation}
Take any sequence $k=k(n)$ such that $(\underline{m}+\varepsilon) n\leq k\leq (\overline{m}-\varepsilon)n$. In view of~\eqref{eq:local_uniform3} the quantity $\theta_{\ast}=\theta_{\ast}(k,n)$ stays bounded away from zero and infinity as $n\to\infty$. Also note that $M_{\infty}(\{0\})+M_{\infty}(\{1\})<1$ implies $\int_{[0,1]}y(1-y)M_{\infty}({\rm d}y)>0$, whence $\sigma_n^2>0$. Thus,
\begin{equation}\label{eq:variance_linear}
0<\liminf_{n\to\infty}n^{-1}\sigma_n^2<\limsup_{n\to\infty}n^{-1}\sigma_n^2<\infty.
\end{equation}
According to Theorem 2 in~\cite{bender}, which is a local limit theorem for sums of independent indicators, applied to $P_n(x):=\widetilde{\mathbb{E}}_{\theta_{\ast}(n,k)}[x^{S_{n}}]=\varphi_{n}(\theta_{\ast}(n,k) x)/\varphi_{n}(\theta_{\ast}(n,k))$, we conclude that the divergence of the variances~\eqref{eq:variance_linear} implies
\begin{equation}\label{eq:llt_changed_measure}
\sup_{x\in \mathbb{R}}\left|\sigma_{n}\widetilde{\mathbb{P}}_{\theta_{\ast}(n,k)}[S_{n}=\lfloor \sigma_{n} x\rfloor+k]-\frac{1}{\sqrt{2\pi}}\eee^{-x^2/2}\right|\to 0,\quad n\to\infty.
\end{equation}
In particular, setting $x=0$ yields,
\begin{align}\label{eq:llt_changed_measure2}
\widetilde{\mathbb{P}}_{\theta_{\ast}(n,k)}[S_{n}=k]~\sim~\frac{1}{\sigma_{n}\sqrt{2\pi}},\quad n\to\infty.
\end{align}
Now~\eqref{eq:llt_uniform} easily follows by contradiction. Assuming the existence of integer sequences $(n_m)$ and $(k_m)$ diverging to infinity such that
\begin{equation*}
\underline{m}+\varepsilon \leq \frac{k_m}{n_m}\leq \overline{m}-\varepsilon\quad\text{and}\quad \limsup_{m\to\infty}\left|\frac{\log\widetilde{\mathbb{P}}_{\theta_{\ast}(n_m,k_m)}[S_{n_m}=k_m]}{c_{n_m}}\right|>0,
\end{equation*}
contradicts~\eqref{eq:llt_changed_measure2} in view of~\eqref{eq:variance_linear}.

This finishes the proof of~\eqref{eq:uniform_convergence_of_profiles} with $g$ given by the right-hand side of~\eqref{eq:phi_def}.

We shall now check the first relation in~\eqref{eq:divergence_outside}. Note that, for every $\lambda>0$, by Markov's inequality
\begin{multline*}
\sup_{0\leq k\leq (\underline{m}-\varepsilon)n}\frac{\log \mathbb{P}[S_n=k]}{n}\leq \frac{\log \mathbb{P}[S_n\leq (\underline{m}-\varepsilon)n]}{n}=\frac{\log \mathbb{P}[\exp{(-\lambda S_n)}\geq \exp{(-\lambda(\underline{m}-\varepsilon)n)}]}{n}\\
\leq \frac{\log \varphi_n(\eee^{-\lambda})}{n}+\lambda(\underline{m}-\varepsilon).
\end{multline*}
Sending $n\to\infty$ and using~\eqref{eq:log_potential_converge}, imply
$$
\limsup_{n\to\infty}\sup_{0\leq k\leq (\underline{m}-\varepsilon)n}\frac{\log \mathbb{P}[S_n=k]}{n}\leq \Psi(\eee^{-\lambda})+\lambda(\underline{m}-\varepsilon)=\lambda(\lambda^{-1}\Psi(\eee^{-\lambda})+\underline{m}-\varepsilon).
$$
Since $\lambda>0$ can be taken arbitrarily large, the right-hand side can be made smaller than any large negative constant because
$$
\lim_{\lambda\to+\infty}\lambda^{-1}\Psi(\eee^{-\lambda})=-\underline{m},
$$
which in turn is a consequence of the Lebesgue dominated convergence theorem and the fact that
$$
\lim_{\lambda\to+\infty}\frac{\log(1-y+y\eee^{-\lambda})}{\lambda}=
\begin{cases}
0,& y\in [0,1),\\
-1,& y=1.
\end{cases}
$$
This proves the first relation in~\eqref{eq:divergence_outside}. The second one can be proved similarly using the inequality
\begin{multline*}
\sup_{(\overline{m}+\varepsilon)n\leq k\leq n}\frac{\log \mathbb{P}[S_n=k]}{n}\leq \frac{\log \mathbb{P}[S_n\geq (\overline{m}+\varepsilon)n]}{n}=\frac{\log \mathbb{P}[\exp{(\lambda S_n)}\geq \exp{(\lambda(\overline{m}+\varepsilon)n)}]}{n}\\
\leq \frac{\log \varphi_n(\eee^{\lambda})}{n}-\lambda(\overline{m}+\varepsilon)
\end{multline*}
and the fact that
$$
\lim_{\lambda\to+\infty}\lambda^{-1}\Psi(\eee^{\lambda})=\overline{m},
$$
The function $\Psi$ is infinitely differentiable on $t>0$ and so is $\Phi$. The function $\Phi^{\leftarrow}$ is infinitely differentiable on $(\underline{m},\overline{m})$, since $\Phi^{\prime}$ does not vanish. The function $u\mapsto \Psi(\eee^u)$ is strictly convex on $\mathbb{R}$, since
\begin{equation}\label{eq:psi_sec_derivative}
(\Psi(\eee^u))^{\prime\prime}=\int_{[0,1]}\frac{(1-y)y \eee^u}{(1-y+y \eee^u)^2}M_{\infty}({\rm d}y)>0,\quad u\in\mathbb{R},
\end{equation}
where the strict inequality follows from $M_{\infty}(\{0\})+M_{\infty}(\{1\})<1$. The strict concavity of $g$ follows from the fact that $-g$ is the Legendre transform of an infinitely differentiable strictly convex function $u\mapsto \Psi(\eee^u)$.

To obtain~\eqref{eq:phi_Phi_inversion} replace in~\eqref{eq:phi_def} $\alpha$ by $\Phi(t)$ and differentiate both sides to get $g'(\Phi(t))=-\log t$, $t>0$. This relation also implies that $\alpha\mapsto g(\alpha)+\alpha\log(\alpha)$ is concave. Indeed, if suffices to show that the derivative $\alpha\mapsto g'(\alpha)+\log \alpha + 1$ is decreasing or, replacing $\alpha$ by a strictly increasing function $t\mapsto \Phi(t)$, that $t\mapsto g'(\Phi(t))+\log \Phi(t)$ is decreasing. But this follows from the fact that $\Psi'$ is decreasing and the equality $g'(\Phi(t))+\log \Phi(t)=\log (\Phi(t)/t)=\log\Psi'(t)$.

The claims in~\eqref{eq:phi_derivatives} follow from~\eqref{eq:phi_Phi_inversion} because
$$
\lim_{\alpha\to \underline{m}+0}\Phi^{\leftarrow}(\alpha)=0\quad\text{and}\quad\lim_{\alpha\to \overline{m}-0}\Phi^{\leftarrow}(\alpha)=+\infty,
$$
see~\eqref{eq:M_infty_range}. The fact that $p_{*}=\Phi(1)\in (\underline{m},\overline{m})$ follows from~\eqref{eq:M_infty_range} and $g(p_{\ast})=0$ is a consequence of $\Phi^{\leftarrow}(p_{*})=1$, $\Psi(1)=0$ and the second equality in~\eqref{eq:phi_def}. In order to check~\eqref{eq:P_s_n_k_quotients}, note that~\eqref{eq:exponential_tilting_probabilities1} (applied with $\theta=\theta_{\ast}(n,k)$) and~\eqref{eq:llt_changed_measure} yield
$$
\frac{\mathbb{P}[S_n=k+1]}{\mathbb{P}[S_n=k]}=\frac{1}{\theta_{\ast}(n,k)}(1+o(1))\overset{\eqref{eq:local_uniform3}}{=}\frac{1}{\Phi^{\leftarrow}(k/n)}(1+o(1))\overset{\eqref{eq:phi_Phi_inversion}}{=}\eee^{g'(k/n)}(1+o(1)),
$$
where constants in $o(1)$ are uniform in the range $(\underline{m}+\varepsilon)n\leq k\leq (\overline{m}-\varepsilon)n$. For the proof of~\eqref{eq:P_s_n_k_quotients_infinity} we exploit a known fact that $(\mathbb{P}[S_n=k])_{\underline{r}_n\leq k\leq\overline{r}_n}$ is a log-concave sequence for every fixed $n\in\mathbb{N}$. This follows from Newton's inequality as explained in~\cite{Lieb:1968}, formula (3). By~\eqref{eq:phi_derivatives} and~\eqref{eq:P_s_n_k_quotients} the ratio $\frac{\mathbb{P}[S_n=k+1]}{\mathbb{P}[S_n=k]}$ can be made arbitrarily large as $k/n$ approaches $\underline{m}$ from the right. By log-concavity this ratio is monotone and thus even larger when $k/n\leq \underline{m}-\varepsilon$. This proves the first relation in~\eqref{eq:P_s_n_k_quotients_infinity}. Similar argument works for the second relation.

It remains to consider the case $\underline{m}=\overline{m}$, in which $M_{\infty}$ is supported by the two-point set $\{0,1\}$. In this case our proof of~\eqref{eq:divergence_outside} remains valid without any changes. For the proof of~\eqref{eq:P_s_n_k_quotients_infinity} first of all note that by the pigeonhole principle, for every $n\in\mathbb{N}$, there exists $k_{\ast}(n)\in\{\underline{r}_n,\ldots,\overline{r}_n\}$ such that $\mathbb{P}[S_n=k_{\ast}(n)]\geq 1/(n+1)$. From~\eqref{eq:divergence_outside} we conclude that, for every $A>0$ and $\varepsilon>0$, there exists $n_0\in\mathbb{N}$ such that
$$
\mathbb{P}[S_n\leq (\underline{m}-\varepsilon)n]\leq \eee^{-An}\quad\text{and}\quad \mathbb{P}[S_n\geq (\overline{m}+\varepsilon)n]\leq \eee^{-An},\quad n\geq n_0.
$$
Therefore, $k_{\ast}(n)\in\{(\underline{m}-\varepsilon)n,\ldots,(\overline{m}+\varepsilon)n\}$, for $n\geq n_0$, and
\begin{multline*}
\frac{1}{n+1}\leq \mathbb{P}[S_n=k_{\ast}(n)]=\mathbb{P}[S_n= \lfloor (\underline{m}-\varepsilon)n\rfloor]\prod_{k=\lfloor (\underline{m}-\varepsilon)n\rfloor}^{k_{\ast}(n)-1}\frac{\mathbb{P}[S_n=k+1]}{\mathbb{P}[S_n=k]}\\
\leq \eee^{-An}\prod_{k=\lfloor (\underline{m}-\varepsilon)n\rfloor}^{k_{\ast}(n)-1}\left(\max\left(\frac{\mathbb{P}[S_n=k+1]}{\mathbb{P}[S_n=k]},1\right)\right)\leq \eee^{-An}\max\left(\frac{\mathbb{P}[S_n=\lfloor (\underline{m}-\varepsilon)n\rfloor+1]}{\mathbb{P}[S_n=\lfloor (\underline{m}-\varepsilon)n\rfloor]},1\right)^n,
\end{multline*}
where on the last step we used monotonicity of $(\frac{\mathbb{P}[S_n=k+1]}{\mathbb{P}[S_n=k]})_{\underline{r}_n\leq k\leq \overline{r}_n}$. Taking logarithms on both sides, dividing by $n$ and sending $n\to\infty$ yield
$$
\liminf_{n\to\infty}\frac{\mathbb{P}[S_n=\lfloor (\underline{m}-\varepsilon)n\rfloor+1]}{\mathbb{P}[S_n=\lfloor (\underline{m}-\varepsilon)n\rfloor]}\geq \eee^A.
$$
Since $A$ can be taken arbitrary large and $(\frac{\mathbb{P}[S_n=k+1]}{\mathbb{P}[S_n=k]})_{\underline{r}_n\leq k\leq \overline{r}_n}$ is monotone, the first relation in ~\eqref{eq:P_s_n_k_quotients_infinity} follows. The second one follows by the same argument.

 The proof of Theorem~\ref{thm:main_probabilistic_direct} is complete.
\end{proof}

\begin{proof}[Proof of Theorem~\ref{thm:main_probabilistic_converse}]
The sequence of probability measures $(M_n)_{n\in\mathbb{N}}$ is tight because they are supported by a compact set $[0,1]$. Suppose that there are two different subsequential limits of $(M_n)_{n\in\mathbb{N}}$, say $M_{\infty}^{(1)}$ and $M_{\infty}^{(2)}$, and
$$
\underline{m}_i:=M_{\infty}^{(i)}(\{1\}),\quad \overline{m}_i:=1-M_{\infty}^{(i)}(\{0\}),\quad i=1,2.
$$
Assume that $M_{\infty}^{(1)}$ (respectively, $M_{\infty}^{(2)}$) is a limit of
$(M_{n_{1,j}})_{j\in\mathbb{N}}$ (respectively, $(M_{n_{2,j}})_{j\in\mathbb{N}}$) as $j\to\infty$.
According to Theorem~\ref{thm:main_probabilistic_direct}, for $i=1,2$, there exist infinitely differentiable strictly concave functions $g_i:(\underline{m}_i,\overline{m}_i)\to(-\infty,0]$  such
that, for every $\varepsilon>0$,
\begin{equation}\label{eq:uniform_convergence_of_profiles12}
\lim_{j\to\infty}\sup_{(\underline{m}_i+\varepsilon)n_{i,j}\leq k\leq (\overline{m}_i-\varepsilon)n_{i,j}}\left|\frac{1}{n_{i,j}}\log\mathbb{P}[S_{n_{i,j}}=k]-g_i\left(\frac {k}{n_{i,j}}\right)\right|~=~0,
\end{equation}
and
\begin{equation}\label{eq:divergence_outside1}
\lim_{j\to\infty}\sup_{0\leq k\leq (\underline{m}_i-\varepsilon) n_{i,j}}\frac{1}{n_{i,j}}\log\mathbb{P}[S_{n_{i,j}}=k]=-\infty,
\end{equation}
and
\begin{equation}\label{eq:divergence_outside12}
\lim_{j\to\infty}\sup_{(\overline{m}_i+\varepsilon) n_{i,j}\leq k\leq n_{i,j}}\frac{1}{n_{i,j}}\log\mathbb{P}[S_{n_{i,j}}=k]= -\infty.
\end{equation}
Picking $k=k(n_{1,j})$ such that $k/n_{1,j}\to \alpha\in(0,1)$ we conclude using~\eqref{eq:main_probabilistic_converse_assumption} that $\underline{m}_1=\underline{m}$, $\overline{m}_1=\overline{m}$ and $g(\alpha)=g_1(\alpha)$, for $\alpha\in(\underline{m},\overline{m})$. In the same vein, picking $k=k(n_{2,j})$ such that $k/n_{2,j}\to \alpha\in(0,1)$ we conclude that $\underline{m}_2=\underline{m}$, $\overline{m}_2=\overline{m}$ and $g(\alpha)=g_2(\alpha)$, for $\alpha\in(\underline{m},\overline{m})$.

The functions $-g_1$ and $-g_2$ are strictly convex and infinitely differentiable on the common interval $(\underline{m},\overline{m})$ and are the Legendre transformations of $u\mapsto \Psi_1(\eee^u)$ and $u\mapsto \Psi_2(\eee^u)$, respectively, where
$$
\Psi_i(t)=\int_{[0,1]}\log(1-y+yt)M_{\infty}^{(i)}({\rm d}y),\quad t>0,\quad i=1,2.
$$
Since the Legendre transformation is an involution on the set of strictly convex functions, we conclude that $\Psi_{1}=\Psi_{2}$. Therefore, $\Psi^{\prime}_1=\Psi^{\prime}_2$ from which $M_{\infty}^{(1)}=M_{\infty}^{(2)}$ by the uniqueness theorem for Cauchy transforms. Thus, any two limit points for $(M_n)_{n\in\mathbb{N}}$ coincide thereby finishing the proof.
\end{proof}

\subsection{Proofs about zeros and profiles of polynomials}
We shall first derive Theorem~\ref{theo:zeros_imply_exp_profile} as a direct corollary of Theorem~\ref{thm:main_probabilistic_direct}.

\begin{proof}[Proof of Theorem \ref{theo:zeros_imply_exp_profile}]
Let $P_n(x) = \sum_{k=0}^n a_{k:n} x^k$ be a polynomial of degree at most $n$ with only nonpositive roots $(-\lambda_{k:n})_{1\leq k\leq \deg P_n}$. If $\deg P_n<n$ we stipulate $\lambda_{k:n}=+\infty$, for $\deg P_n<k\leq n$. Then
$$
\widehat{P}_n(x):=\frac{P_n(x)}{P_n(1)}=\prod_{k=1}^{\deg P_n}\frac{x+\lambda_{k:n}}{1+\lambda_{k:n}}=\prod_{k=1}^{n}\frac{x+\lambda_{k:n}}{1+\lambda_{k:n}}.
$$
is the generating function of the sum $\widehat{S}_n:=\widehat{B}_{1:n}+\cdots+\widehat{B}_{n:n}$,
where $(\widehat{B}_{k:n})_{1\leq k\leq n}$ are mutually independent Bernoulli random variables such that $\mathbb{P}[\widehat{B}_{k:n}=1]=1/(1+\lambda_{k:n})\in (0,1]$ for $k\leq \deg P_n$ and $\mathbb{P}[\widehat{B}_{k:n}=1]=0$, for $\deg P_n<k\leq n$.

Let $T:[-\infty,0]\to [0,1]$ be the mapping $T(x)=\frac{1}{1-x}$. Note the assumption
$$
\frac{1}{n}\sum_{k=1}^{n}\delta_{-\lambda_{k:n}}\toweak \mu
$$
is equivalent to
$$
\frac{1}{n}\sum_{k=1}^{n}\delta_{1/(1+\lambda_{k:n})}\toweak \mu\circ T^{\leftarrow}=:M_{\infty}.
$$
Moreover, $\mu(\{0\})=M_{\infty}(\{1\})$ and $\mu(\{-\infty\})=M_{\infty}(\{0\})$. Thus, Theorem~\ref{thm:main_probabilistic_direct} is applicable with $\mathbb{P}[\widehat{S}_n=k]=a_{k:n}/P_n(1)$ and yields the claim (a), that is,
$$
\sup_{(\underline{m}+\varepsilon) n\leq k \leq (\overline{m}-\varepsilon)n}\left|\frac{\log (a_{k :n}/P_n(1))}{n}-g\left(\frac{k }{n}\right)\right|\ton 0,
$$
where $g$ is given by~\eqref{eq:phi_def}. Part (b) follows from~\eqref{eq:phi_Phi_inversion} since $\Phi(t)=tG_{\mu}(t)$ and (c) follows from \eqref{eq:phi_derivatives}. Part (d) with $\Psi$ given by formula~\eqref{eq:psi_for_polys} follows from
$$
\Psi(t)=\int_{[0,1]}\log (1-y+yt)M_{\infty}({\rm d}y)
$$
upon changing the variable $y=1/(1-z)$.  Relations~\eqref{eq:P_s_n_k_quotients_polys} and~\eqref{eq:P_s_n_k_quotients_infinity_polys} follow from~\eqref{eq:P_s_n_k_quotients} and~\eqref{eq:P_s_n_k_quotients_infinity}. Ultimately, also (f) follows and the proof is complete.
\end{proof}

\begin{remark}
Note that
\begin{equation*}
\frac{\log P_{n}(1)}{n}=\int_{(-\infty,0]}\log (1-x)\lsem P_n\rsem_n ({\rm d}x).
\end{equation*}
The function $x\mapsto \log(1-x)$ is continuous but unbounded on $[-\infty,0]$. Therefore, $\lsem P_n\rsem_n\toweak \mu$ on $[-\infty,0]$ does not in general imply that
\begin{equation}\label{eq:log_moments_convergence}
\frac{\log P_{n}(1)}{n}=\int_{[-\infty,0]}\log (1-x)\lsem P_n\rsem_n ({\rm d}x)\ton \int_{[-\infty,0]}\log(1-x)\mu({\rm d}x)=:L_{\mu}\in (0,+\infty],
\end{equation}
even if $L_{\mu}<\infty$. However, there are situations when~\eqref{eq:log_moments_convergence} holds automatically. Assume, for example that that $\deg P_n=n$ and all roots of $P_n$ are contained in the interval $[c_1,0]$, where $c_1 < 0$  does not depend on $n$. Then $\mu$ is concentrated on $[c_1,0]$, $L_{\mu}=\int_{[c_1,0]}\log(1-x)\mu({\rm d}x)\in (0,+\infty)$ and~\eqref{eq:log_moments_convergence} holds. In this case also $\overline{m}=1$ and~\eqref{eq:zeros_imply_exp_profile_claim} can be replaced by
$$
\sup_{(\underline{m}+\eps) n \leq k \leq (1-\eps) n} \left|\frac{\log a_{k:n}}{n} -L_{\mu} -g\left(\frac kn\right)  \right| \ton 0,
\qquad
\text { for all }\eps\in (0,1).
$$
A necessary and sufficient condition for~\eqref{eq:log_moments_convergence} is the uniform integrability of the family of random variables $(\log (1-Z_n))_{n\in\mathbb{N}}$, where $Z_n$ has the distribution $\lsem P_n\rsem_n$, $n\in\mathbb{N}$.
\end{remark}

We shall now derive Theorem~\ref{theo:exp_profile_implies_zeros} as a consequence of Theorem~\ref{thm:main_probabilistic_converse}.

\begin{proof}[Proof of Theorem \ref{theo:exp_profile_implies_zeros}]
It suffices to check~\eqref{eq:main_polys_converse_assumption_p_n(1)}. Indeed, then
\begin{equation*}
\lim_{n\to\infty}\frac{1}{n}\log \frac{a_{\lfloor \alpha n \rfloor:n}}{P_n(1)}=
\begin{cases}
g(\alpha)-\sup_{\alpha\in (\underline{m},\overline{m})}(g(\alpha)),&\text{if }\alpha\in (\underline{m},\overline{m}),\\
-\infty,&\text{if }\alpha\in [0,\underline{m})\cup (\overline{m},1].
\end{cases}
\end{equation*}
and the statements (a)-(e) follow from Theorem~\ref{thm:main_probabilistic_converse} applied to $\mathbb{P}[S_n=\lfloor\alpha n\rfloor]=a_{\lfloor \alpha n\rfloor:n}/P_n(1)$.

For the proof of~\eqref{eq:main_polys_converse_assumption_p_n(1)} note that
\begin{equation}\label{eq:main_polys_converse_assumption_p_n(1)_proof1}
\max_{0\leq k\leq n}\frac{\log a_{k:n}}{n}\leq \frac{\log P_n(1)}{n} \leq \max_{0\leq k\leq n}\frac{\log a_{k:n}}{n}+\frac{\log (n+1)}{n}.
\end{equation}
Using that $P_n$ has only nonpositive real roots, we conclude that the sequence
$$
b_{k:n}:=\frac{1}{n}\log a_{k:n},\quad k\in\{0,1,\ldots,n\},
$$
is concave for every fixed $n\in\mathbb{N}$. Let $g_n:[0,1]\to \mathbb{R}$ be a linear interpolation between the points $\left(\frac{k}{n},b_{k:n}\right)$, $k\in\{0,1,\ldots,n\}$. Thus, $g_n$ is a concave function and
$$
\max_{0\leq k\leq n}\frac{\log a_{k:n}}{n}=\sup_{\alpha\in [0,1]}(g_n(\alpha))
$$
Condition~\eqref{eq:exp_profile_definition} implies that for every $\alpha\notin\{\underline{m},\overline{m}\}$, $g_n(\alpha)$ converges to the right hand side of~\eqref{eq:exp_profile_definition} as $n\to\infty$. By~\cite[Theorem 7.17]{Rockafellar+Wets:1998} this yields hypo-graphical convergence of $g_n$ to the same limit (which we extend to $\alpha\in\{\underline{m},\overline{m}\}$ by upper semicontinuity). By Theorem 7.33 in the same reference,
$$
\sup_{\alpha\in [0,1]}g_n(\alpha)~\to~
\sup_{\alpha\in (\underline{m},\overline{m})}(g(\alpha)),\quad n\to\infty.
$$
Combining this with~\eqref{eq:main_polys_converse_assumption_p_n(1)_proof1} yields~\eqref{eq:main_polys_converse_assumption_p_n(1)}. The proof is complete.
\end{proof}

\begin{remark}
It seems possible to extend Theorem~\ref{theo:zeros_imply_exp_profile} to polynomials with real coefficients having only roots with nonpositive real part. Since such roots come in complex conjugated pairs and $(x-a-b\ii)(x-a+b\ii) = x^2 - 2ax + a^2+b^2$, it is sufficient to extend Theorem~\ref{thm:main_probabilistic_direct} to random variables taking values in $\{0,1,2\}$. However, there are serious difficulties when trying to extend Theorems~\ref{thm:main_probabilistic_converse} and~\ref{theo:exp_profile_implies_zeros} to this setting since knowing the Cauchy transform on an interval of the real line does not allow us to recover a general distribution in the left half-plane. For example, a $(x+2)^n$ has the same profile as $(x+2)^n - 1$, but the limiting distributions of roots are different.
\end{remark}

\begin{remark}
Using formulas~\eqref{eq:exponential_tilting_probabilities1} and~\eqref{eq:llt_changed_measure} one can write down the exact asymptotics for $\P[S_n = k]$ valid uniformly in the range $\eps n \leq k \leq (1-\eps) n$. Naturally, this asymptotics involves $\theta_{\ast}(n,k)$. We shall return to this discussion in~\cite{jalowy_kabluchko_marynych_zeros_profiles_part_III}, where these exact asymptotics will become key ingredients for deriving functional limit theorems.
\end{remark}

\section{Proofs of the results on finite free convolutions}\label{sec:proof_convolutions}

\subsection{Free probability tools: \texorpdfstring{$R$}{R}- and \texorpdfstring{$S$}{S}-transforms via exponential profiles} \label{subsec:R_S_transf}
Let $\mu$ be a probability measure on $[-\infty,A]$, for some $A\in \R$. Recall that the Cauchy transform of $\mu$ is defined by
$$
G_{\mu}(t):=\int_{[-\infty,A]}\frac{1}{t-z}\dd\mu(z)=\int_{(-\infty,A]}\frac{1}{t-z}\dd\mu(z),  \qquad t\in \C \backslash (-\infty, A].
$$
Suppose that $\mu(\{-\infty\})\neq 1$ to exclude the case $G_\mu \equiv 0$.  Restricted to the interval $(A,\infty)$, $t\mapsto G_\mu(t)$ is a continuous strictly decreasing function and  $\lim_{t\to +\infty}G_\mu(t)=0$. Hence, the inverse function $G_{\mu}^{\leftarrow}(s)$ exists for sufficiently small $s>0$, more precisely, for $s\in (0, G_\mu(A+))$.  The $R$-transform~\cite[Theorem~3.3.1, pp.~23--24]{voiculescu_nica_dykema_book} of $\mu$ is then defined by
\begin{align}\label{eq:summary_R}
R _\mu(s) := s G_{\mu}^{\leftarrow}(s)-1,\quad s\in(0,G_\mu(A+)).
\end{align}

For $S$-transforms, see~\cite[pp.~30-31]{voiculescu_nica_dykema_book}, it is customary to work with probability measures supported on the positive half-line. Accordingly, let $\nu$ be a probability measure on $[0, +\infty]$.
The $\psi$-transform of $\nu$ is defined via its Cauchy transform $G_{\nu}(t)=\int_{[0,+\infty)}\frac{1}{t-z}\dd\nu(z)$ via
\begin{equation}\label{eq:psi_mu_def}
\psi_\nu (t) := G_{\nu}(1/t)/t-1=\int_{[0,+\infty]} \frac{zt}{1-zt}\dd\nu(z) = -\nu(\{+\infty\})+\int_{(0,+\infty)} \frac{zt}{1-zt}\dd\nu(z),\quad t<0.
\end{equation}
Suppose that $\nu(\{0\}) + \nu(\{+\infty\})<1$  to exclude the case $\psi_\nu \equiv -\nu(\{+\infty\})$. Then, it follows from the last equation in~\eqref{eq:psi_mu_def} and the monotone convergence theorem that $\psi_\nu$ defines a continuous, strictly increasing  bijection of $(-\infty,0)$ and $(\nu(\{0\})-1,-\nu(\{+\infty\}))$.
The $S$-transform of $\nu$ is defined by
\begin{align}\label{eq:summary_S}
S_\nu (t) := \frac{t+1}{t}\psi_\nu^{\leftarrow}(t),\quad t\in
(\nu(\{0\})-1,-\nu(\{+\infty\})).
\end{align}

Let us now relate these transforms to exponential profiles. To this end, let $\mu$ be a probability measure on $[-\infty, 0]$ such that $\mu(\{-\infty\})+\mu(\{0\})<1$. For every $n\in \N$ let  $P_n(x)$ be a polynomial of degree at most $n$ with nonpositive zeros, nonnegative coefficients and limiting zero distribution $\lsem P_n\rsem_n\toweak \mu$ on $[-\infty,0]$.  Recall that under these assumptions the exponential profile $g$ of $P_n(x)/P_n(1)$ exists on the open interval $(\underline{m},\overline{m})=(\mu(\{0\}),1-\mu(\{-\infty\}))$ due to Theorem~\ref{theo:zeros_imply_exp_profile}. Moreover, the function $g'$ is a strictly decreasing bijection between $(\underline{m},\overline{m})$ and $(-\infty,+\infty)$ and satisfies
\begin{align}\label{eq:summary_Cauchy}
G_{\mu}(t)=\int_{[-\infty,0]}\frac{1}{t-z}\dd\mu(z)=\int_{(-\infty,0]}\frac{1}{t-z}\dd\mu(z)=\frac 1 t (g')^{\leftarrow}\big(-\log t\big),\quad &t>0.
\end{align}
Define a measure $\mu_{+}$ on $[0,+\infty]$ by $\mu_{+}(A)=\mu(-A)$ for every Borel. The next lemma expresses the $R$- transform of $\mu$ and the $\psi$-, $S$-transforms of $\mu_{+}$ via the exponential profile $g$.
\begin{lemma}\label{lem:r_and_s_transform_vs_profile}
For a probability measure $\mu$ on $[-\infty, 0]$ such that $\mu(\{-\infty\})+\mu(\{0\})<1$ we have
\begin{align}
R_{\mu}(t)
&=
\big[(1+\cdot)\eee^{g'(1+\cdot)}\big]^{\leftarrow}(t)= \big[g'(\cdot)+\log(\cdot)\big]^{\leftarrow}(\log t)-1,\quad t\in(0,G_{\mu}(0+)),
\label{eq:r_trans_via_profile}
\\
\psi_{\mu_{+}} (-t)
&=
(g')^{\leftarrow}(\log t)-1,\quad t>0,
\label{eq:summary_psi}
\\
S_{\mu_{+}}(t)
&=-\frac{t+1}{t}\eee^{g'(t+1)},\quad t\in 
(\underline{m}-1,\overline{m}-1). \label{eq:s_trans_via_profile}
\end{align}
The map $\alpha \mapsto \alpha \eee^{g'(\alpha)}$ defines a strictly decreasing bijection between $(\underline{m},\overline{m})$ and $(0,G_{\mu}(0+))$, which ensures that the inverse function in~\eqref{eq:r_trans_via_profile} is well-defined.
\end{lemma}
\begin{proof}
We start with proving the stated properties of the function $\alpha \mapsto \alpha \eee^{g'(\alpha)}$. The map $\alpha \mapsto \alpha \eee^{g'(\alpha)}=\eee^{g'(\alpha)+\log \alpha}$ is strictly decreasing by Theorem~\ref{theo:zeros_imply_exp_profile}(d). Observe that
$\lim_{\alpha\to \overline{m}-}\alpha \eee^{g'(\alpha)}=0$ since $\lim_{\alpha\to \overline{m}-}g'(\alpha)=-\infty$. If $\underline{m}>0$, then $g'(\underline{m}+)=+\infty$ and hence
$\lim_{\alpha\to \underline{m}+}\alpha \eee^{g'(\alpha)}=+\infty$, which coincides with $G_{\mu}(0+)=+\infty$. Finally, if $\underline{m}=0$, then $\lim_{t\to 0+}tG(t)=\underline{m}=0$ by Theorem~\ref{theo:zeros_imply_exp_profile}(b) and therefore
\begin{equation}\label{eq:G_0_plus_as_limit}
\lim_{\alpha\to\underline{m}+}\alpha\eee^{g'(\alpha)}=\lim_{\alpha\to 0+}\alpha\eee^{g'(\alpha)}=\lim_{\alpha\to 0+}\frac{\alpha}{\eee^{-g'(\alpha)}}=\lim_{t\to 0+}\frac{tG_\mu(t)}{\eee^{-g'(tG_{\mu}(t))}}=\lim_{t\to 0+}G_\mu(t)=G_{\mu}(0+).
\end{equation}
In order to prove~\eqref{eq:r_trans_via_profile} note that
$$
R_{\mu}(G_{\mu}(t))=tG_{\mu}(t)-1,\quad t>0.
$$
Inverting this equality and using Theorem~\ref{theo:zeros_imply_exp_profile}(b) yields
$$
G_{\mu}^{\leftarrow}(R_{\mu}^{\leftarrow}(s))=\eee^{-g'(s+1)},\quad s\in (\underline{m}-1,\overline{m}-1).
$$
Applying $G_{\mu}$ to the both sides and again using Theorem~\ref{theo:zeros_imply_exp_profile}(b) implies
$$
R_{\mu}^{\leftarrow}(s)=G_{\mu}(\eee^{-g'(s+1)})=G_{\mu}(\eee^{-g'(s+1)})\eee^{-g'(s+1)}\eee^{g'(s+1)}=(1+s)\eee^{g'(s+1)},\quad s\in (\underline{m}-1,\overline{m}-1),
$$
which yields~\eqref{eq:r_trans_via_profile}.

For the proof of~\eqref{eq:summary_psi} observe that
$$
G_{\mu_{+}}(-s)=\int_{[0,+\infty)}\frac{\dd \mu_{+}(v)}{-s-v}=\int_{(-\infty,0]}\frac{\dd \mu(z)}{-s+z}=-G_{\mu}(s),\quad s>0.
$$
Hence using the first equality in~\eqref{eq:psi_mu_def} together with the third equality in~\eqref{eq:summary_Cauchy} we obtain
$$
\psi_{\mu_{+}}(-t)=-\frac{1}{t}G_{\mu_{+}}(-1/t)-1=\frac{1}{t}G_{\mu}(1/t)-1=(g')^{\leftarrow}(\log t)-1,\quad t>0.
$$
This proves~\eqref{eq:summary_psi} and again shows that  $\psi_{\mu_{+}}$ is a bijection of $(-\infty,0)$ and $(\underline{m}-1,\overline{m}-1)=\big(\mu(\{0\})-1,-\mu(\{-\infty\})\big)$.
Formula~\eqref{eq:s_trans_via_profile} follows immediately from~\eqref{eq:summary_psi} and~\eqref{eq:summary_S} by taking inverses.
\end{proof}

\begin{remark}
Recall from Theorem~\ref{theo:zeros_imply_exp_profile}~(e) that the quotient $a_{k+1:n}/a_{k:n}$ converges to $\eee^{g'(\alpha)}$ as $n\to\infty$ if $k= k(n)\sim \alpha n$ with $\alpha\in (\underline{m}, \overline{m})$. Equation~\eqref{eq:s_trans_via_profile} relates $\eee^{g'(\alpha)}$ to the $S$-transform of $\mu_+$. This motivated  a definition of a finite $S$-transform in~\cite{arizmendi2024s}.
\end{remark}

\subsection{Definitions of free convolutions}\label{subsec:free_conv_defs}
The $R$- and $S$-transforms  play a central role in free probability theory since they linearize additive and multiplicative free convolutions.
We need an extension of these convolutions to probability measures which may have an atom at $-\infty$ (for $\boxplus$) or $+\infty$ (for $\boxtimes$). The following Propositions~\ref{prop:free_additive_convol_existence} and~\ref{prop:free_multiplicative_convol_existence} will serve as definitions of $\boxplus$ and $\boxtimes$ in this setting.

\begin{proposition}\label{prop:free_additive_convol_existence}
Let $\mu_1$ and $\mu_2$ be probability measures on $[-\infty, A]$ such that
$$
\mu_1(\{-\infty\}) + \mu_2(\{-\infty\}) < 1.
$$
Here, $A\in \R$. Then, there exists a probability measure $\mu_1\boxplus\mu_2$ on $[-\infty, 2A]$ having the $R$-transform
\begin{align}\label{eq:R_free_convolution}
R_{\mu_1\boxplus\mu_2}(t)=R_{\mu_1}(t)+R_{\mu_2}(t),
\end{align}
for all sufficiently small $t>0$.
\end{proposition}
A verification of Proposition~\ref{prop:free_additive_convol_existence} will be a part of the proof (Step 4) of Theorem~\ref{theo:finite_free_additive_conv_free_additive} given in Section~\ref{subsec:proof_additive_conv}. Recall that $R_{\mu_i}$ is  defined on an interval $(0,\eps_i)$ with sufficiently small $\eps_i>0$. Then, $R_{\mu_1\boxplus\mu_2}$ is defined on the interval $(0, \min (\eps_1,\eps_2))$.
\begin{remark}
If $\mu_1$ and $\mu_2$ have no atoms at $-\infty$, then $\mu_1 \boxplus \mu_2$ coincides with  the classical free convolution. This follows from~\eqref{eq:R_free_convolution}. Formula~\eqref{eq:atom_at_infinity_over_R} implies that $\lim_{t\to +\infty} t G_{\mu_i}(t) = 1-{\mu}_i(\{-\infty\})$ and from Lemma~\ref{lem:r_and_s_transform_vs_profile} it follows that $R_{\mu_i}(0+) = -\mu_i(\{-\infty\})$, for $i=1,2$. As a consequence,  $(\mu_1\boxplus\mu_2)(\{-\infty\}) = \mu_1(\{-\infty\}) + \mu_2(\{-\infty\})<1$. If $\mu_1$ and $\mu_2$ are such that $\mu_1(\{-\infty\}) + \mu_2(\{-\infty\})=1$, it is natural to extend the above definition by putting $\mu_1 \boxplus \mu_2 = \delta_{-\infty}$. Finally, if $\mu_1(\{-\infty\}) + \mu_2(\{-\infty\})>1$, then $\mu_1 \boxplus \mu_2$ is not well-defined.
\end{remark}
\begin{example}
Let $\mu_i = p_i \delta_{a_i} + (1-p_i)\delta_{-\infty}$ for $i=1,2$,  where $p_1,p_2\in [0,1]$ are such that $p_1+p_2> 1$ and $a_1,a_2\in \R$. Then, $G_{\mu_i}(t) = \frac{p_i}{t-a_i}$ and $R_{\mu_i}(t) = p_i + a_i t - 1$. Thus,  $R_{\mu_1\boxplus \mu_2}(t) = (p_1+p_2-1) + (a_1+a_2)t -1$ and hence $\mu_1 \boxplus \mu_2 = (p_1+p_2-1) \delta_{a_1+a_2} + (2-p_1-p_2) \delta_{-\infty}$.
\end{example}

A general formula for $\mu_1\boxplus \mu_2$ for probability measures $\mu_1$ and $\mu_2$ which may have atoms at $-\infty$
in terms of the usual operation $\boxplus$ will be given in Proposition~\ref{prop:free_add_conv_with_infty} below.

Let us now turn to $\boxtimes$ and the $S$-transforms. A verification of the next proposition will be given at Step 3 of the proof of Theorem~\ref{theo:finite_free_mult_conv} in Section~\ref{subsec:proof_multipl_conv}.

\begin{proposition}\label{prop:free_multiplicative_convol_existence}
Let $\nu_1$ and $\nu_2$ be probability measures on $[0,+\infty]$ such that~\eqref{eq:non-empty_intersection} holds true. Then, there exists a probability measure $\nu_1\boxtimes\nu_2$ on $[0,+\infty]$ having the $S$-transform
\begin{align}\label{eq:S_free_convolution}
S_{\nu_1 \boxtimes \nu_2}(t) = S_{\nu_1}(t) S_{\nu_2}(t),
\end{align}
for all $t$ in $I:=(\nu_1(\{0\})-1,-\nu_1(\{+\infty\}))\cap (\nu_2(\{0\})-1,-\nu_2(\{+\infty\}))$.
\end{proposition}
Condition~\eqref{eq:non-empty_intersection} on the atoms of $\nu_1$ and $\nu_2$ is equivalent to $I\neq \varnothing$.  If $\nu_1(\{+\infty\})=\nu_2(\{+\infty\})=0$, then $\nu_1\boxtimes \nu_2$ is the classical free multiplicative convolution. This follows from~\cite[Corollary~6.6]{bercovici_voiculescu}.

\subsection{Proof for the additive free convolution}\label{subsec:proof_additive_conv}
In this and the subsequent sections we prove Theorems~\ref{theo:finite_free_additive_conv_free_additive} and~\ref{theo:finite_free_mult_conv}. The idea of both proofs is to pass to exponential profiles. For concreteness, consider the setting of Theorem~\ref{theo:finite_free_additive_conv_free_additive}. Knowing the asymptotic zero distributions of $P_n^{(1)}$ and $P_n^{(2)}$ we compute the exponential profiles of these polynomials using Theorem~\ref{theo:zeros_imply_exp_profile}. Knowing these profiles, we compute the exponential profile of $P_n^{(1)} \boxplus_n P_n^{(2)}$.   Finally, applying Theorem~\ref{theo:exp_profile_implies_zeros} we can compute the asymptotic distribution of roots of the finite free convolution. Identifying the distributions on the last step is most conveniently done using $R$- and $S$- transforms as defined in Section~\ref{subsec:R_S_transf}.

\begin{proof}[Proof of Theorem~\ref{theo:finite_free_additive_conv_free_additive} and Proposition~\ref{prop:free_additive_convol_existence}] The proof is divided into a sequence of steps.

\noindent
\textsc{Step 0: Preparation.} We shall first  give a proof assuming that $A<0$.  This assumption will be removed in the last step of the proof. So, all roots of $P_n^{(1)}(x)= \sum_{k=0}^n a_{k:n}^{(1)} x^k$ and $P_n^{(2)}(x) = \sum_{k=0}^n a_{k:n}^{(2)} x^k$ are negative and lie in some interval $[-\infty,A]$ with $A < 0$.
Also, $\lsem P_n^{(i)}\rsem_n \to \mu_i$ weakly on $[-\infty, A]$. Here and in the following, the index $i$ ranges in $\{1,2\}$. After multiplication by non-zero numbers we may and shall assume that $P_n^{(1)}(1)=P_n^{(2)}(1)=1$. Moreover, since $A<0$ all coefficients $a_{k:n}^{(i)}$, $k\in \{0,\ldots, n\}$,  are nonnegative. Let $\overline{m}_i:= 1- \mu_i(\{-\infty\})$. Since $\mu_i$ is concentrated on $[-\infty,A]$ and $A<0$ it has no atom at $0$ and we have $\underline{m}_i := 0$. Let $G_i(t)= G_{\mu_i}(t)$ be the Cauchy transform of the probability measure $\mu_i$.

\smallskip\noindent
\textsc{Step 1: Profiles of $P_n^{(1)}$ and $P_n^{(2)}$.}
 From Theorem~\ref{theo:zeros_imply_exp_profile} it follows that $P_n^{(i)}$ has an exponential profile $g_i$ on the interval $(0,\overline{m}_i)$, more precisely,
\begin{equation}\label{eq:proof_finite_free_add_profiles_12}
\sup_{\eps n \leq \ell \leq (\overline{m}_i -\eps) n} \left|\frac 1n \log a_{\ell:n}^{(i)} - g_i\left(\frac \ell n\right)  \right| \ton 0
\quad \text{ and } \quad
\sup_{(\overline{m}_i+\eps)n \leq \ell  \leq n} \frac 1n   \log a_{\ell :n}^{(i)} \ton  -\infty,
\end{equation}
for all sufficiently small $\eps>0$.  Here, the profiles $g_i:(0,\overline{m}_i) \to \R$ are strictly concave, infinitely differentiable functions such that $\eee^{-g_i'}:(0,\overline{m}_i) \to (0,+\infty)$ is a strictly increasing bijection with inverse $t\mapsto tG_i(t)$.

\smallskip\noindent
\textsc{Step 2: Profile of $P_n^{(1)} \boxplus_n P_n^{(2)}$.} Knowing that the sequences $P_{n}^{(1)}$ and $P_n^{(2)}$ have exponential profiles $g_1$ and $g_2$, we can compute the exponential profile of their finite free convolution. Recall from Definition \ref{def:finite_free_convolution} that the coefficients of the finite free convolution are given by
\begin{align}\label{eq:d_k:n}
d_{k:n} := \sum_{j_1+j_2=n+k} a_{j_1:n}^{(1)} a_{j_2:n}^{(2)} \frac{j_1! j_2!}{k! n!}, \qquad k\in \{0,\ldots, n\}.
\end{align}
\begin{lemma}\label{lem:exp_profile_finite_free_additive_conv}
Under assumptions of Theorem~\ref{theo:finite_free_additive_conv_free_additive} and $P_n^{(1)}(1)=P_n^{(2)}(1)=1$, the sequence of polynomials  $P_n = P_n^{(1)} \boxplus_n P_n^{(2)}$, $n\in \N$,  has an exponential profile $g:(0,\overline{m}_1 + \overline{m}_2 - 1) \to \R$ given by
\begin{equation}\label{eq:proof_finite_free_add_profile_of_conv}
g(\gamma):= 
\sup_{\substack{\alpha_1\in (0, \overline{m}_1), \; \alpha_2 \in (0, \overline{m}_2)\\\alpha_1 + \alpha_2 = \gamma + 1}} \left(g_1(\alpha_1) + \alpha_1 \log \alpha_1 + g_2(\alpha_2) +  \alpha_2 \log \alpha_2\right) - \gamma \log \gamma,
\end{equation}
for all $\gamma \in (0,\overline{m})$, where $\overline{m}:=\overline{m}_1 + \overline{m}_2 - 1$ is positive by assumption (iii) of Theorem~\ref{theo:finite_free_additive_conv_free_additive}. That is, for all sufficiently small $\eps>0$,
\begin{equation}\label{eq:lem:exp_profile_finite_free_additive_conv}
\sup_{\eps n \leq k \leq (\overline{m} -\eps) n} \left|\frac 1n \log d_{k:n} - g\left(\frac k n\right)  \right| \ton 0
\quad\text{ and }\quad
\sup_{(\overline{m} + \eps)n\leq k\leq n} \frac 1 n \log d_{k:n}  \ton -\infty,
\end{equation}
\end{lemma}
The idea of the proof is simple:  putting formally $k=\gamma n$ and $j_1 = \alpha_1 n$, $j_2 = \alpha_2 n$ with $\gamma,\alpha_1,\alpha_2\in (0,1)$ and $\alpha_1+\alpha_2 = \gamma+1$, we apply~\eqref{eq:proof_finite_free_add_profiles_12} and the Stirling formula to get
$$
\frac 1n \log \left(a_{j_1:n}^{(1)} a_{j_2:n}^{(2)} \frac{j_1! j_2!}{k! n!}\right) \ton   g_1(\alpha_1) + \alpha_1 \log \alpha_1 + g_2(\alpha_2) +  \alpha_2 \log \alpha_2 - \gamma \log \gamma,
$$
where we define $g_i(\alpha_i)=-\infty$ for $\alpha_i\in (\overline{m}_i,1)$ and $i=1,2$. (More details will be given below.) The Laplace method applied to the sum in~\eqref{eq:d_k:n} yields~\eqref{eq:lem:exp_profile_finite_free_additive_conv}.   However, to make these considerations rigorous, some efforts are needed. The first step is done in the next

\begin{lemma}
For every $\gamma \in (0,\overline{m})$ the supremum in~\eqref{eq:proof_finite_free_add_profile_of_conv} is attained at a unique point $(\alpha_1,\alpha_2)$. Also, if $\gamma$ ranges in $(\eps,\overline{m}-\eps)$ for some $\eps>0$, then the point $(\alpha_1, \alpha_2)$ for which  the supremum in~\eqref{eq:proof_finite_free_add_profile_of_conv} is attained, satisfies $\alpha_i\in (\delta, \overline{m_i}-\delta)$, $i=1,2$, for some $\delta = \delta(\eps)>0$.
\end{lemma}
\begin{proof}
Note that $\overline{m}_1 + \overline{m}_2 > 1$ by condition~(iii) of Theorem~\ref{theo:finite_free_additive_conv_free_additive} and hence $0< \overline{m} \leq 1$. The set of $\alpha_1,\alpha_2$ over which the supremum in~\eqref{eq:proof_finite_free_add_profile_of_conv} is taken is nonempty, hence $g(\gamma)$ is well-defined and finite.
 The range of $\alpha_1$ is determined by the conditions $\alpha_1 \in (0, \overline{m}_1)$ and $\alpha_2 = \gamma + 1 - \alpha_1\in (0, \overline{m}_2)$, which simplify to $\alpha_1\in (\gamma+1-\overline{m_2}, \overline{m_1})$. Write
$$
g(\gamma) = \sup_{\alpha_1\in (\gamma+1-\overline{m_2}, \overline{m_1})} \mathfrak{f}_{\infty}(\alpha_1,\gamma)
$$
with
$$
\mathfrak{f}_{\infty}(\alpha_1,\gamma) = g_1(\alpha_1) + \alpha_1 \log \alpha_1 + g_2(\gamma + 1 - \alpha_1) + (\gamma + 1 - \alpha_1) \log (\gamma + 1 - \alpha_1)-\gamma\log \gamma.
$$
Taking derivative in $\alpha_1$ gives
$$
\frac{\partial}{\partial \alpha_1} \mathfrak{f}_{\infty}(\alpha_1, \gamma) = g_1'(\alpha_1) + \log \alpha_1 - g_2'(\gamma + 1 - \alpha_1) - \log (\gamma + 1 - \alpha_1).
$$
By Theorem~\ref{theo:exp_profile_implies_zeros}, the function $\alpha_1 \mapsto g_1'(\alpha_1) + \log \alpha_1$ is strictly decreasing and, by the same theorem, the function $\alpha_2 \mapsto g_2'(\gamma + 1 - \alpha_1) + \log (\gamma + 1 - \alpha_1)$ is strictly increasing.
It follows that $\frac{\partial}{\partial \alpha_1} \mathfrak{f}_{\infty}(\alpha_1, \gamma)$  is strictly increasing on the interval $(\gamma+1-\overline{m_2}, \overline{m_1})$. Also, by Theorem~\ref{theo:zeros_imply_exp_profile}, we have $g_1'(\overline{m}_1-) = -\infty$ and $g_2'(\overline{m}_2-) = -\infty$ which implies that $\frac{\partial}{\partial \alpha_1} \mathfrak{f}_{\infty}(\alpha_1, \gamma)$ has one-sided limits $-\infty$ and $+\infty$ at the endpoints of the interval $(\gamma+1-\overline{m_2}, \overline{m_1})$. This implies the existence and uniqueness of the point $(\alpha_1,\alpha_2)$ at which  the supremum in~\eqref{eq:proof_finite_free_add_profile_of_conv} is attained.

If $\gamma$ ranges in $(\eps,\overline{m}-\eps)$ for some $\eps>0$, then, since $\alpha_1 + \alpha_2 = \gamma + 1 \geq 1+\eps$ and since $\alpha_1 <1$, we have $\alpha_2 >\eps$. Similarly, $\alpha_2>\eps$. To prove that $\alpha_1$ stays separated from $\overline{m_1}$, let us argue by contradiction.  If $\alpha_1$ could be arbitrarily close to $\overline{m}_1$, then $g_1'(\alpha_1)+\log \alpha_1$ could be arbitrarily close to $-\infty$ (recall that $g_1'(\overline{m}_1 - ) = -\infty$).  This is a contradiction since $\alpha_2$ would then be close to $\gamma+1 - \overline{m}_1\in (\eps + 1-\overline{m}_1,\overline{m}_2-\eps)$ and thus bounded away from $0$ and $\overline{m_2}$, so that  $g_2'(\alpha_2) + \log \alpha_2$ would be bounded from below.

\end{proof}

\begin{proof}[Proof of Lemma~\ref{lem:exp_profile_finite_free_additive_conv}]
Fix some integers $j_1, j_2\in [\eps n, n]$ and set $\alpha_1:= j_1/n \in [\eps,1]$ and $\alpha_2:= j_2/n\in [\eps,1]$.  Put also $k= j_1+j_2-n$ and suppose that $\gamma := k/n = \alpha_1+\alpha_2-1\in [\eps, 1]$.  By the Stirling formula, uniformly in all such $j_1,j_2$ we have
\begin{align}\label{eq:ijkn}
\frac{j_1!j_2}{k! n!} = \eee^{o(n)} \cdot \frac{(\alpha_1 n /\eee)^{\alpha_1 n} (\alpha_2 n/\eee)^{\alpha_2 n}}{((\alpha_1 + \alpha_2 -1)n/\eee)^{(\alpha_1 + \alpha_2 -1)n} (n/\eee)^n}
=
\eee^{o(n)} \cdot \frac{\alpha_1^{\alpha_1 n} \alpha_2^{\alpha_2 n}}{(\alpha_1 + \alpha_2 -1)^{(\alpha_1 + \alpha_2-1)n}}.
\end{align}
If, additionally, $\alpha_1\in (\eta, \overline{m}_1-\eta)$ and $\alpha_2\in (\eta, \overline{m}_2-\eta)$ for some $\eta>0$, it follows from~\eqref{eq:proof_finite_free_add_profiles_12} that, again uniformly,
\begin{align}\label{eq:logabij}
\frac 1n \log \left (a_{j_1:n}^{(1)} a_{j_2:n}^{(2)} \frac{j_1! j_2!}{k! n!}\right) =  g_1(\alpha_1) + g_2(\alpha_2)  + \alpha_1 \log \alpha_1 + \alpha_2 \log \alpha_2 - \gamma \log \gamma + o(1).
\end{align}

\vspace*{2mm}
\noindent
\textit{Proof of the lower bound in~\eqref{eq:lem:exp_profile_finite_free_additive_conv}.}
Let an integer $k\in [\eps n, (\overline{m}-\eps) n]$ be given and put $\gamma=k/n\in (\eps, \overline{m}-\eps)$. Here and in the following, let $n$ be sufficiently large. Fix some $\delta>0$. By~\eqref{eq:proof_finite_free_add_profile_of_conv} and by the continuity of the involved functions, there exist $\alpha_1\in (0, \overline{m}_1)\cap n^{-1}\Z$ and $\alpha_2\in (0,\overline{m}_2)\cap n^{-1}\Z$ such that $\alpha_1+\alpha_2 - 1 = \gamma$ and
$$
g_1(\alpha_1) + \alpha_1 \log \alpha_1 + g_2(\alpha_2) +  \alpha_2 \log \alpha_2 - \gamma \log \gamma > g(\gamma) - \delta.
$$
Put $j_1 = \alpha_1 n$ and $j_2 = \alpha_2 n$. Then, $j_1+j_2 = n+k$ and it follows from~\eqref{eq:d_k:n} and~\eqref{eq:logabij} that
$$
\frac 1 n \log d_{k:n} \geq  \frac 1n \log \left(a_{j_1:n}^{(1)} a_{j_2:n}^{(2)} \frac{j_1! j_2!}{k! n!}\right)
=
g_1(\alpha_1) + g_2(\alpha_1)  + \alpha_1 \log \alpha_1 + \alpha_2 \log \alpha_2 - \gamma \log \gamma + o(1)
\geq g(\gamma) - \frac 12 \delta.
$$
The use of~\eqref{eq:d_k:n} and~\eqref{eq:logabij} is justified by the fact that  $\alpha_i\in (0, \overline{m}_i)$ stays bounded away from $0$ and $\overline{m}_i$, as we have shown above.

\vspace*{2mm}
\noindent
\textit{Proof of the upper bound in the first limit relation in~\eqref{eq:lem:exp_profile_finite_free_additive_conv}.} Our goal is to show that
$$
\limsup_{n\to\infty}\sup_{\eps n \leq k \leq (\overline{m} -\eps) n} \left(\frac 1n \log d_{k:n} - g\left(\frac k n\right)  \right)\leq 0.
$$
We argue by contradiction. Suppose that there is a sequence $(k_n)$ such that $k_n/n\in [\eps,\overline{m}-\eps]$ and
\begin{equation}\label{eq:lemma67_contra}
\frac 1n \log d_{k_n:n} - g\left(\frac{k_n}{n}\right)>\delta>0
\end{equation}
for infinitely many $n$. Extracting, if necessary, a subsequence we can and do assume that $k_n/n\to \gamma\in [\eps,\overline{m}-\eps]$ as $n\to\infty$. Define a sequence of functions $\mathfrak{f}_n:[k_n/n,1]\to\mathbb{R}$ by
$$
\mathfrak{f}_n\left(\frac{j_1}{n}\right):=\frac{1}{n}\log \left(a^{(1)}_{j_1:n}a^{(2)}_{n+k_n-j_1:n}\frac{j_{1}! (n+k_n-j_{1})!}{k_n!n!}\right),\quad j_1\in\{k_n,\ldots,n\},
$$
and linear interpolation elsewhere. This sequence of functions satisfies
$$
\lim_{n\to\infty}\mathfrak{f}_n(\alpha_1)=g_1(\alpha_1) + \alpha_1 \log \alpha_1 + g_2(\gamma + 1 - \alpha_1) + (\gamma + 1 - \alpha_1) \log (\gamma + 1 - \alpha_1)-\gamma \log \gamma=\mathfrak{f}_{\infty}(\alpha_1,\gamma),
$$
for all $\alpha_1\in (\gamma,1)\backslash{\{\overline{m}_1,1+\gamma-\overline{m}_2\}}$, where we define $g_i(t)=-\infty$, for $t\in (\overline{m}_i,1)$ and $i=1,2$.

Moreover, for every $n\in\mathbb{N}$ the function $\mathfrak{f}_n$ is concave. To prove this it suffices to check that the sequence
$$
a^{(1)}_{j_1:n}a^{(2)}_{n+k-j_1:n}\frac{j_{1}! (n+k-j_{1})!}{k!n!}=\frac{a^{(1)}_{j_1:n}}{\binom{n}{j_1}}\cdot\frac{a^{(2)}_{n+k-j_1:n}}{\binom{n}{n+k-j_1}}\cdot\binom{n-k}{n-j_1}\cdot\binom{n}{k},\quad j_1\in\{k,\ldots,n\},
$$
is log-concave, for every fixed $n\in\mathbb{N}$ and $k\in\{0,1,\ldots,n\}$. The first factor is log-concave by Newton's inequality, since all roots of $P_n^{(1)}$ are nonpositive. The second factor is log-concave by a similar reason. The log-concavity of the third factor is obvious and the fourth factor does not depend on $j_1$. Thus, $\mathfrak{f}_n$ is concave. By~\cite[Theorem 7.17]{Rockafellar+Wets:1998} this yields hypo-graphical convergence of $\mathfrak{f}_n$ to $\mathfrak{f}_{\infty}$. Using \cite[Theorem 7.33]{Rockafellar+Wets:1998} we conclude that
$$
\sup_{\alpha_1\in [k_n/n,1]}\mathfrak{f}_n(\alpha_1)~\ton~\sup_{\alpha_1\in (\gamma,1)}\mathfrak{f}_{\infty}(\alpha_1,\gamma)=g(\gamma).
$$
Therefore,
$$
\frac{1}{n}\log d_{k_n:n}\leq \frac{\log n}{n}+\max_{j_1\in\{k_n,\ldots,n\}}\mathfrak{f}_n\left(\frac{j_1}{n}\right)\leq \frac{\log n}{n}+\max_{\alpha_1\in [k_n/n,1]}\mathfrak{f}_n(\alpha_1)~\ton~g(\gamma).
$$
This contradicts~\eqref{eq:lemma67_contra} since $k_n/n\to\gamma$ and $g$ is continuous at $\gamma$.

\vspace*{2mm}
\noindent
\textit{Proof of the second relation in~\eqref{eq:lem:exp_profile_finite_free_additive_conv}.} Recall that we assume $P_n^{(1)}(1)=P_n^{(2)}(1)=1$, hence $a_{j_1:n}^{(1)}\leq 1$ and $a_{j_2:n}^{(2)} \leq 1$ for all $j_1,j_2\in \{0,\ldots,n\}$. If $k\geq (\overline{m}+\eps)n$, then $j_1\geq (\overline{m}_1+\eps/2)n$ or $j_2\geq (\overline{m}_2+\eps/2)n$. Therefore,
\begin{align*}
d_{k:n}&\leq \sum_{(\overline{m}_1+\eps/2)n\leq j_1\leq n} a_{j_1:n}^{(1)} \frac{j_1! j_2!}{k! n!}+\sum_{(\overline{m}_2+\eps/2)n\leq j_2\leq n} a_{j_2:n}^{(2)} \frac{j_1! j_2!}{k! n!}\\
&\leq \left(\max_{(\overline{m}_1+\eps/2)n\leq j_1\leq n} a_{j_1:n}^{(1)}\right)\sum_{(\overline{m}_1+\eps/2)n\leq j_1\leq n} \frac{j_1! j_2!}{k! n!}+\left(\max_{(\overline{m}_2+\eps/2)n\leq j_2\leq n} a_{j_2:n}^{(2)}\right)\sum_{(\overline{m}_2+\eps/2)n\leq j_2\leq n} \frac{j_1! j_2!}{k! n!}.
\end{align*}
Passing to logarithms and dividing by $n$ yields the second relation in~\eqref{eq:lem:exp_profile_finite_free_additive_conv} in view of the second claim in~\eqref{eq:proof_finite_free_add_profiles_12} and
$$
\limsup_{n\to\infty}\frac{1}{n}\log \sum_{(\overline{m}_i+\eps/2)n\leq j_i\leq n} \frac{j_i! (n+k-j_i)!}{k! n!}<+\infty,\quad i=1,2
$$
by the Stirling formula.
\end{proof}

\smallskip\noindent
\textsc{Step 3: Modified profiles and their inverses.}
It is convenient to introduce the \emph{modified exponential profiles}
$$
\tilde g(\gamma):= g(\gamma) + \gamma \log \gamma,
\qquad
\tilde g_1(\alpha_1):= g_1(\alpha_1) + \alpha_1 \log \alpha_1,
\qquad
\tilde g_2(\alpha_2):= g_2(\alpha_2) + \alpha_2 \log \alpha_2,
$$
where $\alpha_1\in (0, \overline{m}_1), \alpha_2 \in (0, \overline{m}_2), \gamma \in (0,\overline{m})$ with $\overline{m} = \overline{m}_1 + \overline{m}_2 - 1$.
Note that by Theorem~\ref{theo:exp_profile_implies_zeros}, $\tilde g_i$ is strictly concave and infinitely differentiable on $(0,\overline{m}_i)$. Also, $\tilde g_i'(\overline{m}_i-) = g_i'(\overline{m}_i-) + \log \overline{m}_i + 1 =   -\infty$. On the other hand, by~\eqref{eq:G_0_plus_as_limit} and recalling that $\mu_i$ is concentrated on $[-\infty,A]$ for some $A<0$, we have $\tilde g_i'(0+) < +\infty$ (even though we know that $g_i'(0+) = +\infty$). It follows that the inverse functions of the derivatives, denoted by $(\tilde g_i')^{\leftarrow}: (-\infty, \tilde g_i'(0+)) \to (0,\overline{m}_i)$ are well-defined, infinitely differentiable and strictly monotone decreasing.

\begin{lemma}\label{lem:inverse_modified_profiles}
For all $t\in (-\infty,\min(\tilde g_1'(0+),\tilde g_2'(0+)))$,
\begin{equation}\label{eq:inverse_der_profile_finite_free_add_conv}
(\tilde g')^{\leftarrow}(t) + 1 = (\tilde g_1')^{\leftarrow}(t) + (\tilde g_2')^{\leftarrow}(t).
\end{equation}
\end{lemma}

\begin{proof}
Lemma~\ref{lem:exp_profile_finite_free_additive_conv} states that
$$
\tilde g(\gamma) = \sup_{\substack{\alpha_1\in (0,\overline{m}_1),\; \alpha_2\in (0,\overline{m}_2)\\ \alpha_1 + \alpha_2 = \gamma + 1}} \left( \tilde g_1(\alpha_1) + \tilde g_2(\alpha_2)\right),
\qquad
\gamma\in  (0,\overline{m}).
$$
Writing $\Delta = \gamma + 1$, we have
$$
\tilde g(\Delta-1) = \sup_{\substack{\alpha_1\in (0,\overline{m}_1),\; \alpha_2\in (0,\overline{m}_2)\\ \alpha_1 + \alpha_2 = \Delta}} \left( \tilde g_1(\alpha) + \tilde g_2(\alpha_2)\right),
\qquad
\Delta \in (1,\overline{m}_1+\overline{m}_2).
$$
The function on the right-hand side is the sup-convolution of $\tilde g_1$ and $\tilde g_2$.  Since the sup-convolution is linearized by the Legendre transform, see \cite[Theorem 16.4]{rockafellar}, we conclude that the Legendre transform of $-\tilde g(\cdot - 1)$ is the sum of the Legendre transforms of $-\tilde g_1(\cdot)$ and $-\tilde g_2(\cdot)$. Taking the derivative we further infer that the derivative of the Legendre transform of $-\tilde g(\cdot - 1)$ is the sum of the derivatives of the Legendre transforms of  $-\tilde g_1(\cdot)$ and $-\tilde g_2(\cdot)$.
Now, the derivative of the Legendre transform and the derivative of the original function are inverse to each other. In particular, the derivative of the Legendre transform of $-\tilde g(\cdot - 1)$ is the inverse function of $-\tilde g'(\cdot - 1)$, that is the function $s\mapsto (\tilde g')^{\leftarrow}(-s) + 1$. This function is the sum of $(\tilde g_1')^{\leftarrow}(-s)$  and $(\tilde g_2')^{\leftarrow}(-s)$, which proves~\eqref{eq:inverse_der_profile_finite_free_add_conv}.
\end{proof}

\smallskip\noindent
\textsc{Step 4: Limiting distribution via $R$-transforms.} Assumption (iii) of Theorem~\ref{theo:finite_free_additive_conv_free_additive} implies that $\deg P_n^{(1)}+\deg P_n^{(2)}\geq n$ for all sufficiently large $n\in\N$. By Theorem~\ref{thm:real_rootedness}(i) $P_n^{(1)} \boxplus_n P_n^{(2)}$ is real-rooted and its roots are nonpositive (because its coefficients are nonnegative). By Lemma~\ref{lem:exp_profile_finite_free_additive_conv}, the sequence $P_n^{(1)} \boxplus_n P_n^{(2)}$ has a well-defined exponential profile $g$ in the sense of Definition~\ref{def:exp_profile}. Hence, we can apply Theorem~\ref{theo:exp_profile_implies_zeros} to conclude that, as $n\to\infty$, $\lsem P_n^{(1)} \boxplus_n P_n^{(2)}\rsem_n$ converges to some probability measure $\mu$ weakly on $[-\infty, 0]$.

Let us show that $R_{\mu}(t) = R_{\mu_1}(t) + R_{\mu_2}(t)$ for all sufficiently small $t>0$. Indeed, by formula~\eqref{eq:r_trans_via_profile} and Lemma~\ref{lem:inverse_modified_profiles}  we find
\begin{align*}
1+R_\mu (t)& = (g'(\cdot)+\log(\cdot))^{\leftarrow}(\log t)\\
&=(\tilde g' - 1)^{\leftarrow}(\log t)\\
&=(\tilde g')^{\leftarrow}(1+\log t)\\
&=
(\tilde g_1')^{\leftarrow}(1+\log t)+(\tilde g_{2}')^{\leftarrow}(1+\log t)-1\\
&=1+R_{\mu_1}(t)+R_{\mu_2}(t).
\end{align*}
This proves Proposition~\ref{prop:free_additive_convol_existence} assuming $A<0$. By definition, $\mu= \mu_1 \boxplus \mu_2$.  This proves Theorem~\ref{theo:finite_free_additive_conv_free_additive} under the additional assumption $A<0$.

\smallskip\noindent
\textsc{Step 5: Shift.}
Let now $A\geq 0$. We can first apply the above argument to the polynomials $P_n^{(1)}(x+c)$ and $P_n^{(2)}(x+c)$ with $c>A$ and then switch back to the original polynomials, which is possible since $\boxplus_n$ and $\boxplus$ commute with shifts.
\end{proof}

\subsection{Proof for the multiplicative free convolution}\label{subsec:proof_multipl_conv}
The main ingredient in the proof of Theorem~\ref{theo:finite_free_mult_conv} is the relation \eqref{eq:s_trans_via_profile} between the exponential profile of a sequence of polynomials and the $S$-transform of its limit root distribution.

\begin{proof}[Proof of Theorem~\ref{theo:finite_free_mult_conv} and Proposition~\ref{prop:free_multiplicative_convol_existence}] Again the proof is divided into steps.

\noindent
\textsc{Step 0: Preparation.}
We write  the polynomials as $Q_n^{(i)}(x)= \sum_{k=0}^n (-1)^{n-k} a_{k:n}^{(i)} x^k$, where $i\in \{1,2\}$ throughout the proof.
Recall that $\lsem Q_n^{(i)}\rsem_n$, the empirical distribution of roots of $Q_n^{(i)}$, is a probability measure on $[0,+\infty]$ and has an atom of weight $(n-\deg Q_{n}^{(i)})/n$ at $+\infty$. We shall check that condition~\eqref{eq:non-empty_intersection} implies that $Q_{n}^{(1)}$ is not divisible by $x^{\deg Q_n^{(2)}}$ and
$Q_{n}^{(2)}$ is not divisible by $x^{\deg Q_n^{(1)}}$. By Theorem~\ref{thm:real_rootedness}(ii) this yields that all roots of $Q_n:=Q_n^{(1)} \boxtimes_n Q_n^{(2)}$ are real and nonnegative. We argue by contradiction. Suppose that $Q_{n}^{(2)}$ is divisible by $x^{\deg Q_{n}^{(1)}}$. Then, by the Portmanteau Lemma
$$
\limsup_{n\to\infty}\frac{\deg Q_{n}^{(1)}}{n}\leq \limsup_{n\to\infty}\,\lsem Q_n^{(2)}\rsem_n(\{0\})\leq \nu_2(\{0\}).
$$
On the other hand,
$$
\liminf_{n\to\infty}\frac{\deg Q_{n}^{(1)}}{n}= \liminf_{n\to\infty}\,\lsem Q_n^{(1)}\rsem_n([0,+\infty))\geq 1-\nu_1(\{+\infty\}).
$$
Thus, $1-\nu_{1}(\{+\infty\})\leq \nu_2(\{0\})$ which contradicts~\eqref{eq:non-empty_intersection}. Similarly, if
$Q_{n}^{(1)}$ is divisible by $x^{\deg Q_{n}^{(2)}}$, then $1-\nu_{2}(\{+\infty\})\leq \nu_1(\{0\})$, which again contradicts~\eqref{eq:non-empty_intersection}.

It will be more convenient to work with the polynomials
$$
P_n^{(i)}(x):=
(-1)^n Q_n^{(i)}(-x) =\sum_{k=0}^n a_{k:n}^{(i)} x^k,\quad i\in \{1,2\},
$$
and
\begin{align}\label{eq:Q_n_d}
P_n(x)
&:= 
(-1)^n Q_n(-x)
=
\sum_{k = 0}^n d_{k:n} x^k,
\qquad
d_{k:n} = \frac{a_{k:n}^{(1)}a_{k:n}^{(2)}}{\binom nk},
\end{align}
since their empirical distributions of zeros are probability measures on $[-\infty, 0]$. After multiplication by nonzero numbers, we may and shall assume that $P_n^{(1)}(1) =  P_n^{(2)} (1) = 1$.
Then,  $a_{k:n}^{(1)} \geq 0$, $a_{k:n}^{(2)} \geq 0$ and $d_{k:n} \geq 0$.

\vspace*{2mm}
\noindent
\textsc{Step 1: Profiles of $P_n^{(1)}$ and $P_n^{(2)}$.}
Recall that we assume
that $\lsem Q_n^{(i)}\rsem_n \to \nu_i$ weakly as $n\to\infty$, where $\nu_i$ is a  probability measure on $[0,+\infty]$. It follows that $\lsem P_n^{(i)}\rsem_n \to \mu_i$, where $\mu_i(A) = \nu_i(-A)$ for all Borel sets $A\subset [-\infty,0]$.  Note that $\mu_i$ is a probability measure  on $[-\infty, 0]$ and define $\underline m_i := \mu_i(\{0\})$, $\overline{m}_i := 1-\mu_i(\{-\infty\})$. Also, put $\underline m:= \max(\underline m_1, \underline m_2)\in [0,1]$ and $\overline m:= \min(\overline m_1, \overline m_2)\in [0,1]$. By formula~\eqref{eq:non-empty_intersection}, the assumption $(\underline{m}_1, \overline{m}_1)\cap (\underline{m}_2, \overline{m}_2)\neq \varnothing$ we made in Theorem~\ref{theo:finite_free_mult_conv} is equivalent to  $\underline m<\overline m$. Since all roots of $P_n^{(i)}$ are nonpositive and $P_n^{(1)}(1) =  P_n^{(2)} (1) = 1$,  Theorem~\ref{theo:zeros_imply_exp_profile} gives
\begin{equation}\label{eq:Q_profile_proof_mult}
\sup_{(\underline{m}_i+\eps) n \leq k  \leq (\overline{m}_i-\eps)n}
\left|\frac 1n  \log a_{k :n}^{(i)} -g_i\left(\frac k  n\right)  \right| \ton 0,
\end{equation}
for all $\eps>0$ and for  suitable infinitely differentiable, strictly concave functions $g_i:(\underline m_i,\overline m_i) \to \R$. By the same theorem,
\begin{equation}\label{eq:Q_profile_proof_mult_minus_infty}
\sup_{0\leq k  \leq (\underline{m}_i-\eps)n} \frac 1n   \log  a_{k :n}^{(i)} \ton  -\infty,
\qquad
\sup_{(\overline{m}_i+\eps)n \leq k  \leq n} \frac 1n   \log  a_{k :n}^{(i)} \ton  -\infty.
\end{equation}

\vspace*{2mm}
\noindent
\textsc{Step 2: Profile of $P_n$.}
Knowing the exponential profiles of $P_n^{(1)}$ and $P_n^{(2)}$ we can derive from~\eqref{eq:Q_n_d} the exponential profile of $P_n$.
We claim that for every $\eps>0$,
\begin{equation}\label{eq:exp_profile_finite_free_mult_conv}
\sup_{(\underline m+\eps) n \leq k \leq (\overline m-\eps) n} \left|\frac 1n  \log d_{k:n} - g\left(\frac k n\right)\right| \ton  0.
\end{equation}
and
\begin{equation}\label{eq:exp_profile_finite_free_mult_conv2}
\sup_{0\leq k  \leq (\underline{m}-\eps)n} \frac 1n   \log  d_{k :n} \ton  -\infty,
\qquad
\sup_{(\overline{m}+\eps)n \leq k  \leq n} \frac 1n   \log  d_{k :n} \ton  -\infty,
\end{equation}
where
\begin{equation}\label{eq:exp_profile_finite_free_mult_conv_formula}
g(\alpha) = g_1(\alpha) + g_2(\alpha) + \alpha \log \alpha + (1-\alpha) \log (1-\alpha),
\qquad \alpha \in (\underline m ,\overline m).
\end{equation}
\begin{proof}
Let $H(\alpha): = -\alpha \log \alpha - (1-\alpha) \log (1-\alpha)\geq 0$, with $0< \alpha < 1$,  be the entropy function.
By the Stirling formula,
\begin{equation}\label{eq:profile_binomial_coeff}
\sup_{\eps n \leq k \leq (1-\eps) n} \left|\frac 1n  \log \binom n k - H\left(\frac k n\right)\right| \ton 0,
\end{equation}
for every $\eps>0$.
Taking the logarithm of $d_{k:n}$ in~\eqref{eq:Q_n_d} and applying~\eqref{eq:Q_profile_proof_mult}  and~\eqref{eq:profile_binomial_coeff} gives
$$
\sup_{(\underline m+\eps) n \leq k \leq (\overline m-\eps) n} \left|\frac 1n  \log d_{k:n} - g_1\left(\frac k n\right)- g_2\left(\frac k n\right) - H\left(\frac k n\right)\right| \ton  0.
$$
This  proves~\eqref{eq:exp_profile_finite_free_mult_conv}. To prove~\eqref{eq:exp_profile_finite_free_mult_conv2}, assume without loss of generality that $\underline m_2 \leq \underline m_1$. Then, $\underline m=\underline m_1$.
Since $P_n^{(2)}(1) = 1$ implies  $a_{k:n}^{(2)}\leq 1$, we deduce from~\eqref{eq:Q_n_d} that $d_{k:n} \leq a_{k:n}^{(1)}$, for all $k \in \{0,\ldots, n\}$.  It follows from~\eqref{eq:Q_profile_proof_mult_minus_infty} that
$$
\sup_{0 \leq k \leq (\underline m_1-\eps) n} \frac 1n \log d_{k:n} \leq  \sup_{0 \leq k \leq (\underline m_1-\eps) n} \frac 1n \log a_{k:n}^{(1)} \ton -\infty.
$$
This proves the first claim of~\eqref{eq:exp_profile_finite_free_mult_conv2}, the second one being analogous.
\end{proof}

\vspace*{2mm}
\noindent
\textsc{Step 3: Limiting distribution via $S$-transforms.} Applying formula~\eqref{eq:s_trans_via_profile} from Lemma~\ref{lem:r_and_s_transform_vs_profile} twice, we obtain that the $S$-transforms of $\nu_1$ and $\nu_2$, denoted by $S_{\nu_1}$ and $S_{\nu_2}$, satisfy
\begin{equation}\label{eq:S_trasnf_as_profile_i}
S_{\nu_i}(t) =  - \frac{1+t}{t}  \eee^{g_i'(1+t)},
\qquad
t\in (\underline m_i-1,\overline m _i-1),
\qquad
i\in \{1,2\}.
\end{equation}
Recall that all roots of $P_n=P_n^{(1)} \boxtimes_n P_n^{(2)}$ are nonnegative. Also, we know from~\eqref{eq:exp_profile_finite_free_mult_conv} and~\eqref{eq:exp_profile_finite_free_mult_conv_formula} that $P_n$ possesses an exponential profile $g$ with
\begin{equation}\label{eq:proof_mult_free_g_prime}
g'(\alpha) = g_1'(\alpha) + g_2'(\alpha) + \log \alpha + \log (1-\alpha),
\qquad \alpha \in (\underline m ,\overline m).
\end{equation}
Thus, we may apply Theorem~\ref{theo:exp_profile_implies_zeros} to conclude that $\lsem P_n\rsem_n$ converges weakly to some probability measure $\mu$ on $[-\infty,0]$, as $n\to\infty$. Consider a probability measure $\nu:=\mu_+$ on $[0,+\infty]$ defined by  $\nu(A) = \mu(-A)$ for every Borel set $A$. Then, using equation~\eqref{eq:s_trans_via_profile} of Lemma~\ref{lem:r_and_s_transform_vs_profile} and then equations~\eqref{eq:proof_mult_free_g_prime} and~\eqref{eq:S_trasnf_as_profile_i}, we get
\begin{align}
S_{\nu}(t)
&=
-\frac{t+1}{t}\eee^{g'(t+1)}
=
-\frac{t+1}{t}\eee^{g_1(t+1) + g_2(t+1) + \log (t+1) + \log (-t)} \notag\\
&=
\frac{(t+1)^2}{t^2}\eee^{g_1(t+1) + g_2(t+1)}
=
S_{\nu_1}(t) S_{\nu_2}(t)
,
\quad t\in 
(\underline{m}-1,\overline{m}-1).  \label{eq:proof_mult_s_transf_product}
\end{align}
To summarize, we showed that $\lsem Q_n^{(1)} \boxtimes_n Q_n^{(2)}\rsem_n\to\nu$, where $\nu$ is a probability measure on $[0, +\infty]$ whose $S$-transform satisfies~\eqref{eq:proof_mult_s_transf_product}.
This proves Proposition~\ref{prop:free_multiplicative_convol_existence} and Theorem~\ref{theo:finite_free_mult_conv}.
\end{proof}

\begin{remark}
It is apparent from~\eqref{eq:proof_mult_s_transf_product} that the multiplicativity property of the $S$-transform crucially depends  on the logarithmic terms in the profile $g$  as in \eqref{eq:exp_profile_finite_free_mult_conv_formula}, which come from the binomial coefficient in Definition~\ref{def:finite_free_convolution} of $\boxtimes_n$. This is in contrast to the findings of \cite[Proposition 5.1]{COR23} and \cite[\S 6.2]{diff-paper} in the setting of isotropic distributions on $\C$, where the Hadamard product leads to the finite version of the free isotropic multiplicative convolution $\otimes$. To be slightly more precise, isotropic distributions on $\C$ are determined by the $S$-transforms of their radial part, which multiply when their inverse radial CDF (quantile functions) multiply, which is when profiles add (without any additional logarithmic term). On the level of free probability, this corresponds to the fact that the distribution of the product of free $R$-diagonal elements is determined by the free multiplicative convolution of their positive parts. Hence, the natural finite counterpart of $\otimes$ is given by the Hadamard product, which is the operation that adds profiles.
Moreover, the usual definition of $\boxtimes_n$ is unnatural for random polynomials with independent coefficients, since the resulting profile~\eqref{eq:exp_profile_finite_free_mult_conv_formula} could become non-concave due to the additional logarithmic term.
\end{remark}

\subsection{Proof of Proposition~\ref{prop:char_profiles}: Characterization of profiles}\label{subsec:proof_char_profiles}
We shall first show that if $g$ is a profile of some sequence $(P_n(x))_{n\in \N}$ of polynomials with only nonpositive roots, then $g$ satisfies (i) and (ii). If $g$ is a profile, then Theorem~\ref{theo:exp_profile_implies_zeros} implies that $\lsem P_n\rsem_n$ converges weakly to some probability measure $\mu$ on $[-\infty, 0]$ such that $\mu(\{0\})=\underline{m}$ and $\mu(\{-\infty\})=1-\overline{m}$. By Equation~\eqref{eq:summary_psi} of Lemma~\ref{lem:r_and_s_transform_vs_profile} we have $\psi_{\mu_{+}} (s) = (-\eee^{g'})^{\leftarrow} (s)-1$, $s<0$.
Observe that the assumption $(\underline{m},\overline{m})\neq\varnothing$ implies $\overline{m}>0$.
Let $\tilde \mu$ be a probability measure on $(-\infty, 0]$ defined by $\tilde \mu(A) = \mu(A)/\overline{m}$ for every Borel set $A\subseteq (-\infty, 0]$. Consider the reflected probability measure $\tilde \mu_+$ on $[0,+\infty)$ given by $\tilde \mu_+(A) := \tilde \mu (-A)$, for every Borel set $A\subset [0,+\infty)$. Then, by~\eqref{eq:psi_mu_def},
\begin{equation}\label{eq:psi_tilde_mu_+_from_char_profiles_proof}
\psi_{\tilde \mu_{+}} (s) = \frac 1 {\overline {m}} \int_{(0,+\infty)} \frac{zt}{1-zt} \dd \mu_+(z) = \frac{\psi_{\mu_{+}} (s) + 1-\overline{m}} {\overline{m}}  = \frac{(-\eee^{g'})^{\leftarrow} (s)-\overline{m}}{\overline{m}}, \qquad  s<0.
\end{equation}
Also, $\tilde \mu_+(\{0\}) = \underline{m} / \overline{m} < 1$.  Let $\C^+= \{z\in \C: \Im z >0\}$ be the upper half-plane and $\C^-= \{z\in \C: \Im z <0\}$ the lower half-plane. Then,  $\ii \C^+ = \{z\in \C: \Re z <0\}$ is the left half-plane. \citet[Proposition~6.2]{bercovici_voiculescu} showed that the analytic function $\psi_{\tilde \mu_+}$ is a bijection between  $\ii \C^+$ and the domain $\Omega_{\tilde \mu_+}:= \psi_{\tilde \mu_+}(\ii \C^+)$ which contains the interval $(\underline{m} / \overline{m}-1,0)$ and is contained in the disk having this interval as a diameter.  The $\chi$-transform of $\tilde \mu_+$, denoted by $\chi_{\tilde \mu^+}: \Omega_{\tilde \mu_+} \to \ii \C^+$, is defined to be the inverse function of $\psi_{\tilde \mu_+}$. It follows from~\eqref{eq:psi_tilde_mu_+_from_char_profiles_proof} that
$$
\chi_{\tilde \mu_+}(z) = -\eee^{g'((1 + z)\overline{m})}, \quad  z\in (\underline{m} / \overline{m}-1,0),
\qquad
\eee^{g'(y)} = - \chi_{\tilde \mu_+}\left(\frac{y}{\overline{m}}-1\right),
\quad
y\in (\underline{m},\overline{m}).
$$
Since $\chi_{\tilde \mu_+}:\Omega_{\tilde \mu_+} \to \ii \C^+$ is an analytic bijection defined on  $\Omega_{\tilde \mu_+} \supseteq (\underline{m} / \overline{m}-1,0)$,  we can extend $\eee^{g'}$ to an analytic bijection between the domain $\cD:= \{(1 + z)\overline {m}: z\in \Omega_{\tilde \mu_+}\}\supseteq (\underline{m},\overline{m})$ and the right half-plane $-i\C^+$.
It remains to prove that the matrix~\eqref{eq:nevanlinna_pick_matrix_proof_char_profiles} is positive definite. It follows from the definition of the Cauchyt transform that  $s\mapsto G_{\mu}(1/s)$ is a map from $\C^+$ to $\C^+$. By the Nevanlinna--Pick theorem (see Theorem~3.3.3 on p.~105 in~\cite{akhiezer_book}), for every $s_1,\ldots, s_\ell\in \C^+$, the matrix
\begin{equation}\label{eq:nevanlinna_pick_matrix_proof_char_profiles_1}
\left(\frac{G(1/s_j) - \overline{G(1/s_k)}}{s_j - \overline{s_k}}\right)_{j,k=1}^\ell
\end{equation}
is positive semi-definite. Now, we take $s_j:= \eee^{g'(\overline{y_j})}$, where $y_1,\ldots, y_\ell \in \cD$ satisfy $\Re y_1>0,\ldots, \Re y_\ell>0$ and observe that $s_j\in \C^+$. Indeed, since $\eee^{g'}$ is real on $(\underline{m}, \overline{m})$, it maps $\cD \cap \C^-$ either to a subset of $\C^+$ or to a subset of $\C^-$, but since $\eee^{g'}$ is decreasing on $(\underline{m}, \overline{m})$, the former possibility occurs. By the identity $G(1/s_j) = G(\eee^{-g'(\overline{y_j})}) = \eee^{g'(\overline{y_j})} \eee^{-g'(\overline{y_j})} G(\eee^{-g'(\overline{y_j})}) = \eee^{g'(\overline{y_j})} \overline{y_j}$, the matrices~\eqref{eq:nevanlinna_pick_matrix_proof_char_profiles_1} and~\eqref{eq:nevanlinna_pick_matrix_proof_char_profiles} are equal, up to conjugation, and the proof is complete.

Conversely, let $g:(\underline{m}, \overline{m})\to \R$ satisfy~(i) and~(ii). Since $y\mapsto \eee^{-g'(y)}$ defines an analytic bijection between $\cD$ and $\{z\in \C: \Re z >0\}$, we can define an analytic function $G: \{z\in \C: \Re z >0\} \to \C$ by $G(\eee^{-g'(y)}) = \eee^{g'(y)} y$, $y\in \cD$. By the Nevanlinna--Pick theorem, the positive semi-definite property of the matrix~\eqref{eq:nevanlinna_pick_matrix_proof_char_profiles} implies that the function $s\mapsto G(1/s)$, defined originally for $\Re s>0$,
$\Im s>0$, can be extended to an analytic function defined on the upper half-plane $\C^+$ and mapping $\C^+$ to $\C^+$. We can also extend this function to $\bC^-$ by conjugation. Altogether, this defines an analytic function $G: \C \backslash(-\infty, 0] \to \C$ such that $G(\C^+) \subset \C^-$ and $G$ is positive on $(0,\infty]$. By~\cite[p.~127]{akhiezer_book}, $G$ has a representation $G(t) = \alpha_0 + \int_{(-\infty, 0]} \frac{\tau(\dd u)}{t-u}$ for some $\alpha_0\geq 0$ and a (possibly infinite) measure $\tau$ on $(-\infty,0]$ with $\int_{(-\infty,0]} \frac{\tau(\dd u)}{1+|u|} <\infty$. Recalling that  $G(\eee^{-g'(y)}) = \eee^{g'(y)} y$, $y\in (\underline{m}, \overline{m})$, and letting $y$ approach the boundary points of this interval, we get $G(t) \sim \overline{m}/t$ as $t\to + \infty$ and $G(t) \sim \underline{m}/t$ as $t\to +0$. This implies that $\alpha_0=0$,  $\tau((-\infty, 0]) = \overline{m}$ and $\tau (\{0\}) = \underline{m}$. Defining $\mu = \tau  +  (1-\overline{m}) \delta_{-\infty}$, we obtain that $G$ is the Cauchy transform of the probability measure $\mu$ on $[-\infty, 0]$. Also, by construction, $t\mapsto t G(t)$, $t>0$,  is the inverse of $y\mapsto \eee^{-g'(y)}$, $y\in (\underline{m}, \overline{m})$. Let now $(P_n(x))_{n\in \N}$ be any sequence of nonpositive-rooted polynomials with $\lsem P_n\rsem_n \to \mu$ as $n\to\infty$ and $P_n(1) = 1$. By Theorem~\ref{theo:zeros_imply_exp_profile}, $g$ is the profile of this sequence.

\section{Proof of the results on the repeated differentiation}\label{sec:proof_repeated}
\subsection{Proof of Lemma~\ref{lem:q_n_m_roots_real}.}
By Rolle's theorem, if $P$ is a real-rooted polynomial, then $\mathcal{A}_{a,b}P$ is also real-rooted. Hence, $T_{n,\ell}^{(a,b)}(z) = z^{-\ell\Delta}\mathcal{A}_{a,b}^\ell ((z-1)^n)$ is real-rooted. Since
$(-1)^n T_{n,\ell}^{(a,b)}(-z)$ has only nonnegative coefficients, all roots of $T_{n,\ell}^{(a,b)}$ must be
nonnegative. The claim on the multiplicity of a root at $0$ follows from~\eqref{eq:q_n_m_definition} upon noticing that the smallest index $j$ in the set $\mathcal{I}_{\ell}$ is equal to $b$ is $\Delta \geq 0$ or $a-\ell \Delta$ if $\Delta \leq 0$.

\subsection{Proof of Theorem~\ref{thm:q_n_m_measures_converge}}
We apply Theorem~\ref{theo:exp_profile_implies_zeros} to the polynomials $P_n(z):=(-1)^n T_{n,\ell}^{(a,b)}(-z)$ which have only nonpositive roots by Lemma~\ref{lem:q_n_m_roots_real}. The exponential profile of $(P_n)_{n\in \N}$ can be calculated as follows. If $\alpha\in (0,1)$ and $\alpha+\Delta \kappa<0$, then $[z^{\lfloor\alpha n\rfloor}]T_{n,\ell}^{(a,b)}(z)=0$ for all large enough $n$ by the definition of $T_{n,\ell}^{(a,b)}$. On the other hand, the assumption $1+\Delta \kappa>0$ implies that $\mathcal{I}_{\Delta,\kappa}:=(0,1)\cap (-\Delta \kappa,+\infty)\neq \varnothing$. Thus, for every $\alpha\in\mathcal{I}_{\Delta,\kappa}$, we have
\begin{align}
g(\alpha)&:=\lim_{n\to\infty}\frac{1}{n}\log\left(\binom{n}{\lfloor n\alpha\rfloor}
\frac{1}{n^{\ell b}} \frac{\lfloor n\alpha\rfloor!}{(\lfloor n\alpha\rfloor-b)!} \frac{(\lfloor n\alpha\rfloor+\Delta)!}{(\lfloor n\alpha\rfloor-b+\Delta)!}\ldots \frac{(\lfloor n\alpha\rfloor+(\ell-1)\Delta)!}{(\lfloor n\alpha\rfloor-b + (\ell-1)\Delta)!}\right)\notag\\
&=-\alpha\log \alpha-(1-\alpha)\log(1-\alpha)-\lim_{n\to\infty}\left(\ell b\frac{\log n}{n}-\frac 1n \sum_{j=0}^{\ell-1}\log\frac{(\lfloor n\alpha\rfloor+j\Delta)!}{(\lfloor n\alpha\rfloor-b+j\Delta)!}\right)\notag\\
&=-\alpha\log \alpha-(1-\alpha)\log(1-\alpha)+\lim_{n\to\infty}\sum_{k=0}^{b-1}\left[\frac{1}{n}\sum_{j=0}^{\ell-1}\log \frac{\lfloor n\alpha\rfloor+j\Delta-k}{n}\right].\label{eq:calc_rep_diff}
\end{align}
For every fixed $k\in\{0,\ldots,b-1\}$, the term in the square brackets converges to the same limit
$\int_0^{\kappa} \log(\alpha+\Delta s)\dd s$, as $n\to\infty$. This yields
$$
g(\alpha)=-\alpha\log \alpha-(1-\alpha)\log(1-\alpha)+b\int_0^{\kappa} \log(\alpha+\Delta s)\dd s,\quad \alpha\in \mathcal{I}_{\Delta,\kappa},
$$
and thereupon
$$
\eee^{-g^{\prime}(\alpha)}=
\begin{cases}
\frac{\alpha^{a/\Delta}}{(1-\alpha)(\alpha+\Delta \kappa)^{b/\Delta}},&\Delta\neq 0,\\
\frac{\alpha}{1-\alpha}\eee^{-b\alpha^{-1}\kappa},&\Delta=0,
\end{cases}
\quad \alpha\in \mathcal{I}_{\Delta,\kappa}.
$$
By Theorem~\ref{theo:exp_profile_implies_zeros} applied with $(\underline{m},\overline{m})=\mathcal{I}_{\Delta,\kappa}$, this implies the existence of the weak limit $\nu_{a,b;\kappa}$ of $\lsem T_{n,\ell}^{(a,b)}(z)\rsem_n$ on $[0,\infty)$ such that $\nu_{a,b;\kappa}(\{0\})=\max(0,-\Delta \kappa)$.
By Lemma \ref{lem:r_and_s_transform_vs_profile}, the $S$-transform of $\nu_{a,b;\kappa}$ is given by
\begin{align}\label{eq:S-trafo-nuab}
S_{a,b;\kappa}(t)=\begin{cases}
\left(1+\frac{\Delta \kappa}{t+1}\right)^{\frac{b}{\Delta}}
,&\Delta> 0,\\
\left(\frac{t+1}{1+t+\Delta \kappa}\right)^{\frac{b}{-\Delta}}
,&\Delta< 0,\\
\eee^{\frac{b\kappa}{t+1}},&\Delta=0,
\end{cases}
\quad t+1\in \mathcal{I}_{\Delta,\kappa}.
\end{align}
Note that the analytic expression in the cases $\Delta>0$ and $\Delta<0$ is actually the same.
For $\Delta=0$, $\eee^{\frac{b\kappa}{t+1}}$ is the $S$-transform of an infinitely $\boxtimes$-divisible distribution called `free multiplicative Poisson', see (as a limit of) \cite[Lemma 7.2]{bercovici_voiculescu_levy_hincin} and \cite[Theorem 2.2]{Arizmendi_Hasebe}.
For $\Delta<0$, the $S$-transform is a $\frac b {-\Delta}>1$ power of another function, hence $\nu_{a,b;\kappa}$ is the $\boxtimes {\frac b {-\Delta}}$ free self-convolution of another measure with $S$-transform $\frac{t+1}{1+t+\Delta \kappa}$, see \cite[Theorem 2.6]{BB05_semigroup} for existence of the free multiplicative self-convolution. This measure is $-\Delta\kappa\delta_0+(1+\Delta\kappa)\delta_1$, which one may easily check or look up at \cite[Lemma 4.1]{bercovici_voiculescu_levy_hincin}. For $\Delta>0$, the function $v(z)=\log(S_{a,b;\kappa})(z)$, $z\not\in (-\infty,-1]\cup[0,\infty)$, satisfies the assumptions of \cite[Theorem 7.5 (1)]{bercovici_voiculescu_levy_hincin} and hence is $\boxtimes$-infinitely divisible.

In the case $\Delta\neq 0$ and arbitrary $a,b\in\N_0$, there seems to be no simple closed form for the inverse of the algebraic function $\alpha\mapsto \eee^{-g^{\prime}(\alpha)}$. However, the inversion is possible if $\Delta=0$ and also in the important case $a=0, b=1$, which corresponds to repeated differentiation.

\noindent
{\sc Case $a=0$ and $b=1$.} In this case $\eee^{-g^{\prime}(\alpha)}=(\alpha-\kappa)/(1-\alpha)$. Formula~\eqref{eq:S-trafo-nuab} implies that $S_{0,1;\kappa}(t)=\frac{t+1}{1+t-\kappa}$. Hence $\nu_{0,1;\kappa}=\kappa\delta_0+(1-\kappa)\delta_1$.

\noindent
{\sc Case $\Delta=0$.} In this case, the function $\alpha\mapsto \frac{\alpha}{1-\alpha}\eee^{-b\alpha^{-1}\kappa}$ can be inverted in terms of the function $W_0$ using elementary manipulations. This yields a formula for the Cauchy transform of the reflected measure $\nu_{a,a;\kappa}(-(\cdot))$. Then formula~\eqref{eq:theo:repeated_action_stieltjes} follows upon replacing $t$ by $-t$ and changing the signs. The domain is determined from the constraint $a\kappa \eee^{-a\kappa}/t\notin (-\infty,-1/\eee]$ ensuring analyticity of $W_0(a\kappa \eee^{-a\kappa}/t)$. Formula~\eqref{eq:theo:repeated_action_density} follows  by the Stieltjes--Perron inversion, whereas the mass of the atom at $1$ is a consequence of the obvious fact that the multiplicity of the root $1$ of $T_{n,m}^{(a,b)}(z)$ is exactly $\max(n-am,0)$.

\subsection{Proof of Proposition~\ref{eq:prop_rep_diff_stiel}}
We shall apply Theorem~\ref{theo:repeated_diff_general} to the polynomials
\begin{equation}\label{eq:Stirling_polynomials1}
S_n(x):=x(x+1)\cdots(x+n-1)=\sum_{k=1}^{n}\stirling{n}{k}x^k,
\end{equation}
where $\stirling{n}{k}$ denote the (signless) Stirling numbers of the first kind.
Obviously, the sequence of probability measures $\lsem S_n(nx)\rsem_n$ converges weakly to the uniform distribution $\mu_S=\unif_{[-1,0]}$ on $[-1,0]$. By the classic logarithmic asymptotic
of the Stirling numbers of the first kind, see, for example \cite[Eq.~(5.7)]{moser_wyman}, we have
$$
g_S(\alpha):=\lim_{n\to\infty}\frac{1}{n}\log \left(\stirling{n}{\lfloor \alpha n\rfloor} n^{\lfloor \alpha n\rfloor-n}\right)=-1+\alpha+(1-\alpha)\log \alpha + w_S+(\alpha-1)\log w_S,\quad \alpha\in (0,1),
$$
where $w_S=w_S(\alpha)$ is a unique solution in $(0,\infty)$ of the equation
\begin{equation}\label{eq:w_0_define}
\frac{w_S}{\eee^{w_S}-1}=\alpha.
\end{equation}
We first calculate the exponential profile $g_{S;\kappa}(\alpha)$ of $n^{-n}(\tfrac{\dd}{\dd z})^{\lfloor \kappa n\rfloor}S_n(nz)$. Using similar calculation as in~\eqref{eq:calc_rep_diff}, we obtain
$$
g_{S;\kappa}(\alpha)=g_{S}(\alpha+\kappa)+\int_{0}^{\kappa}\log(\alpha+s)\dd s,\quad \alpha\in (0,1-\kappa).
$$
Therefore,
\begin{multline*}
\eee^{-g^{\prime}_{S;\kappa}(\alpha)}=\eee^{-g^{\prime}_{S}(\alpha+\kappa)}\frac{\alpha}{\alpha+\kappa}=\frac{\alpha}{w_S(\alpha+\kappa)}=\frac{\alpha+\kappa}{w_S(\alpha+\kappa)}-\frac{\kappa}{w_S(\alpha+\kappa)}\\
\overset{\eqref{eq:w_0_define}}{=}\frac{1}{\eee^{w_S(\alpha+\kappa)}-1}-\frac{\kappa}{w_S(\alpha+\kappa)}=:\widehat{Y}_{\kappa}(w_S(\alpha+\kappa)),\quad \alpha\in (0,1-\kappa),
\end{multline*}
where we have set $\widehat{Y}_{\kappa}(t)=(\eee^t-1)^{-1}-\kappa/t$. Thus, the Cauchy transform of the push-forward of $\unif_{[0,1]}\boxtimes \widehat{\nu}_{0,1;\kappa}$ under $z\mapsto -z$ is
$$
t\mapsto \frac{1}{t}\left(\widehat{Y}_{\kappa}(w_S(\cdot+\kappa))\right)^{\leftarrow}(t)=\frac{1}{t}(w_S^{\leftarrow}(\widehat{Y}_{\kappa}^{\leftarrow}(t))-\kappa)\overset{\eqref{eq:w_0_define}}{=}\frac{1}{t}\left(\frac{\widehat{Y}_{\kappa}^{\leftarrow}(t)}{\eee^{\widehat{Y}_{\kappa}^{\leftarrow}(t)}-1}-\kappa\right)=\widehat{Y}_{\kappa}^{\leftarrow}(t),
$$
where we have used $t=\widehat{Y}_{\kappa}(\widehat{Y}_{\kappa}^{\leftarrow}(t))=(\eee^{\widehat{Y}_{\kappa}^{\leftarrow}(t)}-1)^{-1}-\kappa/\widehat{Y}_{\kappa}^{\leftarrow}(t)$ for the last equality.  In other words, we have shown that
$$
\mathbb{E}\left(\frac{1}{t+\xi_{\kappa}}\right)=\widehat{Y}_{\kappa}^{\leftarrow}(t).
$$
By definition of $\mu_{\kappa}$,
\begin{multline*}
G_{\mu_{\kappa}}(t)=\mathbb{E}\left(\frac{1}{t-(2\xi_{\kappa}-1)}\right)=-\frac{1}{2}\mathbb{E}\left(\frac{1}{\xi_{\kappa}-(t+1)/2}\right)\\
=-\frac{1}{2}\widehat{Y}_{\kappa}^{\leftarrow}(-(t+1)/2)=(-2\widehat{Y}_{\kappa}(-2(\cdot))-1)^{\leftarrow}(t)=Y^{\leftarrow}_{\kappa}(t).
\end{multline*}
The points $\pm z_{\kappa}$ are critical points of the function $t\mapsto Y_{\kappa}(t)$, for every fixed $\kappa\in (0,1)$, see Figure~\ref{fig:coth}. They are defined by the equation
$$
Y_{\kappa}^{\prime}(z)=-\frac{1}{\sinh^2(z)}+\frac{\kappa}{z^2}=0,
$$
which is equivalent to~\eqref{eq:sinh_equation}. Thus, $G_{\mu_{\kappa}}$ cannot be analytically continued to $[Y_{\kappa}(-z_{\kappa}),Y_{\kappa}(z_{\kappa})]$ which means that $\mu_{\kappa}$ is supported by $[Y_{\kappa}(-z_{\kappa}),Y_{\kappa}(z_{\kappa})]$.

\begin{figure}
\includegraphics[scale=0.5]{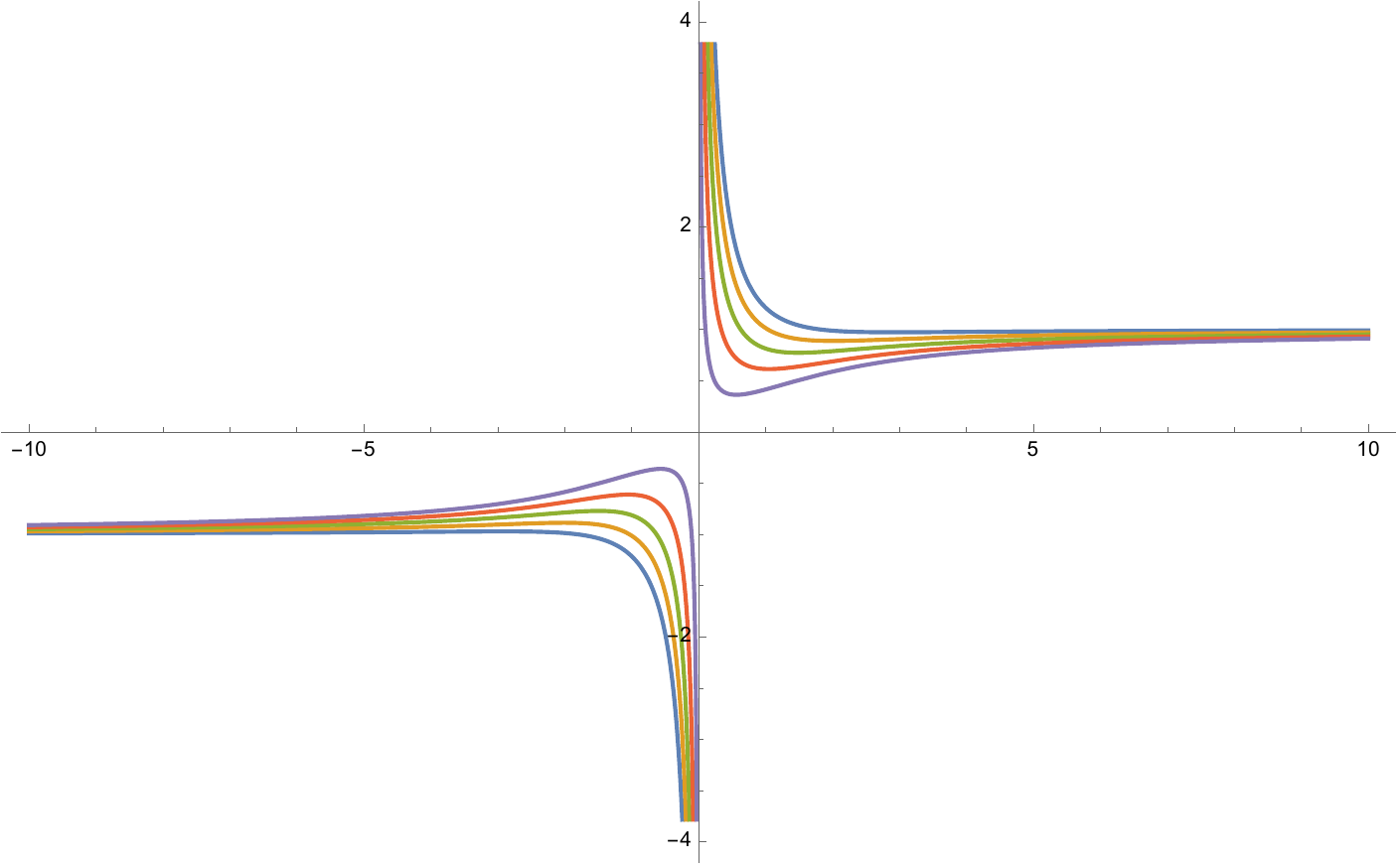}
\caption{Graphs of the functions $t\mapsto Y_{\kappa}(t)=\coth(t)-\kappa/t$ for $\kappa=0.1,0.3,0.5,0.7,0.9$.}
\label{fig:coth}
\end{figure}

\subsection{Additive free convolution for measures with atoms at \texorpdfstring{$-\infty$}{-infinity}}
The next proposition expresses $\mu_1\boxplus \mu_2$ for measures which may have atoms at $-\infty$ via the usual operation $\boxplus$ and the free self-convolution $\mu^{\boxplus \alpha}$, $\alpha>1$; see~\cite[p.228--233]{nica_speicher_book} and ~\cite[Section~1.2]{shlyakhtenko_tao} for its definition and interpretation as a compression by a free projection.
\begin{proposition}\label{prop:free_add_conv_with_infty}
Let $\mu_1$ and $\mu_2$ be probability measures on $[-\infty, A]$ with $A\in \R$, $\overline{m}_i = 1 - \mu_i(\{-\infty\})$, $i=1,2$, and suppose that $\overline{m}_1 + \overline{m}_2 >1$. Write $\mu_i = (1-\overline{m}_i) \delta_{-\infty} + \overline{m}_i \widetilde \mu_i$, $i=1,2$, where $\widetilde \mu_i$ is a probability measure on $(-\infty, A]$ (without an atom at $-\infty$). Then,
$$
\mu_1 \boxplus \mu_2 = (\overline{m}_1 + \overline{m}_2 - 1) \left(\widetilde \mu_1^{\boxplus \frac{\overline{m}_1}{\overline{m}_1 + \overline{m}_2-1}}\left(\frac{\overline{m}_1}{\overline{m}_1 + \overline{m}_2 - 1}\,\, (\cdot) \right) \boxplus \widetilde \mu_2^{\boxplus \frac{\overline{m}_2}{\overline{m}_1 + \overline{m}_2-1}}\left(\frac{\overline{m}_2}{\overline{m}_1 + \overline{m}_2 - 1}\,\, (\cdot) \right) \right) + (2-\overline{m}_1 -\overline{m}_2) \delta_{-\infty}.
$$
\end{proposition}
\begin{proof}
For $i=1,2$, let $P_n^{(i)}(x)$ be some polynomial of degree $d_{n;i}:= \lfloor\overline{m}_i n\rfloor$ such that $\lsem P_n^{(i)}\rsem_{d_{n;i}} \to \widetilde \mu_i$ weakly, as $n\to\infty$. Then, $\lsem P_n^{(i)}\rsem_n \to \overline{m}_i \widetilde \mu_i + (1-\overline{m}_i) \delta_{-\infty}$ by our convention how to handle polynomials of degree less than $n$ in $\lsem \cdot \rsem_n$; see Definition~\ref{def:empirical_distr}.
Now, we recall the identity
$$
P_n^{(1)}(x) \boxplus_n P_{n}^{(2)}(x) =  \frac{\left((\dd /\dd x)^{n-d_{n;2}} P_n^{(1)}(x)\right) \boxplus_{d_{n;1}+d_{n;2}-n}  \left((\dd /\dd x)^{n-d_{n;1}} P_n^{(2)}(x)\right)}{n(n-1) \ldots (d_{n;1}+d_{n;2}-n+1)}.
$$
On the one-hand side, $\lsem P_n^{(1)}(x) \boxplus_n P_{n}^{(2)}(x)\rsem_n \to \mu_1 \boxplus \mu_2$. On the other hand,
$$
\lsem(\dd /\dd x)^{n-d_{n;2}} P_n^{(1)}(x)\rsem_{d_{n;1} + d_{n;2} - n} \to \widetilde \mu_1^{\boxplus \frac{\overline{m}_1}{\overline{m}_1 + \overline{m}_2 - 1}} \left(\frac{\overline{m}_1}{\overline{m}_1 + \overline{m}_2 - 1}\,\, (\cdot) \right).
$$
by Corollary~\ref{cor:repeated_diff} with $n$ replaced by $d_{n;1}$ and with $\kappa =\lim_{n\to\infty} \frac{n-d_{n,2}}{d_{n;1}}= \frac{1-\overline{m}_2}{\overline{m}_1}$, $\frac 1 {1-\kappa} = \frac{\overline{m}_1}{\overline{m}_1 + \overline{m}_2 - 1}$.
Similarly,
$$
\lsem(\dd /\dd x)^{n-d_{n;1}} P_n^{(2)}(x)\rsem_{d_{n;1} + d_{n;2} - n} \to \widetilde \mu_2^{\boxplus \frac{\overline{m}_2}{\overline{m}_1 + \overline{m}_2 - 1}} \left(\frac{\overline{m}_2}{\overline{m}_1 + \overline{m}_2 - 1}\,\, (\cdot) \right).
$$
Let $R_n:=((\dd /\dd x)^{n-d_{n;2}} P_n^{(1)}(x)) \boxplus_{d_{n;1}+d_{n;2}-n}  (\dd /\dd x)^{n-d_{n;1}} P_n^{(2)}(x)$ and observe that $\deg R_n = d_{n;1}+d_{n;2}-n$.  It follows Definition~\ref{def:empirical_distr} and Theorem~\ref{theo:finite_free_additive_conv_free_additive} that
\begin{align*}
\lsem R_n
\rsem_n
&=
\frac{d_{n;1}+d_{n;2}-n}{n} \lsem R_n
\rsem_{d_{n;1}+d_{n;2}-n} + \frac{2n- d_{n;1}- d_{n;2}}{n} \delta_{-\infty}
\\
&\toweak
(\overline{m}_1 + \overline{m}_2 - 1) \left(\mu_1^{\boxplus \frac{\overline{m}_1}{\overline{m}_1 + \overline{m}_2 - 1}} \left(\frac{\overline{m}_1}{\overline{m}_1 + \overline{m}_2 - 1}\,\, (\cdot) \right) \boxplus \widetilde \mu_2^{\boxplus \frac{\overline{m}_2}{\overline{m}_1 + \overline{m}_2 - 1}} \left(\frac{\overline{m}_2}{\overline{m}_1 + \overline{m}_2 - 1}\,\, (\cdot) \right) \right)\\
&\hspace{11cm} + (2-\overline{m}_1 -\overline{m}_2) \delta_{-\infty}.
\end{align*}
Since $\lsem P_n^{(1)} \boxplus_n P_{n}^{(2)}\rsem_n = \lsem R_n\rsem_n$, we get the claim.
\end{proof}

\section*{Acknowledgement}
We thank Andrew Campbell and Brian Hall for several insightful comments.
ZK was supported by the Deutsche Forschungsgemeinschaft (DFG, German Research Foundation) under Germany's Excellence Strategy EXC 2044 - 390685587, Mathematics M\"unster: \emph{Dynamics-Geometry-Structure}. ZK and JJ have been supported by the DFG priority program SPP 2265 \emph{Random Geometric Systems}. A part of this work
was done while AM was visiting the University of M\"{u}nster in July-September 2024 as a M\"{u}nster research fellow.

\bibliographystyle{plainnat}
\bibliography{bib_zeros_profiles_polys}
\end{document}